\documentclass[final]{siamltex}

\usepackage{epsfig,amssymb,latexsym}
\usepackage{amsfonts,psfrag,amsmath,bbm,color}
\usepackage{cancel}
\usepackage{siunitx}
\usepackage{comment}
\usepackage{mathrsfs}
\usepackage{graphicx}
\usepackage{textcomp}
\usepackage{multirow}
\usepackage{enumerate}
\usepackage{cancel}
\usepackage{algpseudocode}
\usepackage{caption}  
\usepackage{subcaption}
\usepackage{url}
\usepackage{rotating}
\usepackage{slashbox}
\usepackage{bigints}
\usepackage{xcolor}
\usepackage{ulem}
\usepackage{mathrsfs}
\usepackage[hidelinks]{hyperref}
\usepackage{placeins}

\newcommand{\bE}{{\bf E}}

\def\grad{{\nabla}}
\usepackage[ruled,vlined]{algorithm2e}
\usepackage{geometry}
\vbadness=\maxdimen

\newcommand{\footremember}[2]{%
	\footnote{#2}
	\newcounter{#1}
	\setcounter{#1}{\value{footnote}}%
}
\newcommand{\footrecall}[1]{%
	\footnotemark[\value{#1}]%
} 
\usepackage{comment}
\usepackage{tikz}
\usetikzlibrary{shapes.geometric, arrows}
\tikzstyle{startstop} = [rectangle, rounded corners, minimum width=1cm, minimum height=1cm,text centered, draw=black]
\tikzstyle{io} = [trapezium, trapezium left angle=70, trapezium right angle=110, minimum width=1cm, minimum height=1cm, text centered, draw=black, fill=blue!30]
\tikzstyle{method} = [rectangle, rounded corners, minimum width=1cm, minimum height =1cm, text centered, draw=black]
\tikzstyle{process} = [rectangle, minimum width=1cm, minimum height=1cm, text centered, draw=black]
\tikzstyle{decision} = [diamond, minimum width=0.5cm, minimum height=0.5cm, text centered, draw=black, fill=green!30]
\tikzstyle{arrow} = [thick,->,>=stealth]

\usepackage{amscd}

\graphicspath{{./figs/}}

\newcommand{\lp}{\left(}
\newcommand{\rp}{\right)}

\newcommand{\lab}{<\hspace{-1mm}}
\newcommand{\rab}{\hspace{-1mm}>}

\newtheorem{remark}{Remark}[section]

\def\PP{{{\rm l}\kern - .15em {\rm P} }}
\def\PN2{{\PP_{N}-\PP_{N-2}}}

\newcommand{\cD}{\mathcal{D}}

\newcommand{\bphi}{\boldsymbol{\varphi}}

\newcommand{\bif}{\textbf{\textit{f}}\hspace{-0.5mm}}

\newcommand{\bchi}{\pmb{\chi}}

\newcommand{\be}{\boldsymbol{e}}

\newcommand{\bg}{\textbf{\textit{g}}}

\newcommand{\bH}{\boldsymbol{H}}

\newcommand{\bL}{\boldsymbol{L}}

\newcommand{\hp}{{\hat{p}}}

\newcommand{\bR}{\boldsymbol{R}}

\newcommand{\bu}{\boldsymbol{u}}

\newcommand{\bv}{\boldsymbol{v}}

\newcommand{\bhu}{\hat{\boldsymbol{u}}}

\newcommand{\bnh}{\hat{\textbf{\textit{n}}}}

\newcommand{\bV}{\boldsymbol{V}}

\newcommand{\bw}{\boldsymbol{w}}

\newcommand{\btu}{\tilde{\boldsymbol{u}}}

\newcommand{\bx}{\boldsymbol{x}}
\newcommand{\bX}{\boldsymbol{X}}

\newcommand{\bY}{\boldsymbol{Y}}

\newcommand{\deleted}[1]{{}}

\begin{document}
	\title {An Efficient Second-Order-in-Time Penalty-Projection Ensemble Eddy Viscosity Method for Parameterized Navier-Stokes Flows}
	\author{
		Md Mahmudul Islam\footremember{uabm}{D\MakeLowercase{epartment of} M\MakeLowercase{athematics}, U\MakeLowercase{niversity of} A\MakeLowercase{labama at} B\MakeLowercase{irmingham}, B\MakeLowercase{irmingham}, AL 35294, USA}%
		\and Muhammad Mohebujjaman\footnote{S\MakeLowercase{upported by the} N\MakeLowercase{ational} S\MakeLowercase{cience} F\MakeLowercase{oundation grant} DMS-2425308; C\MakeLowercase{orrespondence: mmohebuj@uab.edu}}\hspace{1mm}\footrecall{uabm}
		\and  
		Jahrul Alam\footremember{mun}{D\MakeLowercase{epartment of} M\MakeLowercase{athematics and} S\MakeLowercase{tatistics,} M\MakeLowercase{emorial} U\MakeLowercase{niversity of} N\MakeLowercase{ewfoundland}, S\MakeLowercase{t.} J\MakeLowercase{ohn's}, A1C 5S7, N\MakeLowercase{ewfoundland}, C\MakeLowercase{anada}}
	}
	
	\maketitle

	\begin{abstract} 
		We propose a novel, robust, and second-order-accurate  parameterized penalty-projection ensemble algorithm for incompressible Navier–Stokes flow problems. The resulting linearized algorithm, based on the second-order Backward Differentiation Formula (BDF-2), is computationally efficient because it shares the same coefficient matrix across all realizations for each subproblem at every time step. To enhance robustness in convection-dominated flows, the scheme incorporates Ensemble Eddy Viscosity (EEV) regularization. In addition, it is equipped with grad-div stabilization parameter $\gamma$, which controls the splitting error; under the assumptions of the analysis, the splitting error decreases and vanishes asymptotically as $\gamma\to\infty$.
		
		We establish the stability of the proposed scheme and rigorously prove its optimal convergence by demonstrating that, as $\gamma\to\infty$, the scheme converges to an equivalent coupled formulation. We further validate the method through a series of numerical experiments designed to verify the theoretically predicted convergence rates and evaluate its performance on benchmark convection-dominated problems. The numerical results are in excellent agreement with the theoretical analysis and confirm the effectiveness of the proposed scheme.
		
	\end{abstract}

	
	
	\pagestyle{myheadings}
	\thispagestyle{plain}

	\markboth{\MakeUppercase{Efficient Second-Order Penalty-Projection EEV Method for Parameterized NSE}}{\MakeUppercase{ M. M. Islam, M. Mohebujjaman, and J. Alam}}
	
	\section{Introduction} 
	In Uncertainty Quantification (UQ), a widely used strategy involves drawing an ensemble of randomized input samples from appropriate probability distributions. For each sampled input set, a full-order simulation is performed to generate the corresponding output. Statistical estimates of a quantity of interest (QoI) are then obtained by computing either a simple or weighted average of QoI values across all simulations. To that end, we consider the following parameterized ensemble system for dimensionless, incompressible Navier-Stokes Equations (NSE) for convection-dominated flows \cite{gunzburger2019efficient, gunzburger2019secondorde}: \begin{align}
		\bu_{j,t}+\bu_j\cdot\nabla \bu_j-\nabla\cdot\left(\nu_j(\bx) \nabla \bu_j\right)+\nabla p_j &=  \bif_j(t,\bx), \hspace{2mm}\text{in}\hspace{2mm} (0,T]\times\cD, \label{gov1}\\
		\nabla\cdot \bu_j & = 0, \hspace{11.7mm}\text{in}\hspace{2.mm} (0,T]\times \cD,\label{gov2}\\
		\bu_j(t,\bx)&=\bg_j(\bx),\hspace{5mm}\text{on}\hspace{1.4mm} (0,T]\times\partial\cD,\label{gov3}\\
		\bu_j(0,\bx)& =\bu_j^0(\bx),\hspace{5mm}\text{in}\hspace{2mm}\mathcal{D}.\label{gov4}
	\end{align}
	Here, $\bu_j$ and $p_j$ represent the velocity and pressure solutions, respectively, for each $j = 1, 2, \dots, J$, corresponding to distinct kinematic viscosities $\nu_j$, body forces $\bif_j$, initial conditions $\bu_j^0$, and boundary conditions $\bg_j$. $J$ is the total number of realizations, typically large. Let $\mathcal{D} \subset \mathbb{R}^d\ (d \in \{2,3\})$ be a convex polygonal or polyhedral physical domain with boundary $\partial \mathcal{D}$. The simulation end time is denoted by $T > 0$, $\bx$ represents the spatial variable, and $t$ is the time variable. The unknown quantities are the velocity field $\bu_j(t,\bx) \in \mathbb{R}^d$ and the pressure $p_j(t,\bx) \in \mathbb{R}$, which is assumed to have zero mean. To simplify the analysis, we set $\bg_j = \textbf{0}$.
	
	It is well known that a single NSE simulation for realistic flows (which are often highly ill-conditioned) is computationally expensive due to the substantial memory requirements and computational cost. Thus, the computational complexity of solving the ensemble nonlinear Partial Differential Equation (PDE) system given in \eqref{gov1}–\eqref{gov4} is $J$ times the cost of one NSE simulation.
	
	To reduce the high computational cost, many efficient schemes for parameterized NSE flow UQ problems have been proposed recently \cite{berry2025efficient, gunzburger2019efficient,gunzburger2019secondorde,jiang2023artificial,jiang2024second, raveendran2026efficient}. These methods are elegantly designed so that all realizations share the same coefficient matrix at each time step. Thus, for small-scale problems, the result of a single direct solve at each time step can be reused for all realizations, while for large-scale problems, they can take advantage of block Krylov subspace methods. As a result, the computational complexity of these schemes can be reduced to the order of a single NSE simulation.
	
	Realistic flows often occur in convection-dominated regimes and are modeled by closed eddy viscosity systems \cite{berselli2006mathematics}. Standard schemes for \eqref{gov1}--\eqref{gov4} applied to convection-dominated, parameterized UQ flow problems often blow up due to the failure to capture energy transfer from large, resolved scales to small, unresolved scales, as well as backscatter of energy \cite{berry2025efficient,berselli2006mathematics,mohebujjaman2024efficient,raveendran2026efficient}.  As a result, convection-dominated parameterized UQ problems are commonly simulated using closed eddy-viscosity schemes \cite{jiang2015higher, jiang2015numerical, mohebujjaman2024efficient, mohebujjaman2017efficient, mohebujjaman2022efficient, raveendran2026efficient} in which the Reynolds Stress Tensor \cite{berselli2006mathematics} is estimated in terms of Ensemble Eddy Viscosity (EEV) \cite{berry2025efficient}.
	
	Since the viscosity parameter cannot be measured exactly, parameterized numerical schemes are used to account for its uncertainty. A BDF family of efficient parameterized EEV coupled schemes for \eqref{gov1}-\eqref{gov4} was investigated in \cite{berry2025efficient}. While the approach in \cite{berry2025efficient} greatly lowers the computational cost, one of the fundamental hurdles in NSE simulations is the coupling between velocity and pressure variables, which leads to a $2\times 2$ block coefficient matrix in a saddle point system.
	
	Chorin \cite{chorin1968numerical} and Temam \cite{temam1969approximation} pioneered projection schemes that address the coupling issue in the NSE using a discrete Hodge decomposition at every time-step. This is a two-step process: In Step 1, the pressure is treated explicitly, and a non-divergence-free velocity is computed by solving a system involving only the velocity. In Step 2, the pressure is corrected by projecting the computed velocity onto a divergence-free space using a nonphysical pressure boundary condition, yielding a second velocity that violates the boundary condition. The no-slip boundary condition is enforced in Step 1, while Step 2 applies the no-penetration boundary condition. The cost of Step 2 can be reduced to that of solving a Poisson equation. For large-scale problems, the combined cost of solving the two-step projection methods is significantly lower than that of the coupled method. A first-order-accurate scalar auxiliary variable rotational pressure-correction method for ensemble flow was studied in \cite{jiang2024efficient}. 
	
	The solution of the projection method loses accuracy due to the splitting error. To improve accuracy, a grad-div term is introduced into the momentum equation as a penalty, leading to what is known as the Stabilized Penalty-Projection (SPP) method. In the continuous case, the grad-div term vanishes exactly, but in the discrete case, it does not. It has been proven that a large penalty parameter in the PP method eliminates the splitting error in the deterministic NSE \cite{linke2017connection}. Efficient parameterized first-order BDF-1-EEV-SPP time-stepping schemes for ensemble NSE and Magnetohydrodynamic (MHD) flow problems were proposed, analyzed, and tested in \cite{raveendran2026efficient} and \cite{mohebujjaman2024efficient}, respectively, for convection-dominated flows.
	
	To improve the temporal accuracy of the BDF-1-EEV-SPP ensemble NSE scheme \cite{raveendran2026efficient}, in this paper, we propose an efficient  second-order BDF-2-based EEV-SPP (BDF-2-EEV-SPP)  ensemble time-stepping scheme. We consider a uniform discretization in time with time-step size $\Delta t$ throughout the paper. Let $t_n = n \Delta t$ for $n = 0, 1, \dots$. Then compute as follows: For $j = 1, \dots, J,$ \\Step 1: Seek $\bhu_{j,h}^{n+1}$:
	\begin{align}
		&\frac{3}{2\Delta t} \bhu^{n+1}_{j,h}   + <{\bhu}_h>^n\cdot \nabla\bhu^{n+1}_{j,h} -\nabla \cdot \left(\bar{\nu} \nabla \bhu^{n+1}_{j,h}\right) - \gamma\nabla \left(\nabla \cdot \bhu^{n+1}_{j,h}\right) - \nabla \cdot \left( \nu_T(\hat{u}_{h}^{'},t_{n})\nabla \bhu_{j,h}^{n+1}\right) \nonumber \\ & = \bif_{j}(t_{n+1}) + \frac{3}{2\Delta t}\left(4\btu^n_{j,h} - \btu^{n-1}_{j,h}\right) - \bhu^{'n}_{j,h} \cdot \nabla\left(2\bhu^{n}_{j,h} - \bhu^{n-1}_{j,h}\right) + \nabla \cdot\left( \nu^{'}_j (2\bhu^n_{j,h} - \bhu^{n-1}_{j,h})\right),\label{prop-1} \\  &\bhu_{j,h}^{n+1}\big|_{\partial \cD} = 0. \label{prop-1-cond}
	\end{align}
	\\Step 2: Seek $\btu^{n+1}_{j,h} \text{and }\hat{p}^{n+1}_{j,h}$:
	\begin{align}
		\frac{3}{2\Delta t}\hspace{0.2mm} \btu^{n+1}_{j,h} - \nabla\hat{p}^{n+1}_{j,h} &= \frac{3}{2\Delta t}\hspace{0.2mm}\bhu^{n+1}_{j,h},
		\label{prop-2}
		\\
		\nabla \cdot \btu^{n+1}_{j,h} &= 0, \label{prop-2-cond1} \\
		\btu_{j,h}^{n+1}\cdot \hat{\boldsymbol{n}}\big|_{\partial \cD} &= 0. \label{prop-2-cond2}
	\end{align}
	
	Here, $\bhu^{n}_{j,h}$ and $\hp^{n}_{j,h}$ are approximations to $\bu_j(\cdot, t_n)$ and $p_j(\cdot, t_n)$, respectively, and $\btu^{n}_{j,h}$ is the projection of $\bhu^{n}_{j,h}$ onto the divergence-free space. The grad-div stabilization coefficient is denoted by $\gamma > 0$, and the outward unit normal vector is denoted by  $\hat{\boldsymbol{n}}$. We define the ensemble means 
	\begin{align*}
		<{\bhu}_{h}>^n: = \frac{1}{J} \sum^J_{j=1} (2\bhu^{n}_{j,h} - \bhu^{n-1}_{j,h}) ,\text{ and } \bar{\nu}:= \frac{1}{J}\sum_{j=1}^{J}\nu_{j}
	\end{align*}
	and fluctuations
	\begin{align*}
		\bhu^{'n}_{j,h} := 2\bhu^n_{j,h} - \bhu^{n-1}_{j,h} - <{\bhu}_{h}>^n, \text{ and}\hspace{0.15cm} \nu_{j}^{'} := \nu_{j}-\bar{\nu}.
	\end{align*}
	Using the concept of mixing length, the Ensemble Eddy Viscosity (EEV) is defined as \cite{jiang2015higher}:
	\begin{align} \nu_T(\hat{u}^{'}_{h}, t_{n}):=\mu\Delta t(\mathcal{L}^n(\hat{u}^{'}_{h}))^2,\hspace{0.5mm}\text{ where}\hspace{1mm} \hat{u}^{'}_{h}=(\bhu_{1,h}^{'n}|\bhu_{2,h}^{'n}|\cdots|\bhu_{J,h}^{'n})\;\text{and }\mathcal{L}^n(\hat{u}^{'}_{h}):=\sqrt{\sum_{j=1}^J|\bhu_{j,h}^{'n}|^2}.\label{mixing-length}
	\end{align}
	Here, $\mathcal{L}^n$ is a mixing length, $\mu$ is a calibration parameter, and $|\cdot|$ represents the Euclidean norm. 
	
	At each time step, both subproblems share the same coefficient matrix across all realizations; thus, it is an efficient EEV-SPP scheme. In Step 1, the system to be solved is significantly smaller than that in a coupled saddle-point problem. In contrast, Step 2 involves solving a  $2 \times 2$ block linear system similar to that in a coupled saddle-point system; however, it is much cheaper to solve because the resulting block matrix does not contain contributions from the nonlinear and stiffness terms. Moreover, its cost can be reduced to that of solving a Poisson equation at each time-step.
	
	\textbf{Significance of the work:} In this paper, we propose, analyze, and test an efficient, linearized, robust, and fully discrete BDF-2-EEV-SPP method that achieves second-order temporal and optimal spatial accuracy for parameterized ensemble NSE. We provide a rigorous analysis of stability and convergence and show that, as $\gamma \to \infty$, the solution of the proposed BDF-2-EEV-SPP scheme converges to that of an equivalent second-order temporally accurate coupled scheme. We verify the predicted convergence rates using a manufactured solution. Finally, we implement both the proposed BDF-2-EEV-SPP scheme and an equivalent coupled scheme on the benchmark problems of channel flow over a step and flow past a circular cylinder, and find excellent agreement between the two schemes. The BDF-2-EEV-SPP scheme is linearized and therefore significantly reduces the computational cost compared to nonlinear methods. It also shares the coefficient matrix across all realizations, thereby saving system matrix assembly time, solver time, and substantial memory. An appropriately large penalty parameter allows the use of non–pointwise divergence-free elements, such as the Taylor–Hood (TH) element \cite{wieners2003taylor}, which further reduces computational cost compared to pointwise divergence-free elements, such as the Scott–Vogelius (SV) element \cite{scott1985norm}. Moreover, the presence of the EEV term provides robustness on unresolved meshes.
	
	To the authors' knowledge, the proposed efficient BDF-2-EEV-SPP scheme for ensemble NSE has not been investigated to date. We expect that the proposed BDF-2-EEV-SPP scheme will be an enabling tool for large-scale, realistic, convection-dominated flow problems. 
	
	The remaining sections of this paper are organized as follows. In Section \ref{notation-prelims}, we introduce the necessary notation and mathematical preliminaries to facilitate the finite element analysis. Section \ref{fully-discrete-scheme} presents an efficient second-order BDF-2 based  EEV (BDF-2-EEV-Coupled) time-stepping algorithm \cite{berry2025efficient} for ensemble NSE, along with its stability and convergence theorems. In Section \ref{proj-section}, we propose a novel and more efficient BDF-2-EEV-SPP algorithm in the fully discrete setting and rigorously prove its stability and convergence. It is shown that, as $\gamma \to \infty$, the BDF-2-EEV-SPP scheme converges to the BDF-2-EEV-Coupled scheme. Finally, Section \ref{numerical-experiment} presents several numerical experiments that validate the theory.

	\section{Notation and Preliminaries}\label{notation-prelims}
	The usual $L^2(\cD)$ norm and inner product are denoted by $\|\cdot\|$ and $(\cdot,\cdot)$, respectively. Similarly, the $L^p(\cD)$ norms and the Sobolev $W_p^k(\cD)$ norms are $\|\cdot\|_{L^p}$ and $\|\cdot\|_{W_p^k}$, respectively, for $k\in\mathbb{N}\text{ and }\hspace{1mm}1\le p\le \infty$. The Sobolev space $W_2^k(\cD)$ is represented by $H^k(\cD)$ with norm $\|\cdot\|_k$. The vector-valued spaces are $$\bL^p(\cD)=(L^p(\cD))^d, \hspace{1mm}\text{and}\hspace{1mm}\bH^k(\cD)=(H^k(\cD))^d.$$
	For a normed function space $\bX$ on $\cD$, $L^p(0,T;\bX)$ is the space of all functions defined on $(0,T]\times\cD$ for which the following norm\begin{align*} \|\bu\|_{L^p(0,T;\bX)}=\lp\int_0^T\|\bu\|_{\bX}^pdt\rp^\frac{1}{p},\hspace{2mm}p\in[1,\infty)
	\end{align*}
	is finite. For $p=\infty$, the usual modification is used in the definition of this space. The natural function spaces for our problem are
	\begin{align*}
		\bX:&=\bH_0^1(\cD)=\{\bv\in \bL^2(\cD) :\nabla \bv\in L^2(\cD)^{d\times d}, \bv=0 \hspace{2mm} \mbox{on}\hspace{2mm}   \partial \cD\},\\
		\bY:&=\{\bv\in \bH^1(\cD),\bv\cdot\bnh\big|_{\partial\cD}=0\},\\
		Q:&=L_0^2(\cD)=\{ q\in L^2(\cD): \int_\cD q\hspace{1mm}d\bx=0\}.
	\end{align*}
	
	Recall that the Poincar\'e inequality holds in $\bX$: There exists a constant $C$, depending only on $\cD$, such that, for all $\bphi\in \bX$,
	\[
	\| \bphi \| \le C \| \nabla \bphi \|.
	\]

	We define the skew-symmetric trilinear form $b^*:\bX\times \bX\times \bX\rightarrow \mathbb{R}$ by
	\[
	b^*(\bu,\bv,\bw):=\frac12(\bu\cdot\nabla \bv,\bw)-\frac12(\bu\cdot\nabla \bw,\bv). 
	\]
	
	By the divergence theorem \cite{jiang2015higher}, it can be shown that \begin{align}
		b^*(\bu,\bv,\bw)=(\bu\cdot\nabla \bv,\bw)+\frac12(\nabla\cdot\bu,\bv\cdot\bw).\label{trilinear-identitiy}
	\end{align}
	Recall from \cite{layton2008introduction, lee2011error, linke2017connection} that, for any $\bu,\bv,\bw\in 
	\bX$
	\begin{align}
		b^*(\bu,\bv,\bw)&\leq C(\cD)\|\bu\|^{\frac12}\|\nabla \bu\|^{\frac12}\|\nabla \bv\|\|\nabla \bw\|,\label{nonlinearbound1}\\
		b^*(\bu,\bv,\bw)&\leq C(\cD)\|\nabla \bu\|\|\nabla \bv\|\|\nabla \bw\|,\label{nonlinearbound}
	\end{align}	
	and additionally, if $\bv\in \bL^\infty(\cD)$, and $\nabla\bv\in\bL^3(\cD)$, then 
	\begin{align}
		b^*(\bu,\bv,\bw)\leq C(\cD)\|\bu\|\left(\|\nabla\bv\|_{L^3}+\|\bv\|_{L^\infty}\right)\|\nabla\bw\|. \label{nonlinearbound3}
	\end{align}
	
	The following basic inequalities will be used
	\begin{align}
		\|\bu\cdot\nabla\bv\|&\le\||\bu|\nabla\bv\|,\label{basic-ineq}\\
		\|\nabla\cdot\bu\|_{L^\infty}&\le C\|\nabla\bu\|_{L^\infty}.\label{basic-ineq-infinity}
	\end{align}

	The space of divergence-free functions is given by
	$$\bV:=\{\bv\in\bX:(\nabla\cdot\bv,q)=0,\forall q\in Q\}.$$
	
	Let $\mathcal{T}_h(\cD)$ be a shape-regular and quasi-uniform family of conforming meshes of $\cD$, consisting of triangles or quadrilaterals for $d=2$, and tetrahedra or hexahedra for $d=3$.  The subscript $h$ is defined as $$h=\max_{\forall K\in\mathcal{T}_h(\cD)} \text{diameter} (K).$$
	
	The conforming finite element spaces are denoted by $\bX_h\subset \bX$ and  $Q_h\subset Q$. We assume that $(\bX_h,Q_h)$ satisfies the usual discrete inf-sup condition
	\begin{eqnarray}
		\inf_{q_h\in Q_h}\sup_{\bv_h\in \bX_h}\frac{(q_h,\grad\cdot \bv_h)}{\|q_h\|\|\grad \bv_h\|}\geq\beta>0,\label{infsup}
	\end{eqnarray}
	where $\beta$ is independent of $h$. 
	
	The discretely divergence-free subspace of $\bX_h$ is $$\bV_h:=\{\bv_h\in\bX_h:\left(\nabla\cdot\bv_h,q_h\right)=0,\forall q_h\in Q_h\}.$$
	
	The following lemma for the discrete Gr\"onwall inequality was given in \cite{HR90}. 
	\begin{lemma}\label{dgl}
		Let $\Delta t$, $\mathcal{E}$, $a_n$, $b_n$, $c_n$, $d_n$ be nonnegative numbers for $n=1,\cdots, M$ such that
		$$a_M+\Delta t \sum_{n=1}^Mb_n\leq \Delta t\sum_{n=1}^{M-1}{d_na_n}+\Delta 
		t\sum_{n=1}^Mc_n+\mathcal{E}\hspace{3mm}\mbox{for}\hspace{2mm}M\in\mathbb{N},$$
		then for all $\Delta t> 0,$
		$$a_M+\Delta t\sum_{n=1}^Mb_n\leq \exp\left(\Delta t\sum_{n=1}^{M-1}d_n\right)\lp\Delta 
		t\sum_{n=1}^Mc_n+\mathcal{E}\rp\hspace{2mm}\mbox{for}\hspace{2mm}M\in\mathbb{N}.$$
	\end{lemma}
	We assume $\nu_j(\bx)\in L^\infty(\cD)$, and $\nu_j(\bx)\ge\nu_{j,\min}>0$, where $\nu_{j,\min}=\min\limits_{\bx\in\cD}\nu_j(\bx)$, for $j=1,2,\cdots,J$. Also, define $\Bar{\nu}_{\min}:=\min\limits_{\bx\in\cD}\Bar{\nu}(\bx)$.

	\section{Efficient and Second-order-Accurate  BDF-2-EEV-Coupled Parameterized Scheme}\label{fully-discrete-scheme}
	In this section, we present Algorithm \ref{coupled-alg-com}, a fully discrete, efficient, optimally accurate, linearly extrapolated, velocity–pressure-coupled BDF-2-EEV-Coupled method. The method is an EEV- and grad-div-regularized finite element time-stepping scheme for the parameterized NSE. Algorithm \ref{coupled-alg-com} is a variation of the algorithm given in (2.7) in \cite{gunzburger2019secondorde} (without the EEV and grad-div terms). 
	
	\begin{algorithm}[H]\label{coupled-alg-com}
		\caption{Efficient second-order BDF-2-EEV-Coupled scheme for \eqref{gov1}-\eqref{gov4}} Input: $\Delta t>0$, $T>0$, $\bif_{j}\in$ $ L^\infty\left( 0,T;\bH^{-1}(\cD)\right)$, and initial conditions $\bu_j^0,\bu_j^1\in\bL^2(\cD)$ with no slip boundary conditions, where $\bu_{j,h}^0=Proj^{L^2}_{\bV_h}(\bu_j^0)$ and $\bu_{j,h}^1=Proj^{L^2}_{\bV_h}(\bu_j^1)$  for $j=1,2,\cdots\hspace{-0.35mm},J$.\\
		Set $M=T/\Delta t$ and compute: Find $\bu_{j,h}^{n+1}\in \bX_h, \text{ and }\; p_{j,h}^{n+1}\in Q_h$ satisfying, for all $\bchi_h\in \bX_h \text{ and }\;q_{h}\in Q_h$:
		\begin{align}
			&\frac{1}{2\Delta t}\left(3\bu_{j,h}^{n+1}-4\bu_{j,h}^{n}+\bu_{j,h}^{n-1}, \bchi_{h}\right)+b^*\left(\hspace{-1mm}<\bu_h>^n, \bu_{j,h}^{n+1},\bchi_h\right)+\left(\Bar{\nu}\nabla \bu_{j,h}^{n+1},\nabla \bchi_{h}\right)\nonumber\\&+\left(\gamma\nabla\cdot\bu_{j,h}^{n+1}-p_{j,h}^{n+1},\nabla\cdot\bchi_h\right)+\left(\nu_T(u_h^{'},t_n)\nabla \bu_{j,h}^{n+1},\nabla\bchi_h\right)= \left(\bif_{j}(t_{n+1}), \bchi_h\right)\nonumber\\&-b^*(\bu_{j,h}^{'n}, 2\bu_{j,h}^{n}-\bu_{j,h}^{n-1},\bchi_{h})-\left( \nu_j^{'}\nabla(2 \bu_{j,h}^{n}-\bu_{j,h}^{n-1}),\nabla\bchi_{h}\right),
			\label{couple-eqn-1-new}\\&\left(\nabla\cdot\bu_{j,h}^{n+1},q_{h}\right)=0,\label{couple-incompressibility-new}
		\end{align}
	\end{algorithm}
	\noindent where the approximations to $\bu_j(\cdot,t_n)$ and $p_j(\cdot,t_n)$ are denoted by $\bu_{j,h}^n$ and $p_{j,h}^{n}$, respectively. The ensemble average and the corresponding fluctuations are defined as follows:\vspace{-2ex}\begin{align*}
		<\hspace{-1mm}\bu_h\hspace{-1mm}>^n:=\frac{1}{J}\sum_{j=1}^J\left(2\bu_{j,h}^n-\bu_{j,h}^{n-1}\right),\;\bu_{j,h}^{'n}:=2\bu_{j,h}^n-\bu_{j,h}^{n-1}-<\bu_h>^n,\; \text{and } u^{'}_{h}=(\bu_{1,h}^{'n}|\bu_{2,h}^{'n}|\cdots|\bu_{J,h}^{'n}).
	\end{align*} Define $
	\alpha_j := \Bar{\nu}_{\min} - 3 \|\nu^{'}_j\|_\infty > 0\;\text{and}\;\alpha_{\min}:=\min\limits_{1\le j\le J}\alpha_j$. That is, $\frac{\|\nu^{'}_j\|_\infty}{\Bar{\nu}_{\min}}<\frac13.$
	
	A sufficiently large grad-div stabilization parameter $\gamma$ strongly penalizes the discrete divergence of Taylor–Hood velocity solutions and, under appropriate finite-element and mesh assumptions, such solutions are known to converge to corresponding pointwise divergence-free solutions as $\gamma\to\infty$ \cite{jenkins2014parameter,linke2011convergence}.
	
	This approach avoids the need for SV elements, which require barycentrically refined triangular or tetrahedral meshes and involve a larger number of Degrees of Freedom (DoFs). Meanwhile, the presence of the EEV term provides long-time stability for under-resolved meshes in convection-dominated flows. Algorithm \ref{coupled-alg-com} is efficient because, at each time step, it shares a common coefficient matrix across all realizations, which significantly reduces both computational time and memory usage.
	
	Under the time-step restriction \begin{align}
		\Delta t\le\min_{\substack{1\le j\le J \\ 1\le n\le M}}\frac{C\alpha_j}{\|\nabla\cdot\bu_{j,h}^{'n}\|^2_{L^\infty}},\label{time-step-size}
	\end{align} the following stability \begin{align}
		\|\bu_{j,h}^M\|^2+\|2\bu_{j,h}^M-\bu_{j,h}^{M-1}\|^2+2\alpha_{j} \Delta t \sum_{n=2}^{M}\|\nabla \bu_{j,h}^{n}\|^2 +4\gamma\Delta t \sum_{n=2}^{M}\|\nabla \cdot \bu_{j,h}^{n}\|^2 \leq C(data),\label{stability-bdf2-statement}
	\end{align} and error estimate for $(\mathbb{P}_k^d,\mathbb{P}_{k-1})$ or $(\mathbb{Q}_k^d,\mathbb{Q}_{k-1})$ finite element pair
	 \begin{align}
		\sum_{j=1}^J\|\bu_{j}(t_{M})-\bu_{j,h}^M\|^2+2\alpha_{\min}\Delta t\sum_{n=2}^{M}\sum_{j=1}^J\|\nabla\left(\bu_{j}(t_{n})-\bu_{j,h}^n\right)\|^2\le \frac{C}{\alpha_{\min}}(h^{2k}+\Delta t^4).\label{error-eqn-coupled}
	\end{align}
	 were proved in Theorems 3.4 and 3.5 in \cite{berry2025efficient}, respectively, where $(\bu_j, p_j)$ denotes the exact solution of \eqref{gov1}-\eqref{gov4}.
	\begin{lemma}\label{lemma-L3-infty}
		If $\bu_j\in L^\infty(0,T;\bH^{k+1}(\cD)^d)$, $k\ge 2$ then for TH element and $O(h^{2k-1})\le\Delta t\le O(h^{\frac13})$, $\exists$ $C_*\in\mathbb{R}^+$ (which does not depend on $h$, $\Delta t$, and $\gamma$) such that
		\begin{align*}
			\Delta t\sum_{n=2}^M\|p_{j,h}^n\|^2+\max_{1\le n\le M}\Big(\|\nabla \bu_{j,h}^n\|_{L^3}+\|\bu_{j,h}^n\|_{L^\infty}\Big)&\le C_*,\hspace{2mm}\text{for all}\hspace{2mm}j=1,2,\cdots,J.
		\end{align*}
	\end{lemma}
	\begin{proof}
		The proof is given in Appendix \ref{appendix-C}.
	\end{proof}
	
	\section{Efficient and Second-order-Accurate BDF-2-EEV-SPP Scheme} \label{proj-section}

	In this section, we propose, analyze, and test a more efficient and robust second-order penalty-projection-based BDF-2-EEV-SPP algorithm for parameterized NSE flow problems. To this end, we define an additional space $\bY_h\subset \bY$. The fully discrete, grad-div-regularized, two-step  BDF-2-EEV-SPP algorithm for the solution of \eqref{gov1}--\eqref{gov4} is given in Algorithm \ref{NSE-FEM}.
	
	\begin{algorithm}[H]\label{NSE-FEM}
		\caption{Efficient second-order BDF-2-EEV-SPP scheme for \eqref{gov1}-\eqref{gov4}}
		Input: $\Delta t>0, T>0$,  $\bif_{j}$ $\in$ $L^\infty$ $(0, T; \bH^{-1}(\cD))$, initial conditions $\bu^0_j$, $\bu^1_j$ $\in\bL^2(\cD)$ with no slip boundary conditions, and $\bhu_{j,h}^0=$  $\bu^0_{j,h}$, $\bhu_{j,h}^1=\bu^1_{j,h}$  for all  $j=1,2,\cdots,J$. Set $M=T/\Delta t$. For $n=1,2,\cdots,M-1$, compute:\\
		Step 1: Find ${\bhu}^{n+1}_{j,h} \in \bX_h$ satisfying, for all 
		$\bchi_h \in \bX_h$,
		\begin{align}
			&\frac{1}{2\Delta t} \left({3}\bhu^{n+1}_{j,h} - 4\btu^n_{j,h} + \btu^{n-1}_{j,h}, \bchi_h\right) + b^* \left(<{\bhu}_h>^n, \bhu^{n+1}_{j,h}, \bchi_h\right) + \left(\bar{\nu} \nabla \bhu^{n+1}_{j,h}, \nabla \bchi_h\right) \nonumber \\ &+ \gamma \left(\nabla \cdot \bhu^{n+1}_{j,h}, \nabla \cdot\bchi_h\right) + \left(\nu_T(\hat{u}_{h}^{'},t_n)\nabla \bhu_{j,h}^{n+1},\nabla\bchi_h\right) = \left(\bif_{j}(t_{n+1}), \bchi_h\right) \nonumber \\ &- b^* \left(\bhu^{'n}_{j,h} , 2\bhu^{n}_{j,h} - \bhu^{n-1}_{j,h} , \bchi_h\right) - \left(\nu^{'}_j \nabla (2\bhu^n_{j,h} - \bhu^{n-1}_{j,h}) , \nabla \bchi_h\right). \label{NSE-step-1}
		\end{align}
		Step 2: Find $\left(\btu^{n+1}_{j,h},\hat{p}^{n+1}_{j,h} \right) \in \bY_h \times Q_h$ satisfying, for all $\left(\bv_h, q_h \right) \in \bY_h \times Q_h$,
		\begin{align}
			\frac{3}{2\Delta t} \left(\btu^{n+1}_{j,h} - \bhu^{n+1}_{j,h}, \bv_h \right) - \left(\hat{p}^{n+1}_{j,h}, \nabla \cdot \bv_h \right) &= 0, \label{NSE-step-2-1} \\
			\left(\nabla \cdot \btu^{n+1}_{j,h}, q_h \right) &= 0. \label{NSE-step-2-2}
		\end{align}
	\end{algorithm}
	In Step 1, we solve a $1 \times 1$ block system of the form
	\begin{align}
		\mathbb{\hat{A}}\left(\hat{\textbf{U}}_1 \big\rvert\hat{\textbf{U}}_2 \big\rvert \cdots \big\rvert\hat{\textbf{U}}_J\right)=\left(\hat{\textbf{F}}_1\big\rvert \hat{\textbf{F}}_2\big\rvert \cdots \big\rvert \hat{\textbf{F}}_J \right),\label{sparse-system-block-pr1}
	\end{align}
	where $\mathbb{\hat{A}}$ is a coefficient matrix independent of the index $j$, and $\hat{\textbf{U}}_j$ is the nodal vector for $\bhu_{j,h}^{n+1}$. The size of $\mathbb{\hat{A}}$ is much smaller than that of the matrix in Algorithm \ref{coupled-alg-com}. 
	
	In Step 2, we solve a $2 \times 2$ block system of the form\begin{align}
		\begin{pmatrix}\mathbb{\tilde{A}} & \mathbb{\tilde{B}}^T\\\mathbb{\tilde{B}} & \mathcal{O}\end{pmatrix}\bigg(\begin{matrix}
			\tilde{\textbf{U}}_1\\
			\hat{\textbf{P}}_1
		\end{matrix}
		\bigg\rvert\begin{matrix}
			\tilde{\textbf{U}}_2\\
			\hat{\textbf{P}}_2
		\end{matrix}\bigg\rvert\begin{matrix}
			\cdots\\
			\cdots
		\end{matrix}\bigg\rvert\begin{matrix}
			\tilde{\textbf{U}}_J\\
			\hat{\textbf{P}}_J
		\end{matrix}\bigg)=\bigg(\begin{matrix}
			\tilde{\textbf{F}}_1\\\textbf{G}_1
		\end{matrix}\bigg\rvert\begin{matrix}
			\tilde{\textbf{F}}_2\\\textbf{G}_2
		\end{matrix}\bigg\rvert\begin{matrix}
			\cdots\\
			\cdots
		\end{matrix}\bigg\rvert\begin{matrix}
			\tilde{\textbf{F}}_J\\\textbf{G}_J
		\end{matrix}\bigg),\label{sparse-system-block-pr2}
	\end{align}
	where $\mathbb{\tilde{A}}$ is the mass matrix, $\mathbb{\tilde{B}}$ is the gradient operator, and $\mathbb{\tilde{B}}^T$ is the adjoint of $\mathbb{\tilde{B}}$. $\tilde{\textbf{U}}_j$ and $\hat{\textbf{P}}_j$ are the nodal vectors associated with $\btu_{j,h}^{n+1}$ and $\hat{p}_{j,h}^{n+1}$, respectively. Since $\mathbb{\tilde{A}}$ is symmetric positive definite, the system is significantly cheaper to solve than the one arising in Algorithm \ref{coupled-alg-com}. Moreover, by taking divergence and the dot product with the normal vector $\bnh$ on both sides of \eqref{prop-2}, and then using the incompressibility constraint and boundary condition, respectively, \eqref{sparse-system-block-pr2} can be written as	\begin{align}
		\mathbb{\hat{S}}\left(\hat{\textbf{P}}_1 \big\rvert\hat{\textbf{P}}_2 \big\rvert \cdots \big\rvert\hat{\textbf{P}}_J\right)=\left(\hat{\textbf{G}}_1\big\rvert \hat{\textbf{G}}_2\big\rvert \cdots \big\rvert \hat{\textbf{G}}_J \right).\label{poisson}
	\end{align}Here, $\mathbb{\hat{S}}$
	denotes the stiffness matrix, while 
	$\hat{\textbf{F}}_j$, $\tilde{\textbf{F}}_j$, and $\hat{\textbf{G}}_j$ for
	$j=1,2,\cdots,J$ represent the corresponding right-hand-side vectors in equations \eqref{sparse-system-block-pr1}, \eqref{sparse-system-block-pr2}, and \eqref{poisson}, respectively. The matrix $\mathbb{\hat{S}}$ is independent of the index $j$ and time; thus, it needs to be constructed only once and can then be reused across all time steps and all realizations. A block Poisson solver can be used to solve \eqref{poisson}. Note that the use of the system in \eqref{poisson} leads to an explicit pressure term in \eqref{sparse-system-block-pr2}.
	
	As a result, in practice, the combined cost of Algorithm \ref{NSE-FEM} will be significantly less than the cost of Algorithm \ref{coupled-alg-com}.  Additionally, as $\gamma$ increases, the splitting error decreases and vanishes asymptotically as $\gamma\to\infty$, thereby allowing the projection method to recover the accuracy of the coupled scheme. 
	
	\begin{remark}
		An appropriately large grad-div stabilization parameter $\gamma$ in the coupled method of Algorithm \ref{coupled-alg-com} reduces the divergence error, and a strategy for selecting $\gamma$ is discussed in \cite{jenkins2014parameter}. In Algorithm \ref{NSE-FEM}, a sufficiently large $\gamma$ reduces both the divergence error and the splitting error. 
	\end{remark}
	\subsection{Stability Analysis}\label{stability-analysis}
	We now prove stability and well-posedness for Algorithm \ref{NSE-FEM}.
	
	\begin{lemma}\label{uniform-boundedness-lemma-proof} For fixed $h$ and $\Delta t>0$, there exists a constant $C_*$, independent of $h$ and $\Delta t$, such that the solution of Algorithm \ref{NSE-FEM} satisfies
		\begin{align}
			\lim_{\gamma\to\infty}\max_{0\le n\le M}\|\bhu_{j,h}^n\|_{L^\infty}\le C_*,\hspace{2mm}\text{for all}\hspace{2mm}j=1,2,\cdots,J.
		\end{align}
	\end{lemma}
	
	\begin{proof}
		See Appendix \ref{appendix}.
	\end{proof}
	
	Define $$D_{\infty}:=\max_{\substack{1\le j\le J \\ 1\le n\le M}}\|\nabla \cdot \bhu^{'n}_{j,h}\|_{L^\infty}, \text{ and }\zeta:=\frac{\alpha_{\min} }{2}-\frac{C}{\mu \Delta t} - \frac{C}{\alpha_{\min}} D_{\infty}^2.$$
	Note that $\zeta> 0$ implies \begin{align*}
		\mu&>\frac{C\alpha_{\min}}{\Delta t\left(\alpha_{\min}^2-CD_{\infty}^2\right)}.
	\end{align*}
	
	\begin{theorem}\label{stability-penalty-theorem}
		Assume $\bif_{\hspace{0.4mm}{j}} \in L^2 \left(0,T,\bH^{-1}(\cD)\right), \bhu^0_{j,h}, \bhu^1_{j,h} \in \bH^1(\cD)$, $\zeta> 0$, $\alpha_{\min}>CD_{\infty}$. Choose $$\mu>\frac{C\alpha_{\min}}{\Delta t\left(\alpha_{\min}^2-CD_{\infty}^2\right)}.$$ Then, for all $\Delta t>0$, the solutions of Algorithm \ref{NSE-FEM} are stable,
		\begin{align}
			& \|\bhu^{M}_{j,h}\|^2+ \zeta\Delta t\sum_{n=1}^{M}\|\nabla\bhu_{j,h}^{n}\|^2+ \Delta t \gamma \sum^{M}_{n=2} \| \nabla \cdot \bhu^{n}_{j,h} \|^2 \nonumber\\&\leq   6\|\bhu^{1}_{j,h}\|^2 + 2\| \bhu^{0}_{j,h} \|^2 + 2 \bar{\nu}_{\min} \Delta t\left( \| \nabla \bhu^{0}_{j,h}\|^2+\| \nabla \bhu^{1}_{j,h}\|^2\right)+\frac{4 \Delta t}{\alpha_{\min}} \sum^{M}_{n=2} \| \bif_{\hspace{0.4mm}{j}}(t_{n}) \|^2_{-1}. \label{stability-penalty}
		\end{align}
	\end{theorem}
	\begin{proof}
		The proof follows by letting $ \bchi_h= \bhu^{n+1}_{j,h}$ in \eqref{NSE-step-1},
		\begin{align}
			&\frac{1}{2\Delta t} \left({3}\bhu^{n+1}_{j,h} - 4\btu^n_{j,h} + \btu^{n-1}_{j,h}, \bhu^{n+1}_{j,h}\right) + b^* \left(<{\bhu}_h>^n, \bhu^{n+1}_{j,h}, \bhu^{n+1}_{j,h}\right) + \left(\bar{\nu}\nabla \bhu^{n+1}_{j,h}, \nabla \bhu^{n+1}_{j,h}\right) \notag \\ &+ \gamma \left(\nabla \cdot \bhu^{n+1}_{j,h}, \nabla \cdot\bhu^{n+1}_{j,h}\right) + \mu \Delta t \left((\mathcal{L}^n(\hat{u}^{'}_{h}))^2 \nabla\bhu^{n+1}_{j,h}, \nabla\bhu^{n+1}_{j,h} \right) = \left(\bif_{j}(t_{n+1}), \bhu^{n+1}_{j,h}\right) \notag \\ &- b^* \left(\bhu^{'n}_{j,h} , 2\bhu^{n}_{j,h} - \bhu^{n-1}_{j,h} , \bhu^{n+1}_{j,h}\right) - \left(\nu^{'}_j \nabla (2\bhu^n_{j,h} - \bhu^{n-1}_{j,h}) , \nabla \bhu^{n+1}_{j,h}\right). \notag
		\end{align}
		Using the Cauchy-Schwarz and Young's inequalities, as well as the following algebraic identity
		\begin{align}
			\frac{1}{2}(3a - 4b + c)a = \frac{1}{4}\left[a^2 + \left(2a - b\right)^2\right] - \frac{1}{4}\left[b^2 + \left(2b - c\right)^2\right] + \frac{1}{4}\left(a - 2b + c\right)^2 , \label{BDF-2-identity}
		\end{align}
		we obtain
		\begin{align}
			&\frac{1}{4 \Delta t} \left( \|\bhu^{n+1}_{j,h}\|^2 + \|2\bhu^{n+1}_{j,h} - \btu^{n}_{j,h}\|^2 - \|\btu^{n}_{j,h}\|^2 - \|2\btu^{n}_{j,h} - \btu^{n-1}_{j,h}\|^2 + \|\bhu^{n+1}_{j,h} - 2\btu^{n}_{j,h} + \btu^{n-1}_{j,h}\|^2\right) \notag \\ &+ \|\bar{\nu}^{\frac{1}{2}} \nabla \bhu^{n+1}_{j,h}\|^2 + \gamma \| \nabla \cdot \bhu^{n+1}_{j,h} \|^2 + \mu \Delta t \| \mathcal{L}^n(\hat{u}^{'}_{h}) \nabla \bhu^{n+1}_{j,h}\|^2 = \left(\bif_{j}(t_{n+1}), \bhu^{n+1}_{j,h} \right) \notag \\ &- b^* \left( \bhu'^{n}_{j,h} , 2\bhu^n_{j,h} - \bhu^{n-1}_{j,h} , \bhu^{n+1}_{j,h} \right) - \left(\nu^{'}_j \nabla (2\bhu^n_{j,h} - \bhu^{n-1}_{j,h}) , \nabla \bhu^{n+1}_{j,h}\right). \label{NSE-eq-1}
		\end{align}
		Using the Cauchy-Schwarz, Young's, and H\"older's inequalities, we obtain
		\begin{align}
			\left(\bif_{j}(t_{n+1}), \bhu^{n+1}_{j,h} \right) &\leq \| \bif_{j}(t_{n+1}) \|_{-1} \| \nabla \bhu^{n+1}_{j,h} \|
			\leq \frac{\alpha_j}{4} \| \nabla \bhu^{n+1}_{j,h} \|^2 + \frac{1}{\alpha_j} \| \bif_{j}(t_{n+1}) \|^2_{-1}, \notag \\
			\left(\nu^{'}_{j} \nabla (2\bhu^n_{j,h} - \bhu^{n-1}_{j,h}) , \nabla \bhu^{n+1}_{j,h}\right)
			&\leq 2 \|\nu^{'}_{j}\|_{\infty} \| \nabla \bhu^n_{j,h}\|  \| \nabla \bhu^{n+1}_{j,h} \| + \|\nu^{'}_{j}\|_{\infty} \| \nabla \bhu^{n-1}_{j,h}\| \| \nabla \bhu^{n+1}_{j,h} \| \notag \\
			&\leq \frac{3}{2}\|\nu^{'}_j \|_{\infty} \| \nabla \bhu^{n+1}_{j,h} \|^2 + \|\nu^{'}_j\|_{\infty} \|\nabla \bhu^{n}_{j,h}\|^2 + \frac{1}{2} \|\nu^{'}_j\|_{\infty} \|\nabla \bhu^{n-1}_{j,h}\|^2 . \notag
		\end{align}
		Using identity \eqref{trilinear-identitiy} and the Cauchy-Schwarz, triangle,  H\"older's, and Poincaré inequalities to estimate the trilinear form, we obtain
		\begin{align}
			&- b^* \left( \bhu^{'n}_{j,h} , 2\bhu^n_{j,h} - \bhu^{n-1}_{j,h} , \bhu^{n+1}_{j,h} \right)
			= b^* \left( \bhu^{'n}_{j,h} , \bhu^{n+1}_{j,h}, 2\bhu^n_{j,h} - \bhu^{n-1}_{j,h} \right) \notag \\
			&= \left( \bhu^{'n}_{j,h} \cdot \nabla \bhu^{n+1}_{j,h}, (2\bhu^n_{j,h} - \bhu^{n-1}_{j,h}) \right) + \frac{1}{2} \left((\nabla \cdot \bhu^{'n}_{j,h}) \bhu^{n+1}_{j,h}, 2\bhu^n_{j,h} - \bhu^{n-1}_{j,h} \right) \notag \\
			&\leq \| \bhu^{'n}_{j,h} \cdot \nabla \bhu^{n+1}_{j,h}\| \|2 \bhu^n_{j,h}- \bhu^{n-1}_{j,h} \| + \frac{1}{2}\| \nabla \cdot \bhu^{'n}_{j,h}\|_{L^\infty} \| \bhu^{n+1}_{j,h} \| \|2 \bhu^n_{j,h}- \bhu^{n-1}_{j,h} \| \notag \\
			& \le C\left(2\|\nabla \bhu^{n}_{j,h} \| +  \| \nabla\bhu^{n-1}_{j,h} \|\right)\left(\| \bhu^{'n}_{j,h} \cdot \nabla \bhu^{n+1}_{j,h} \| + \frac12 \| \nabla\cdot \bhu^{'n}_{j,h}\|_{L^\infty} \| \nabla \bhu^{n+1}_{j,h} \|\right). \notag
		\end{align}
		Again, using \eqref{basic-ineq}--\eqref{basic-ineq-infinity} and Young's inequality, we obtain
		\begin{align} 
			&- b^* \left( \bhu^{'n}_{j,h} , 2\bhu^n_{j,h} - \bhu^{n-1}_{j,h} , \bhu^{n+1}_{j,h} \right)\le C\left(\|\nabla \bhu^{n}_{j,h} \| +  \| \nabla\bhu^{n-1}_{j,h} \|\right)\left(\|| \bhu^{'n}_{j,h}| \nabla \bhu^{n+1}_{j,h} \| + \frac{1}{2}\|\nabla \cdot \bhu^{'n}_{j,h}\|_{L^\infty} \| \nabla \bhu^{n+1}_{j,h} \|\right) \nonumber \\&\le \frac{\alpha_j}{4}\| \nabla \bhu^{n+1}_{j,h} \|^{2} + C\left(\|\nabla \bhu^{n}_{j,h} \| +  \| \nabla\bhu^{n-1}_{j,h} \|\right)\|\mathcal{L}^n(\hat{u}^{'}_{h}) \nabla \bhu^{n+1}_{j,h} \| + \frac{C}{\alpha_j}\left(\|\nabla \bhu^{n}_{j,h} \|^{2} +  \| \nabla\bhu^{n-1}_{j,h} \|^{2}\right) \|\nabla \cdot \bhu^{'n}_{j,h}\|^{2}_{L^\infty} \nonumber\\
			& \le \frac{\alpha_j}{4}\| \nabla \bhu^{n+1}_{j,h} \|^{2} + \frac{\mu \Delta t}{2} \|\mathcal{L}^n(\hat{u}^{'}_{h}) \nabla \bhu^{n+1}_{j,h} \|^{2} + \left(\frac{C}{\mu \Delta t} + \frac{C}{\alpha_j} \|\nabla \cdot \bhu^{'n}_{j,h}\|^{2}_{L^\infty}\right) \left(\|\nabla \bhu^{n}_{j,h} \|^{2} +  \| \nabla\bhu^{n-1}_{j,h} \|^{2}\right).
		\end{align}
		Applying the above bounds, omitting the nonnegative term from the left-hand side, and rearranging \eqref{NSE-eq-1}, yields
		\begin{align}
			&\frac{1}{4 \Delta t} \left( \|\bhu^{n+1}_{j,h}\|^2 + \|2\bhu^{n+1}_{j,h} - \btu^{n}_{j,h}\|^2 - \|\btu^{n}_{j,h}\|^2 - \|2\btu^{n}_{j,h} - \btu^{n-1}_{j,h}\|^2 \right) + \frac{\bar{\nu}_{\min}}{2}  \| \nabla \bhu^{n+1}_{j,h}\|^2  + \gamma \| \nabla \cdot \bhu^{n+1}_{j,h} \|^2\notag \\ & + \frac{\mu \Delta t}{2} \| \mathcal{L}^n(\hat{u}^{'}_{h}) \nabla \bhu^{n+1}_{j,h}\|^2 \leq \frac{1}{\alpha_j} \| \bif_{j}(t_{n+1}) \|^2_{-1} + \left(\frac{C}{\mu \Delta t} + \frac{C}{\alpha_j} \|\nabla \cdot \bhu^{'n}_{j,h}\|^{2}_{L^\infty} + \|\nu^{'}_j\|_{\infty} \right) \| \nabla\bhu^n_{j,h}\|^2 \notag \\
			&+ \left( \frac{C}{\mu \Delta t} + \frac{C}{\alpha_j} \|\nabla \cdot \bhu^{'n}_{j,h}\|^{2}_{L^\infty} + \frac{1}{2} \|\nu^{'}_j\|_{\infty} \right) \| \nabla\bhu^{n-1}_{j,h}\|^2. \label{NSE-eq-2}
		\end{align}
		Now, from Step 2 in \eqref{NSE-step-2-1}, adding and subtracting $\btu^n_{j,h}$, gives
		\begin{align}
			\frac{3}{4\Delta t} \left(2\btu^{n+1}_{j,h} - 2\bhu^{n+1}_{j,h} + \btu^n_{j,h} - \btu^n_{j,h}, \bv_h \right) - \left(\hat{p}^{n+1}_{j,h}, \nabla \cdot \bv_h \right) = 0. \nonumber
		\end{align}
		Rearranging and choosing $\bv_h = 2\btu^{n+1}_{j,h} - \btu^{n}_{j,h}$, we write
		\begin{align} 
			\frac{3}{4\Delta t} \left(2\btu^{n+1}_{j,h} - \btu^n_{j,h}  - (2\bhu^{n+1}_{j,h} - \btu^n_{j,h}) , 2\btu^{n+1}_{j,h} - \btu^n_{j,h} \right) = 0. \nonumber
		\end{align}
		Now, using the Cauchy-Schwarz inequality, we have
		\begin{align}
			\|2\btu^{n+1}_{j,h} - \btu^n_{j,h}\|^2 \leq \|2\bhu^{n+1}_{j,h} - \btu^n_{j,h} \| \|2\btu^{n+1}_{j,h} - \btu^n_{j,h}\|, \nonumber
		\end{align}
		which gives
		\begin{align}
			\|2\btu^{n+1}_{j,h} - \btu^n_{j,h}\| \leq \|2\bhu^{n+1}_{j,h} - \btu^n_{j,h} \|. \label{bound-tild}
		\end{align}
		Similarly, substituting $\bv_h= \btu^{n+1}_{j,h}$ in \eqref{NSE-step-2-1}, we have
		\begin{align}
			\frac{3}{2\Delta t} \left(\btu^{n+1}_{j,h} - \bhu^{n+1}_{j,h}, \btu^{n+1}_{j,h} \right) = 0. \nonumber
		\end{align}
		Using the Cauchy-Schwarz inequality, we obtain
		\begin{align}
			\|\btu^{n+1}_{j,h}\|^2 \leq \| \bhu^{n+1}_{j,h} \| \|\btu^{n+1}_{j,h}\|.\nonumber
		\end{align}
		Therefore
		\begin{align}
			\|\btu^{n+1}_{j,h}\| \leq \| \bhu^{n+1}_{j,h} \|.\label{tild-hat}
		\end{align}
		Substituting the above bound into \eqref{NSE-eq-2}, and rearranging yields
		\begin{align}
			&\frac{1}{4\Delta t} \left( \|\bhu^{n+1}_{j,h}\|^2- \|\bhu^{n}_{j,h}\|^2 + \|2\bhu^{n+1}_{j,h} - \btu^n_{j,h} \|^2 - \|2\bhu^{n}_{j,h} - \btu^{n-1}_{j,h} \|^2 \right) + \frac{\bar{\nu}_{\min}}{2} \left(\| \nabla \bhu^{n+1}_{j,h}\|^2 - \| \nabla \bhu^{n}_{j,h}\|^2 \right) \notag \\ &+ \gamma \| \nabla \cdot \bhu^{n+1}_{j,h} \|^2 + \left(\frac{\bar{\nu}_{\min}}{2}-\|\nu^{'}_j\|_{\infty} -\frac{C}{\mu \Delta t} - \frac{C}{\alpha_j} \|\nabla \cdot \bhu^{'n}_{j,h}\|^{2}_{L^\infty} \right) \left(\| \nabla \bhu^{n}_{j,h}\|^2 - \| \nabla \bhu^{n-1}_{j,h}\|^2 \right)\nonumber\\& + \left(\frac{\alpha_j}{2} -\frac{C}{\mu \Delta t} - \frac{C}{\alpha_j} \|\nabla \cdot \bhu^{'n}_{j,h}\|^{2}_{L^\infty} \right) \| \nabla \bhu^{n-1}_{j,h} \|^2    + \frac{\mu \Delta t}{2} \|\mathcal{L}^n(\hat{u}^{'}_{h}) \nabla \bhu^{n+1}_{j,h}\|^2 \leq \frac{1}{\alpha_j} \| \bif_{j}(t_{n+1}) \|^2_{-1}. \notag
		\end{align}
		After multiplying both sides by $4 \Delta t$, taking the sum over $n=1,2,\cdots,M-1$, omitting the nonnegative term on the left-hand side, and invoking the triangle inequality, Young's inequality, and \eqref{tild-hat}, we arrive at
		\begin{align}
			& \|\bhu^{M}_{j,h}\|^2  +2 \bar{\nu}_{\min}\Delta t\| \nabla \bhu^{M}_{j,h} \|^2+ \left(2\bar{\nu}_{\min}-4\|\nu^{'}_j\|_{\infty} -\frac{C}{\mu \Delta t} - \frac{C}{\alpha_j} D_{\infty}^2 \right)\Delta t\|\nabla\bhu_{j,h}^{M-1}\|^2\nonumber\\&+4\zeta\Delta t\sum_{n=1}^{M-2}\|\nabla\bhu_{j,h}^{n}\|^2+ 4 \Delta t \gamma \sum^{M}_{n=2} \| \nabla \cdot \bhu^{n}_{j,h} \|^2 \nonumber\\&\leq \frac{4 \Delta t}{\alpha_j} \sum^{M}_{n=2} \| \bif_{\hspace{0.4mm}{j}}(t_{n}) \|^2_{-1}  + 6\|\bhu^{1}_{j,h}\|^2 + 2\| \bhu^{0}_{j,h} \|^2 + 2 \bar{\nu}_{\min} \Delta t\left( \| \nabla \bhu^{0}_{j,h}\|^2+\| \nabla \bhu^{1}_{j,h}\|^2\right). 
		\end{align}
		If $\alpha_{\min}>CD_{\infty}$, choose $\mu>\frac{C\alpha_{\min}}{\Delta t\left(\alpha_{\min}^2-CD_{\infty}^2\right)},$
		and simplify to complete the proof.
	\end{proof}
	
	We now prove that the penalty-projection-based Algorithm \ref{NSE-FEM} converges to coupled Algorithm \ref{coupled-alg-com} as $\gamma\rightarrow\infty$. Thus, we need to define the space $\bR_h:=\bV_h^\perp\subset\bX_h$ as the orthogonal complement of $\bV_h$ with respect to the norm $\bH^1(\cD)$.
	
	For the analysis in this paper, we additionally assume that the finite element spaces satisfy $\nabla\cdot\bX_h\subset Q_h$. Although the Taylor–Hood elements employed in our numerical experiments do not generally satisfy this inclusion, sufficiently large grad-div stabilization strongly penalizes their discrete divergence, and, under appropriate finite-element and mesh assumptions, grad-div-stabilized Taylor–Hood solutions are known to converge to corresponding pointwise divergence-free solutions as $\gamma\to\infty$.
	
	\begin{lemma}\label{CR-lemma}
		Let the finite element pair $(\bX_h,Q_h)\subset(\bX,Q)$ satisfy the \textit{inf-sup condition} \eqref{infsup} and the divergence-free property, i.e., $\nabla\cdot\bX_h\subset Q_h$. Then there exists a constant $C_R$ independent of $h$ such that $$\|\nabla\bv_h\|\le C_R\|\nabla\cdot\bv_h\|,\hspace{3mm}\forall\bv_h\in \bR_h.$$
	\end{lemma}
	\begin{proof}
		See \cite{GR86, linke2017connection}.
	\end{proof}

	We now combine \eqref{NSE-step-1} and \eqref{NSE-step-2-1} from Step 1 and Step 2 of Algorithm \ref{NSE-FEM}. Since $\bX_h\subset\bY_h$, we choose $\bv_h=\bchi_h$ in \eqref{NSE-step-2-1}
	and substitute $\btu^{n}_{j,h}$ and $\btu^{n-1}_{j,h}$ into \eqref{NSE-step-1} to obtain
	\begin{align}
		&\frac{1}{2\Delta t} \left({3}\bhu^{n+1}_{j,h} - 4\bhu^n_{j,h} + \bhu^{n-1}_{j,h}, \bchi_{h}\right) + b^* \left(<{\bhu}_h>^n, \bhu^{n+1}_{j,h}, \bchi_{h}\right) + \left(\bar{\nu}\nabla \bhu^{n+1}_{j,h}, \nabla \bchi_{h}\right) \notag \\ &+ \gamma \left(\nabla \cdot \bhu^{n+1}_{j,h}, \nabla \cdot\bchi_{h}\right) -\frac{1}{3} \left(4\hat{p}^{n}_{j,h}- \hat{p}^{n-1}_{j,h}, \nabla \cdot \bchi_h\right) + \mu \Delta t \left((\mathcal{L}^n(\hat{u}^{'}_{h}))^2 \nabla\bhu^{n+1}_{j,h}, \nabla\bchi_{h} \right) \notag \\ &= \left(\bif_{j}(t_{n+1}), \bchi_{h}\right) - b^* \left(\bhu^{'n}_{j,h} , 2\bhu^{n}_{j,h} - \bhu^{n-1}_{j,h} , \bchi_{h}\right) - \left(\nu^{'}_j \nabla (2\bhu^n_{j,h} - \bhu^{n-1}_{j,h}) , \nabla \bchi_{h}\right). \label{NSE-error-1}
	\end{align}
	\begin{theorem} (Convergence)\label{Eddy-viscosity-convergence} Let $(\bu_{j,h}^{n+1}
		,p_{j,h}^{n+1})$ and $(\bhu_{j,h}^{n+1},\tilde{\bu}_{j,h}^{n+1}
		,\hp_{j,h}^{n+1})$ be the solutions of Algorithms \ref{coupled-alg-com} and \ref{NSE-FEM}, respectively, for $n=1,2,\cdots,M-1$. For a given $\gamma>0$, and $\Delta t>0$, if we choose $$\mu>\frac{C\alpha_{\min}}{\Delta t\left(\alpha_{\min}^2-CD_{\infty}^2\right)},$$ then
		\begin{align}
			&\left(\Delta t\sum_{n=2}^M\|\nabla\hspace{-1mm}\lab\bu_h\rab^n-\nabla\hspace{-1mm}\lab\bhu_h\rab^n\|^2\right)^{\frac12}\leq \frac{CC_{R}}{\alpha_{\min}^{\frac12}\gamma} \exp\left\{\frac{C}{\alpha_{\min}} \left(\frac{\Delta t}{h^3\alpha_{\min}}+1\right)\right\}\nonumber\\&\times\Bigg[1+\exp\left\{ \frac{C}{\alpha_{\min}} \left(\frac{\alpha_{\min}}{\Delta t}+\alpha_{\min}+D_{\infty}^2\right)\right\}  \left(\frac{1}{\Delta t\alpha_{\min}\zeta}+\frac{1}{\Delta t}+\frac{1}{\Delta t^2}+\Delta t\right)\Bigg]^{\frac12},\label{gamma-theorem}
		\end{align}
		and
		\begin{align}
			&\left\{\Delta t\sum_{j=1}^{J}\sum_{n=1}^{M-1}\Big\|p^{n+1}_{j,h} -\left( \frac{4}{3}\hat{p}^n_{j,h} - \frac{1}{3}\hat{p}^{n-1}_{j,h}-\gamma\nabla \cdot \bhu^{n+1}_{j,h}\right)\Big\|^2\right\}^{\frac12}\le \frac{CC_{R}}{\alpha_{\min}^{\frac12}\gamma} \exp\left\{\frac{C}{\alpha_{\min}} \left(\frac{\Delta t}{h^3\alpha_{\min}}+1\right)\right\}\nonumber\\&\times\Bigg[1+\exp\left\{ \frac{C}{\alpha_{\min}} \left(\frac{\alpha_{\min}}{\Delta t}+\alpha_{\min}+D_{\infty}^2\right)\right\}  \left(\frac{1}{\Delta t\alpha_{\min}\zeta}+\frac{1}{\Delta t}+\frac{1}{\Delta t^2}+\Delta t\right)\Bigg]^{\frac12}\nonumber\\ &\times\left(\frac{1}{\zeta\Delta t}+1+\max_{1\le j\le J}\|\nu_{j}^{'}\|_{\infty}^2+\frac{1}{\Delta t^2}+\frac{\Delta t}{\alpha_{\min}h^3}+\Delta t^2+\frac{1}{\alpha_{\min}\zeta^{\frac12}\Delta t^{\frac12}}+\|\bar{\nu}\|_{\infty}^2\right)^{\frac12}.
		\end{align}
		
	\end{theorem}
	
	\begin{proof}
		Denote $\be_{j}^{n+1}:=\bu_{j,h}^{n+1}-\bhu_{j,h}^{n+1}$, $\tilde{\be}_{j}^{n+1}:=\bu_{j,h}^{n+1}-\btu_{j,h}^{n+1}$ and use the following $H^1$-orthogonal decomposition of the error
		$$\be_{j}^{n+1}:=\be_{j,0}^{n+1}+\be_{j,\bR}^{n+1},$$
		with $\be_{j,0}^{n+1}\in\bV_h$ and $\be_{j,\bR}^{n+1}\in\bR_h$, for $n=0,1,\cdots,M-1$.\\
		\textbf{Step 1:} Estimate of $\be_{j,\bR}^{n+1}$: Subtracting \eqref{NSE-step-1} from \eqref{couple-eqn-1-new} gives\begin{align}
			&\frac{1}{2 \Delta t} \left(3 \be^{n+1}_j - 4\tilde{\be}^n_{j} + \tilde{\be}^{n-1}_{j}, \bchi_{h} \right) + b^* \left(<{\bhu}_h>^n, \be^{n+1}_j , \bchi_{h} \right) + b^* \left(<{\be}>^n, \bu^{n+1}_{j,h}, \bchi_{h}\right)  \nonumber \\ &+ \left(\bar{\nu}\nabla \be^{n+1}_{j}, \nabla \bchi_{h}\right) + \gamma \left(\nabla \cdot \be^{n+1}_{j}, \nabla \cdot\bchi_{h}\right) - \left(p^{n+1}_{j,h}, \nabla \cdot 
			\bchi_{h} \right)  + \left(\nu_T(\hat{u}_h^{'},t_n)\nabla \be_j^{n+1},\nabla\bchi_h\right)\nonumber\\& + \mu\Delta t\left(\big\{(\mathcal{L}^n(u^{'}_{h}))^2-(\mathcal{L}^n(\hat{u}^{'}_{h}))^2\big\}\nabla\bu_{j,h}^{n+1},\nabla\bchi_h\right) = - b^* \left(\bhu^{'n}_{j,h} , 2\be^{n}_{j} - \be^{n-1}_{j} , \bchi_{h}\right) \nonumber \\ & - b^* \left(\be^{'n}_{j} , 2\bu^{n}_{j,h} - \bu^{n-1}_{j,h} , \bchi_{h}\right) - \left(\nu^{'}_j \nabla (2\be^n_{j} - \be^{n-1}_{j}) , \nabla \bchi_{h}\right). \label{NSE-error-3-n}
		\end{align}
		Take $\bchi_h=\be_j^{n+1}$ in \eqref{NSE-error-3-n} which yields $ b^* \left(<{\bhu}_h>^n, \be^{n+1}_j , \be^{n+1}_{j} \right)=0$, and use the identity in \eqref{BDF-2-identity} to get
		\begin{align}
			&\frac{1}{4 \Delta t} \left( \|\be^{n+1}_{j}\|^2 + \|2\be^{n+1}_{j} - \tilde{\be}^{n}_{j}\|^2 - \|\tilde{\be}^{n}_{j}\|^2 - \|2\tilde{\be}^{n}_{j} - \tilde{\be}^{n-1}_{j}\|^2 + \|\be^{n+1}_{j} - 2\tilde{\be}^{n}_{j} + \tilde{\be}^{n-1}_{j}\|^2\right) \nonumber \\ &+ b^* \left(<{\be}>^n, \bu^{n+1}_{j,h}, \be^{n+1}_{j}\right) + \|\bar{\nu}^{\frac{1}{2}}\nabla \be^{n+1}_{j}\|^2 + \gamma \|\nabla \cdot \be^{n+1}_{j,\bR}\|^2 - \left(p^{n+1}_{j,h} , \nabla \cdot 
			\be^{n+1}_{j,\bR} \right) \nonumber \\ & + \mu\Delta t\|\mathcal{L}^n(\hat{u}^{'}_{h})\nabla\be_j^{n+1}\|^2\nonumber+\mu\Delta t\Big(\big\{(\mathcal{L}^n(u^{'}_{h}))^2-(\mathcal{L}^n(\hat{u}^{'}_{h}))^2\big\}\nabla\bu_{j,h}^{n+1},\nabla\be_j^{n+1}\Big) \nonumber \\ &=  - b^* \left(\bhu^{'n}_{j,h} , 2\be^{n}_{j} - \be^{n-1}_{j} , \be^{n+1}_{j}\right) - b^* \left(\be^{'n}_{j} , 2\bu^{n}_{j,h} - \bu^{n-1}_{j,h} , \be^{n+1}_{j}\right) - \left(\nu^{'}_j \nabla (2\be^n_{j} - \be^{n-1}_{j}) , \nabla \be^{n+1}_{j}\right). \label{NSE-error-4-n}
		\end{align}
		Now, using nonlinear bound from \eqref{nonlinearbound3}, estimate in Lemma \ref{lemma-L3-infty}, and Young's inequality on the trilinear form in the left-hand side
		\begin{align}
			b^* \left(<{\be}>^n, \bu^{n+1}_{j,h}, \be^{n+1}_{j}\right) &\leq C \|<{\be}>^n\| \left( \|\nabla \bu^{n+1}_{j,h}\|_{L^3} + \| \bu^{n+1}_{j,h}\|_{L^\infty} \right) \| \nabla \be^{n+1}_{j}\| \nonumber\\ 
			&\leq CC_{*} \|<{\be}>^n\| \| \nabla \be^{n+1}_{j}\| \leq \frac{\alpha_j}{8} \|\nabla \be^{n+1}_{j}\|^2 + \frac{C}{\alpha_j}\|<{\be}>^n\|^2.\label{nbd-1}
		\end{align}
		Applying Cauchy-Schwarz and Young’s inequalities, we have
		\begin{align}
			\left(p^{n+1}_{j,h} , \nabla \cdot 
			\be^{n+1}_{j,\bR} \right) \leq \frac{1}{2\gamma} \|p^{n+1}_{j,h}\|^2 + \frac{\gamma}{2} \|\nabla \cdot \be^{n+1}_{j,\bR}\|^2. \nonumber
		\end{align}
		Using H\"older's and Young's inequalities, yields
		\begin{align}
			\left(\nu^{'}_{j} \nabla (2\be^n_{j} - \be^{n-1}_{j}) , \nabla \be^{n+1}_{j}\right)
			&\leq 2 \|\nu^{'}_{j}\|_{\infty} \| \nabla \be^n_{j}\|  \| \nabla \be^{n+1}_{j} \| + \|\nu^{'}_{j}\|_{\infty} \| \nabla \be^{n-1}_{j}\| \| \nabla \be^{n+1}_{j} \| \nonumber \\
			&\leq \|\nu^{'}_j \|_{\infty} \| \nabla \be^{n+1}_{j} \|^2 + \|\nu^{'}_j\|_{\infty} \|\nabla \be^{n}_{j}\|^2 + \frac{1}{2} \|\nu^{'}_j\|_{\infty} \| \nabla \be^{n+1}_{j} \|^2 + \frac{1}{2} \|\nu^{'}_j\|_{\infty} \|\nabla \be^{n-1}_{j}\|^2 \nonumber
			\\
			&\leq \frac{3}{2}\|\nu^{'}_j \|_{\infty} \| \nabla \be^{n+1}_{j} \|^2 + \|\nu^{'}_j\|_{\infty} \|\nabla \be^{n}_{j}\|^2 + \frac{1}{2} \|\nu^{'}_j\|_{\infty} \|\nabla \be^{n-1}_{j}\|^2 . \label{vis-fl}
		\end{align}
		Using identity \eqref{trilinear-identitiy}, the Cauchy-Schwarz, H\"older's, Poincar\'e and triangle inequalities to estimate the trilinear form, we obtain
		\begin{align}
			&- b^* \left( \bhu^{'n}_{j,h} , 2\be^n_{j} - \be^{n-1}_{j} , \be^{n+1}_{j} \right)
			= b^* \left( \bhu^{'n}_{j,h} , \be^{n+1}_{j}, 2\be^n_{j} - \be^{n-1}_{j} \right) \notag \\
			&= \left( \bhu^{'n}_{j,h} \cdot \nabla \be^{n+1}_{j}, (2\be^n_{j} - \be^{n-1}_{j}) \right) + \frac{1}{2} \left((\nabla \cdot \bhu^{'n}_{j,h}) \be^{n+1}_{j}, 2\be^n_{j} - \be^{n-1}_{j} \right) \notag \\
			&\leq \| \bhu^{'n}_{j,h} \cdot \nabla \be^{n+1}_{j}\| \|2 \be^n_{j}- \be^{n-1}_{j} \| + \frac{1}{2}\| \nabla \cdot \bhu^{'n}_{j,h}\|_{L^\infty} \| \be^{n+1}_{j} \| \|2 \be^n_{j}- \be^{n-1}_{j} \| \notag \\
			& \le C\left(2\|\nabla \be^{n}_{j} \| +  \| \nabla\be^{n-1}_{j} \|\right)\left(\| \bhu^{'n}_{j,h} \cdot \nabla \be^{n+1}_{j} \| + \frac12 \| \nabla\cdot \bhu^{'n}_{j,h}\|_{L^\infty} \| \nabla \be^{n+1}_{j} \|\right). \notag
		\end{align}
		Again, using \eqref{basic-ineq} and Young's inequality, gives\begin{align} 
			&- b^* \left( \bhu^{'n}_{j,h} , 2\be^n_{j} - \be^{n-1}_{j} , \be^{n+1}_{j} \right)\le C\left(\|\nabla \be^{n}_{j} \| +  \| \nabla\be^{n-1}_{j} \|\right)\left(\|| \bhu^{'n}_{j,h}| \nabla \be^{n+1}_{j} \| + \frac{1}{2}\|\nabla \cdot \bhu^{'n}_{j,h}\|_{L^\infty} \| \nabla \be^{n+1}_{j} \|\right) \nonumber \\&\le \frac{\alpha_j}{8}\| \nabla \be^{n+1}_{j} \|^{2} + C\left(\|\nabla \be^{n}_{j} \| +  \| \nabla\be^{n-1}_{j} \|\right)\|\mathcal{L}^n(\hat{u}^{'}_{h}) \nabla \be^{n+1}_{j} \| + \frac{C}{\alpha_j}\left(\|\nabla \be^{n}_{j} \|^{2} +  \| \nabla\be^{n-1}_{j} \|^{2}\right) \|\nabla \cdot \bhu^{'n}_{j,h}\|^{2}_{L^\infty} \nonumber\\
			& \le \frac{\alpha_j}{8}\| \nabla \be^{n+1}_{j} \|^{2} + \frac{\mu \Delta t}{2} \|\mathcal{L}^n(\hat{u}^{'}_{h}) \nabla \be^{n+1}_{j} \|^{2} + \left(\frac{C}{\mu \Delta t} + \frac{C}{\alpha_j} \|\nabla \cdot \bhu^{'n}_{j,h}\|^{2}_{L^\infty}\right) \left(\|\nabla \be^{n}_{j} \|^{2} +  \| \nabla\be^{n-1}_{j} \|^{2}\right).\label{nonlin-bd-1}
		\end{align}
		Using the nonlinear bound in \eqref{nonlinearbound3}, triangle inequality, the estimate in Lemma \ref{lemma-L3-infty}, and Young's inequality provides
		\begin{align}
			-b^* \left(\be^{'n}_{j} , 2\bu^{n}_{j,h} - \bu^{n-1}_{j,h} , \be^{n+1}_{j}\right) &\leq C \|\be^{'n}_{j}\| \left( \|\nabla (2\bu^{n}_{j,h} - \bu^{n-1}_{j,h}) \|_{L^3} + \|2\bu^{n}_{j,h} - \bu^{n-1}_{j,h}\|_{L^\infty}\right) \|\nabla \be^{n+1}_{j}\| \nonumber \\
			& \leq CC_{*}\|\be^{'n}_{j}\|\|\nabla \be^{n+1}_{j}\| \leq \frac{\alpha_j}{8} \|\nabla \be^{n+1}_{j}\|^2 + \frac{C}{\alpha_j}\| \be^{'n}_{j}\|^2.\label{nonlin-bd-2-new2}
		\end{align}
		We apply H\"older’s inequality to get
		\begin{align}
			\mu\Delta t\Big(\big\{(\mathcal{L}^n(u^{'}_{h}))^2-&(\mathcal{L}^n(\hat{u}^{'}_{h}))^2\big\}\nabla\bu_{j,h}^{n+1},\nabla\be_{j}^{n+1}\Big)\nonumber\\&\leq \mu\Delta t\|(\mathcal{L}^n(u^{'}_{h}))^2-(\mathcal{L}^n(\hat{u}^{'}_{h}))^2\|_{L^\infty}\|\nabla\bu_{j,h}^{n+1}\|\|\nabla\be_j^{n+1}\|.\label{prandtl-term}
		\end{align}
		Apply the triangle inequality, Lemma \ref{lemma-L3-infty}  and Lemma \ref{uniform-boundedness-lemma-proof},  Agmon’s inequality \cite{Robinson2016Three-Dimensional}, and a discrete inverse inequality to obtain
		\begin{align}
			\|(\mathcal{L}^n(u^{'}_{h}))^2-(\mathcal{L}^n(\hat{u}^{'}_{h}))^2\|_{L^\infty} &=\|\sum_{i=1}^J\left(|\bu_{i,h}^{'n}|^2-|\bhu_{i,h}^{'n}|^2\right)\|_{L^\infty}\nonumber\\
			&\leq\sum_{i=1}^J\|(\bu_{i,h}^{'n}-\bhu_{i,h}^{'n})\cdot(\bu_{i,h}^{'n}+\bhu_{i,h}^{'n})\|_{L^\infty}\nonumber\\&\leq \sum_{i=1}^J\|\bu_{i,h}^{'n}-\bhu_{i,h}^{'n}\|_{L^\infty}\|\bu_{i,h}^{'n}+\bhu_{i,h}^{'n}\|_{L^\infty}\nonumber\\
			&\le\sum_{i=1}^J\|\be_i^{'n}\|_{L^\infty}\left(\|\bu_{i,h}^{'n}\|_{L^\infty}+\|\bhu_{i,h}^{'n}\|_{L^\infty}\right)\nonumber\\ 
			&\le \sum_{i=1}^J\|\be_i^{'n}\|_{L^\infty}\le Ch^{-\frac32} \sum_{i=1}^J \|\be^{'n}_{i}\|. \label{prandtl-diff}
		\end{align}
		Use \eqref{prandtl-diff} in \eqref{prandtl-term}, the stability estimate for Algorithm \ref{coupled-alg-com}, and Young's inequalities to get
		\begin{align}
			\mu\Delta t\Big(\big\{(\mathcal{L}^n(u^{'}_{h}))^2-&(\mathcal{L}^n(\hat{u}^{'}_{h}))^2\big\}\nabla\bu_{j,h}^{n+1},\nabla\be_{j}^{n+1}\Big)\leq C h^{-\frac32}\Delta t \|\nabla\bu_{j,h}^{n+1}\|\|\nabla\be_j^{n+1}\|\sum_{i=1}^J \|\be^{'n}_{i}\|\nonumber\\&\le C h^{-\frac32}\Delta t^{\frac12}\alpha_{j}^{-\frac12} \sum_{i=1}^J \|\be^{'n}_{i}\|\|\nabla\be_j^{n+1}\|\nonumber\\&\le \frac{\alpha_j}{8}\|\nabla\be^{n+1}_{j}\|^2+\frac{C\Delta t}{h^3\alpha_{j}^2}\sum_{i=1}^J\|\be^{'n}_{i}\|^2.
		\end{align}
		Using the above estimates in \eqref{NSE-error-4-n} and dropping the nonnegative term involving $\|\be^{n+1}_{j} - 2\tilde{\be}^{n}_{j} + \tilde{\be}^{n-1}_{j}\|^2$, produces
		\begin{align}
			&\frac{1}{4 \Delta t} \left( \|\be^{n+1}_{j}\|^2 + \|2\be^{n+1}_{j} - \tilde{\be}^{n}_{j}\|^2 - \|\tilde{\be}^{n}_{j}\|^2 - \|2\tilde{\be}^{n}_{j} - \tilde{\be}^{n-1}_{j}\|^2\right)+\frac{\bar{\nu}_{\min}}{2}\|\nabla \be^{n+1}_{j}\|^2+\frac{\gamma}{2}\|\nabla \cdot \be^{n+1}_{j,\bR}\|^2 \nonumber \\ &+\frac{\mu \Delta t}{2} \|\mathcal{L}^n(\hat{u}^{'}_{h})\nabla\be^{n+1}_{j}\|^2 \leq \left(\frac{C}{\mu \Delta t} + \frac{C}{\alpha_j} \|\nabla \cdot \bhu^{'n}_{j,h}\|^{2}_{L^\infty}+ \|\nu^{'}_j\|_{\infty}\right)\|\nabla \be^{n}_{j}\|^2+ \frac{C\Delta t}{h^3\alpha_j^2}\sum_{i=1}^J\|\be^{'n}_{i}\|^2\nonumber \\ & +\left(\frac{C}{\mu \Delta t} + \frac{C}{\alpha_j} \|\nabla \cdot \bhu^{'n}_{j,h}\|^{2}_{L^\infty}+ \frac{1}{2} \|\nu^{'}_j\|_{\infty}\right)\|\nabla \be^{n-1}_{j}\|^2+ \frac{1}{2\gamma} \|p^{n+1}_{j,h}\|^2 + \frac{C}{\alpha_{j}} \left(\|\hspace{-1mm}<\hspace{-1mm}{\be}\hspace{-1mm}>^n\hspace{-1mm}\|^2 + \| \be^{'n}_{j}\|^2\right). \label{NSE-error-6-n}
		\end{align}
		On the left hand side, we will bound the term $-\|\tilde{\be}^n\|^2$ and $-\|2{\be}^{n+1}-\tilde{\be}^n\|^2$. Choose $\bv_h=\tilde{\be}_{j}^{n+1}$ in \eqref{NSE-step-2-1} to get \begin{align}
			\|\tilde{\be}_j^{n+1}\|\le\|\be_j^{n+1}\|.\label{tilde-bd1}
		\end{align} Now for \eqref{NSE-step-2-1}, we have,
		\begin{align*}
			\frac{3}{2\Delta t} \left(\tilde{\be}^{n+1}_{j} - \be^{n+1}_{j}, \bv_h \right) - \left(\hat{p}^{n+1}_{j,h}, \nabla \cdot \bv_h \right) = 0.
		\end{align*}
		Choosing $\bv_h=2\tilde{\be}^{n+1}-\tilde{\be}^n$, the pressure term vanishes. Adding and subtracting $\tilde{\be}_j^n$ and using the Cauchy-Schwarz inequality, we get
		\begin{align}
			\|2\tilde{\be}^{n+1}_{j}-\tilde{\be}^n_{j}\|\le\|2\be^{n+1}_{j}-\tilde{\be}^n_{j}\|.\label{tilde-bd2}
		\end{align}
		Substituting the above bounds into \eqref{NSE-error-6-n} and rearranging yields,
		\begin{align}
			&\frac{1}{4\Delta t} \left( \|\be^{n+1}_{j}\|^2- \|\be^{n}_{j}\|^2 + \|2\be^{n+1}_{j} - \tilde{\be}^n_{j} \|^2 - \|2\be^{n}_{j} - \tilde{\be}^{n-1}_{j} \|^2 \right) + \frac{\bar{\nu}_{\min}}{2} \left(\| \nabla \be^{n+1}_{j,h}\|^2 - \| \nabla \be^{n}_{j,h}\|^2 \right) \notag \\ &+ \frac{\gamma}{2}\|\nabla \cdot \be^{n+1}_{j,\bR}\|^2+ \left(\frac{\bar{\nu}_{\min}}{2}-\|\nu^{'}_j\|_{\infty} -\frac{C}{\mu \Delta t} - \frac{C}{\alpha_j} \|\nabla \cdot \bhu^{'n}_{j,h}\|^{2}_{L^\infty} \right) \left(\| \nabla \be^{n}_{j,h}\|^2 - \| \nabla \be^{n-1}_{j,h}\|^2 \right)\nonumber\\& + \left(\frac{\alpha_j}{2} -\frac{C}{\mu \Delta t} - \frac{C}{\alpha_j} \|\nabla \cdot \bhu^{'n}_{j,h}\|^{2}_{L^\infty} \right) \| \nabla \be^{n-1}_{j,h} \|^2    + \frac{\mu \Delta t}{2} \|\mathcal{L}^n(\hat{u}^{'}_{h}) \nabla \be^{n+1}_{j,h}\|^2 \leq  \frac{1}{2\gamma} \|p^{n+1}_{j,h}\|^2 \nonumber\\&+ \frac{C}{\alpha_{j}} \left(\|\hspace{-1mm}<\hspace{-1mm}{\be}\hspace{-1mm}>^n\hspace{-1mm}\|^2 + \| \be^{'n}_{j}\|^2\right) + \frac{C\Delta t}{h^3\alpha_j^2}\sum_{i=1}^J\|\be^{'n}_{i}\|^2. \notag
		\end{align}
		Now, multiplying both sides by $4\Delta t$ and summing over the time steps $n=1,2,\cdots,M-1$, we obtain
		\begin{align}
			& \|\be^{M}_{j}\|^2 +\|2\be_j^M-\tilde{\be}_j^{M-1}\|^2+4\zeta\Delta t\sum_{n=1}^{M}\|\nabla\be_{j}^{n}\|^2+ 2 \gamma\Delta t \sum_{n=2}^{M}\|\nabla \cdot \be^{n}_{j,\bR}\|^2+ 2\mu \Delta t\sum_{n=2}^M \|\mathcal{L}^n(\hat{u}^{'}_{h}) \nabla \be^{n+1}_{j,h}\|^2\nonumber\\& \leq  6\|\be^{1}_{j}\|^2 + 2\| \be^{0}_{j} \|^2 + 2 \bar{\nu}_{\min} \Delta t\left( \| \nabla \be^{0}_{j}\|^2+\| \nabla \be^{1}_{j}\|^2\right) \nonumber\\&+ \frac{C}{\alpha_j} \Delta t \sum^{M-1}_{n=1} \left(\|\hspace{-1mm}<\hspace{-1mm}{\be}\hspace{-1mm}>^n\hspace{-1mm}\|^2 + \| \be^{'n}_{j}\|^2 + \frac{\Delta t}{\alpha_{j}h^3}\sum_{i=1}^J\|\be^{'n}_{i}\|^2\right)+\frac{2}{\gamma}\Delta t\sum_{n=2}^{M}\|p^{n}_{j,h}\|^2. 
		\end{align}
		If $\alpha_{\min}>CD_{\infty}$, choosing $\mu>\frac{C\alpha_{\min}}{\Delta t\left(\alpha_{\min}^2-CD_{\infty}^2\right)}$ and
		dropping the nonnegative term from left-hand-side yields
		\begin{align}
			& \|\be^{M}_{j}\|^2+  \gamma\Delta t \sum^{M}_{n=2} \| \nabla \cdot \be^{n}_{j,\bR} \|^2 \leq  6\|\be^{1}_{j}\|^2 + 2\| \be^{0}_{j} \|^2 + 2 \bar{\nu}_{\min} \Delta t\left( \| \nabla \be^{0}_{j,h}\|^2+\| \nabla \be^{1}_{j,h}\|^2\right)\nonumber\\&+ \frac{C}{\alpha_{\min}} \Delta t \sum^{M-1}_{n=1} \left(\|\hspace{-1mm}<\hspace{-1mm}{\be}\hspace{-1mm}>^n\hspace{-1mm}\|^2 + \| \be^{'n}_{j}\|^2 + \frac{\Delta t}{h^3\alpha_{\min}}\sum_{i=1}^J\|\be^{'n}_{i}\|^2\right)+\frac{2}{\gamma}\Delta t\sum_{n=2}^{M}\|p^{n}_{j,h}\|^2. \label{befor-inital}
		\end{align}
		Using the Cauchy-Schwarz inequality the following identity can be proven
		\begin{align}
			\left(\sum_{i=1}^{n}a_i\right)^2\le n\sum_{i=1}^na_i^2.\label{n-inequality}
		\end{align}
		Using the triangle inequality and \eqref{n-inequality}, we obtain
		\begin{align}
			\|\hspace{-1mm}<\hspace{-1mm}{\be}\hspace{-1mm}>^n\hspace{-1mm}\|^2=\Big\|\frac1J\sum_{k=1}^{J}\left(2\be_k^n-\be_{k}^{n-1}\right)\Big\|^2\le\frac{1}{J^2}\left(\sum_{k=1}^J\|2\be_k^n-\be_{k}^{n-1}\|\right)^2\le \frac1J\sum_{k=1}^{J}\|2\be_k^n-\be_{k}^{n-1}\|^2.\label{mean-bd}
		\end{align}
		Using triangle, Young's and \eqref{mean-bd} inequalities, we can write
		\begin{align}
			\|\be^{'n}_{j}\|^2=\|2\be^{n}_{j}-\be^{n-1}_{j}-<\hspace{-1mm}{\be}\hspace{-1mm}>^n\|^2&\le 2\|2\be^{n}_{j}-\be^{n-1}_{j}\|^2+2\|<\hspace{-1mm}{\be}\hspace{-1mm}>^n\|^2\nonumber\\&\le  2\|2\be^{n}_{j}-\be^{n-1}_{j}\|^2+\frac{2}{J}\sum_{k=1}^{J}\|2\be_k^n-\be_k^{n-1}\|^2.\label{eprime-short}
		\end{align}
		Using \eqref{mean-bd}-\eqref{eprime-short} together with the triangle and Young's inequalities, we reduce as below
		\begin{align}
			&\|\hspace{-1mm}<\hspace{-1mm}{\be}\hspace{-1mm}>^n\hspace{-1mm}\|^2 + \| \be^{'n}_{j}\|^2 + \frac{\Delta t}{h^3\alpha_{\min}}\sum_{i=1}^J\|\be^{'n}_{i}\|^2 \nonumber\\&\le \frac3J\sum_{k=1}^{J}\|2\be_k^n-\be_{k}^{n-1}\|^2+2\|2\be^{n}_{j}-\be^{n-1}_{j}\|^2+ \frac{\Delta t}{h^3\alpha_{\min}}\sum_{i=1}^J\left(2\|2\be^{n}_{i}-\be^{n-1}_{i}\|^2+\frac{2}{J}\sum_{k=1}^{J}\|2\be_k^n-\be_k^{n-1}\|^2\right)\nonumber\\&\le\left(\frac{3}{J}+\frac{4\Delta t}{h^3\alpha_{\min}}\right)\sum_{i=1}^{J}\|2\be_i^n-\be_i^{n-1}\|^2+2\|2\be^{n}_{j}-\be^{n-1}_{j}\|^2\nonumber\\ &\le\left(\frac{6}{J}+\frac{8\Delta t}{h^3\alpha_{\min}}\right)\sum_{i=1}^{J}\left(4\|\be_i^n\|^2+\|\be_i^{n-1}\|^2\right)+8\|\be^{n}_{j}\|^2+2\|\be^{n-1}_{j}\|^2. \label{eprime-bd}
		\end{align}
		Using \eqref{eprime-bd}, Lemma \eqref{lemma-L3-infty} and assuming that $\bu_{j,h}^0=\bhu_{j,h}^0$ and $\bu_{j,h}^1=\bhu_{j,h}^1$,  then  \eqref{befor-inital} reduces to
		\begin{align}
			&\|\be^{M}_{j}\|^2 + \gamma \Delta t \sum^{M}_{n=2} \|\nabla \cdot \be^{n}_{j,\bR}\|^2 \leq \frac{C}{\gamma}+\frac{C}{\alpha_{\min}} \Delta t \sum^{M-1}_{n=2} \left( \left(\frac{1}{J}+\frac{\Delta t}{h^3\alpha_{\min}}\right)\sum_{i=1}^J\|\be_i^n\|^2+\|\be_j^n\|^2\right).
		\end{align}
		Summing over $j=1,2, \cdots, J$, we have
		\begin{align}
			\sum^{J}_{j=1}\|\be^{M}_{j}\|^2+ \gamma \Delta t \sum^{M}_{n=2} \sum^{J}_{j=1} \|\nabla \cdot \be^{n}_{j,\bR}\|^2 \leq \frac{C}{\gamma}+\frac{C}{\alpha_{\min}} \Delta t \left(\frac{J\Delta t}{h^3\alpha_{\min}}+1\right)\sum^{M-1}_{n=2} \sum^J_{j=1} \|\be_j^n\|^2. \label{NSE-error-10-n}
		\end{align}
		Applying the discrete Gr\"onwall inequality given in Lemma \ref{dgl}, we obtain
		\begin{align}
			&\sum^{J}_{j=1}\|\be^{M}_{j}\|^2+ \gamma \Delta t \sum^{M}_{n=2} \sum^{J}_{j=1} \|\nabla \cdot \be^{n}_{j,\bR}\|^2 \leq \frac{C}{\gamma} \exp\left\{\frac{C}{\alpha_{\min}} \left(\frac{\Delta t}{h^3\alpha_{\min}}+1\right)\right\} . \label{after-gronwall-n} 
		\end{align}
		Using Lemma \ref{CR-lemma} with \eqref{after-gronwall-n} yields the following bound
		\begin{align}
			&\Delta t \sum^{M}_{n=2}\sum^{J}_{j=1} \|\nabla\be^{n}_{j,\bR}\|^2 \leq C^{2}_{R}\Delta t \sum^{M}_{n=2}\sum^{J}_{j=1} \|\nabla \cdot \be^{n}_{j,\bR}\|^2 \leq \frac{C^{2}_{R}}{\gamma^2} \exp\left\{\frac{C}{\alpha_{\min}} \left(\frac{\Delta t}{h^3\alpha_{\min}}+1\right)\right\}, \label{NSE-error-12-n}
		\end{align}
		where $C$ is independent of $\gamma$, $\Delta t$ and $h$.\\
		\textbf{Step 2:} Estimate of $\be^{n}_{j,0}:$ To find a bound on $\Delta t \sum\limits^{M}_{n=2} \sum\limits^{J}_{j=1} \|\nabla \be^{n}_{j,0}\|$, take $\bchi_h= \be^{n+1}_{j,0}$ in \eqref{NSE-error-3-n}, which yields
		\begin{align}
			&\frac{1}{2 \Delta t} \left(3 \be^{n+1}_j - 4\tilde{\be}^n_{j} + \tilde{\be}^{n-1}_{j}, \be^{n+1}_{j,0} \right) + b^* \left(<\hspace{-1mm}{\bhu}_h\hspace{-1mm}>^n, \be^{n+1}_{j,\bR} , \be^{n+1}_{j,0} \right) + b^* \left(<\hspace{-1mm}{\be}\hspace{-1mm}>^n, \bu^{n+1}_{j,h}, \be^{n+1}_{j,0}\right) \nonumber \\ & + \left(\bar{\nu}\nabla \be^{n+1}_{j,0}, \nabla \be^{n+1}_{j,0}\right) +\mu\Delta t\left((\mathcal{L}^n(\hat{u}^{'}_{h}))^2\nabla\be_j^{n+1}, \nabla\be^{n+1}_{j,0}\right)\nonumber \\ & + \mu\Delta t\left(\big\{(\mathcal{L}^n(u^{'}_{h}))^2-(\mathcal{L}^n(\hat{u}^{'}_{h}))^2\big\}\nabla\bu_{j,h}^{n+1},\nabla \be^{n+1}_{j,0}\right) = - b^* \left(\bhu^{'n}_{j,h} , 2\be^{n}_{j} - \be^{n-1}_{j} , \be^{n+1}_{j,0}\right)\nonumber \\ & - b^* \left(\be^{'n}_{j} , 2\bu^{n}_{j,h} - \bu^{n-1}_{j,h} , \be^{n+1}_{j,0}\right) - \left(\nu^{'}_j \nabla (2\be^n_{j,0} - \be^{n-1}_{j,0}) , \nabla \be^{n+1}_{j,0}\right). \label{NSE-error-13-n}
		\end{align}
		Applying the bound in \eqref{nonlinearbound}, using the stability estimate in \eqref{stability-penalty} and Young's inequality, gives
		\begin{align}
			- &b^* \Big(<\hspace{-1mm}{\bhu}_h\hspace{-1mm}>^n, \be^{n+1}_{j,\bR} , \be^{n+1}_{j,0} \Big)\le C\|\nabla\hspace{-1mm} <\hspace{-1mm}{\bhu}_h\hspace{-1mm}>^n\hspace{-1mm}\| \|\nabla\be^{n+1}_{j,\bR}\| \|\nabla\be^{n+1}_{j,0}\|\nonumber\\&\le C\zeta^{-\frac12}\Delta t^{-\frac12}\|\nabla\be^{n+1}_{j,\bR}\| \|\nabla\be^{n+1}_{j,0}\|\nonumber\\&\le\frac{\alpha_j}{10}\|\nabla \be^{n+1}_{j,0}\|^2+\frac{C}{\Delta t\alpha_j\zeta}\|\nabla\be^{n+1}_{j,\bR}\|^2.\label{step-2-bd1}
		\end{align}
		Applying the bound in 
		\eqref{nonlinearbound3}, using Lemma \ref{lemma-L3-infty} and Young's inequality, we have
		\begin{align}
			-b^* \left(<\hspace{-1mm}{\be}\hspace{-1mm}>^n, \bu^{n+1}_{j,h}, \be^{n+1}_{j,0}\right)&\le C\|\hspace{-1mm}<\hspace{-1mm}{\be}\hspace{-1mm}>^n\hspace{-1mm}\|\left(\|\nabla\bu^{n+1}_{j,h}\|_{L^3}+\|\bu^{n+1}_{j,h}\|_{L^\infty}\right)\|\nabla \be^{n+1}_{j,0}\|\nonumber\\&\le C\|\hspace{-1mm}<\hspace{-1mm}{\be}\hspace{-1mm}>^n\hspace{-1mm}\|\|\nabla \be^{n+1}_{j,0}\|\nonumber\\&\le\frac{\alpha_j}{10}\|\nabla \be^{n+1}_{j,0}\|^2+\frac{C}{\alpha_j}\|\hspace{-1mm}<\hspace{-1mm}{\be}\hspace{-1mm}>^n\hspace{-1mm}\|^2.
		\end{align}
		Decompose \begin{align}
			\left((\mathcal{L}^n(\hat{u}^{'}_{h}))^2\nabla\be_j^{n+1}, \nabla\be^{n+1}_{j,0}\right)=\left((\mathcal{L}^n(\hat{u}^{'}_{h}))^2\nabla\be_{j,0}^{n+1}, \nabla\be^{n+1}_{j,0}\right)+\left((\mathcal{L}^n(\hat{u}^{'}_{h}))^2\nabla\be_{j,\bR}^{n+1}, \nabla\be^{n+1}_{j,0}\right).
		\end{align}
		Apply the Cauchy–Schwarz and Young's inequalities, and the uniform boundedness result in Lemma \ref{uniform-boundedness-lemma-proof} (which holds for sufficiently large $\gamma$) to obtain
		\begin{align}
			-\Big((\mathcal{L}^n(\hat{u}^{'}_{h}))^2\nabla\be_{j,\bR}^{n+1},\nabla\be_{j,0}^{n+1}\Big)&\le\|\mathcal{L}^n(\hat{u}^{'}_{h})\nabla\be_{j,\bR}^{n+1}\|\|\mathcal{L}^n(\hat{u}^{'}_{h})\nabla\be_{j,0}^{n+1}\|\nonumber\\
			&\le\|\mathcal{L}^n(\hat{u}^{'}_{h})\nabla\be_{j,\bR}^{n+1}\|^2+\frac{1}{4}\|\mathcal{L}^n(\hat{u}^{'}_{h})\nabla\be_{j,0}^{n+1}\|^2\nonumber\\
			&\le\|\mathcal{L}^n(\hat{u}^{'}_{h})\|_{L^\infty}^2\|\nabla\be_{j,\bR}^{n+1}\|^2+\frac{1}{4}\|\mathcal{L}^n(\hat{u}^{'}_{h})\nabla\be_{j,0}^{n+1}\|^2\nonumber\\
			&\le C\|\nabla\be_{j,\bR}^{n+1}\|^2+\frac{1}{4}\|\mathcal{L}^n(\hat{u}^{'}_{h})\nabla\be_{j,0}^{n+1}\|^2.
		\end{align}
		Apply H\"older’s inequality, \eqref{prandtl-diff}, the stability estimate for Algorithm \ref{coupled-alg-com}, and Young's inequalities to get
		\begin{align}
			-\mu\Delta t\Big(\big\{(\mathcal{L}^n(u^{'}_{h}))^2&-(\mathcal{L}^n(\hat{u}^{'}_{h}))^2\big\}\nabla\bu_{j,h}^{n+1},\nabla\be_{j,0}^{n+1}\Big)\nonumber\\&\le\mu\Delta t\|(\mathcal{L}^n(u^{'}_{h}))^2-(\mathcal{L}^n(\hat{u}^{'}_{h}))^2\|_{L^\infty}\|\nabla\bu_{j,h}^{n+1}\|\|\nabla\be_{j,0}^{n+1}\|\nonumber\\&\le Ch^{-\frac32}\Delta t \|\nabla\bu_{j,h}^{n+1}\|\sum_{i=1}^J \|\be^{'n}_{i}\|\|\nabla\be_{j,0}^{n+1}\|\nonumber\\&\le Ch^{-\frac32}\Delta t^{\frac12}\alpha_{j}^{-\frac12}\sum_{i=1}^J \|\be^{'n}_{i}\|\|\nabla\be_{j,0}^{n+1}\|\nonumber\\& \le\frac{\alpha_j}{10}\|\nabla\be_{j,0}^{n+1}\|^2+\frac{C\Delta t}{h^3\alpha_j^2}\sum_{i=1}^J\|\be_i^{'n}\|^2.
		\end{align}
		Using identity \eqref{trilinear-identitiy}, the Cauchy-Schwarz, H\"older's, and triangle inequalities to estimate the trilinear form, we obtain
		\begin{align}
			&- b^* \left( \bhu^{'n}_{j,h} , 2\be^n_{j} - \be^{n-1}_{j} , \be^{n+1}_{j,0} \right)
			= b^* \left( \bhu^{'n}_{j,h} , \be^{n+1}_{j,0}, 2\be^n_{j} - \be^{n-1}_{j} \right) \notag \\
			&= \left( \bhu^{'n}_{j,h} \cdot \nabla \be^{n+1}_{j,0}, (2\be^n_{j} - \be^{n-1}_{j}) \right) + \frac{1}{2} \left((\nabla \cdot \bhu^{'n}_{j,h}) \be^{n+1}_{j,0}, 2\be^n_{j} - \be^{n-1}_{j} \right) \notag \\
			&\leq \| \bhu^{'n}_{j,h} \cdot \nabla \be^{n+1}_{j,0}\| \|2 \be^n_{j}- \be^{n-1}_{j} \| + \frac{1}{2}\| \nabla \cdot \bhu^{'n}_{j,h}\|_{L^\infty} \| \be^{n+1}_{j,0} \| \|2 \be^n_{j}- \be^{n-1}_{j} \| \notag \\
			& \le \left(2\|\be^{n}_{j} \| +  \|\be^{n-1}_{j} \|\right)\left(\| \bhu^{'n}_{j,h} \cdot \nabla \be^{n+1}_{j,0} \| + \frac12 \| \nabla\cdot \bhu^{'n}_{j,h}\|_{L^\infty} \| \nabla \be^{n+1}_{j,0} \|\right). \notag
		\end{align}
		Again, using \eqref{basic-ineq} and Young's inequality, gives\begin{align} 
			&- b^* \left( \bhu^{'n}_{j,h} , 2\be^n_{j} - \be^{n-1}_{j} , \be^{n+1}_{j,0} \right)\le C\left(\|\be^{n}_{j} \| +  \| \be^{n-1}_{j} \|\right)\left(\|| \bhu^{'n}_{j,h}| \nabla \be^{n+1}_{j,0} \| + \|\nabla \cdot \bhu^{'n}_{j,h}\|_{L^\infty} \| \nabla \be^{n+1}_{j,0} \|\right) \nonumber \\&\le \frac{\alpha_j}{10}\| \nabla \be^{n+1}_{j,0} \|^{2} + C\left(\|\be^{n}_{j} \| +  \|\be^{n-1}_{j} \|\right)\|\mathcal{L}^n(\hat{u}^{'}_{h}) \nabla \be^{n+1}_{j,0} \| + \frac{C}{\alpha_j}\left(\| \be^{n}_{j} \|^{2} +  \|\be^{n-1}_{j} \|^{2}\right) \|\nabla \cdot \bhu^{'n}_{j,h}\|^{2}_{L^\infty} \nonumber\\
			& \le \frac{\alpha_j}{10}\| \nabla \be^{n+1}_{j,0} \|^{2} + \frac{\mu \Delta t}{4} \|\mathcal{L}^n(\hat{u}^{'}_{h}) \nabla \be^{n+1}_{j,0} \|^{2} + \left(\frac{C}{\Delta t} + \frac{C}{\alpha_j} \|\nabla \cdot \bhu^{'n}_{j,h}\|^{2}_{L^\infty}\right) \left(\|\be^{n}_{j} \|^{2} +  \| \be^{n-1}_{j} \|^{2}\right).\label{nonlin-bd-2}
		\end{align}
		Using the nonlinear bound in \eqref{nonlinearbound3}, triangle inequality, estimate in Lemma \ref{lemma-L3-infty}, and Young's inequality provides
		\begin{align}
			-b^* \left(\be^{'n}_{j} , 2\bu^{n}_{j,h} - \bu^{n-1}_{j,h} , \be^{n+1}_{j,0}\right) &\leq C \|\be^{'n}_{j}\| \left( \|\nabla (2\bu^{n}_{j,h} - \bu^{n-1}_{j,h}) \|_{L^3} + \|2\bu^{n}_{j,h} - \bu^{n-1}_{j,h}\|_{L^\infty}\right) \|\nabla \be^{n+1}_{j,0}\| \nonumber \\
			& \leq CC_{*}\|\be^{'n}_{j}\|\|\nabla \be^{n+1}_{j,0}\| \leq \frac{\alpha_j}{10} \|\nabla \be^{n+1}_{j,0}\|^2 + \frac{C}{\alpha_j}\| \be^{'n}_{j}\|^2. 
		\end{align}
		Using H\"older's and Young's inequalities, yields as in \eqref{vis-fl},
		\begin{align}
			-\Big(\nu^{'}_{j} \nabla (2\be^n_{j,0} - \be^{n-1}_{j,0}) , \nabla \be^{n+1}_{j,0}\Big)
			\leq \frac{3}{2}\|\nu^{'}_j \|_{\infty} \| \nabla \be^{n+1}_{j,0} \|^2 + \|\nu^{'}_j\|_{\infty} \|\nabla \be^{n}_{j,0}\|^2 + \frac{1}{2} \|\nu^{'}_j\|_{\infty} \|\nabla \be^{n-1}_{j,0}\|^2.\label{step2-bd2}
		\end{align}
		Using \eqref{step-2-bd1}-\eqref{step2-bd2} into \eqref{NSE-error-13-n}, applying the identity in \eqref{BDF-2-identity} and reducing, we have\begin{align}
			&\frac{1}{2 \Delta t} \left(3 \be^{n+1}_j - 4\tilde{\be}^n_{j} + \tilde{\be}^{n-1}_{j}, \be^{n+1}_{j,0} \right)+\frac{\bar{\nu}_{\min}}{2}\|\nabla\be^{n+1}_{j,0}\|^2 +\frac{\mu \Delta t}{2} \|\mathcal{L}^n(\hat{u}^{'}_{h})\nabla \be^{n+1}_{j,0}\|^2 \nonumber \\ & \leq \frac{C}{\Delta t}\left(\frac{1}{\alpha_j\zeta}+\Delta t^2\right)\|\nabla\be^{n+1}_{j,\bR}\|^2+ \frac{C}{\alpha_{\min}}\left(\|\hspace{-1mm}<\hspace{-1mm}{\be}\hspace{-1mm}>^n\hspace{-1mm}\|^2+\| \be^{'n}_{j}\|^2+\frac{\Delta t}{h^3\alpha_{\min}}\sum_{i=1}^J\|\be_i^{'n}\|^2  \right)\nonumber \\ &+\left(\frac{C}{\Delta t} + \frac{C}{\alpha_{\min}} \|\nabla \cdot \bhu^{'n}_{j,h}\|^{2}_{L^\infty}\right) \left(\|\be^{n}_{j} \|^{2} +  \|\be^{n-1}_{j} \|^{2}\right)+\|\nu^{'}_j\|_{\infty} \|\nabla \be^{n}_{j,0}\|^2 + \frac{1}{2} \|\nu^{'}_j\|_{\infty} \|\nabla \be^{n-1}_{j,0}\|^2. \label{NSE-error-14-n}
		\end{align}
		Rewrite the time-derivative term\begin{align}
			&\frac{1}{2 \Delta t} \left(3 \be^{n+1}_j - 4\tilde{\be}^n_{j} + \tilde{\be}^{n-1}_{j}, \be^{n+1}_{j,0} \right)\nonumber\\&=\frac{1}{2 \Delta t} \left(3 \be^{n+1}_j - 4\tilde{\be}^n_{j} + \tilde{\be}^{n-1}_{j}, \be^{n+1}_{j} \right)- \frac{1}{2 \Delta t}\left(3 \be^{n+1}_j - 4\tilde{\be}^n_{j} + \tilde{\be}^{n-1}_{j}, \be^{n+1}_{j,\bR} \right).\label{rewrite-time}
		\end{align}
		Use the Cauchy-Schwarz, Poincar\'e's, triangle, $\|\tilde{\be}_j^{n}\|\le\|\be_j^{n}\|$ from \eqref{tilde-bd1} and Young's inequalities
		\begin{align}
			\frac{1}{2 \Delta t} \Big(3 \be^{n+1}_j &- 4\tilde{\be}^n_{j} + \tilde{\be}^{n-1}_{j}, \be^{n+1}_{j,\bR} \Big) \leq \frac{C}{2\Delta t}\|3\be^{n+1}_{j}-4\tilde{\be}^{n}_{j}+\tilde{\be}^{n-1}_{j}\|\|\nabla\be^{n+1}_{j,\bR}\|\nonumber \\
			& \leq \frac{C}{2\Delta t}\Big(3\|\be^{n+1}_{j}-2\tilde{\be}^{n}_{j}+\tilde{\be}^{n-1}_{j}\|+2\|\tilde{\be}^{n}_{j}-\tilde{\be}^{n-1}_{j}\|\Big)\|\nabla\be^{n+1}_{j,\bR}\| \nonumber \\
			& \leq \frac{1}{4\Delta t}\|\be^{n+1}_{j}-2\tilde{\be}^{n}_{j}+\tilde{\be}^{n-1}_{j}\|^2+\frac{C}{\Delta t}\|\nabla\be^{n+1}_{j,\bR}\|^2+\|\be^{n}_{j}\|^2+\|\be^{n-1}_{j}\|^2+ \frac{C}{\Delta t ^2}\|\nabla\be^{n+1}_{j,\bR}\|^2. \label{time-ber}
		\end{align}
		Using \eqref{rewrite-time}-\eqref{time-ber} and identity in \eqref{BDF-2-identity}  into \eqref{NSE-error-14-n}, dropping nonnegative term from left-hand-side and reducing, gives
		\begin{align}
			&\frac{1}{4 \Delta t} \left( \|\be^{n+1}_{j}\|^2 + \|2\be^{n+1}_{j} - \tilde{\be}^{n}_{j}\|^2 - \|\tilde{\be}^{n}_{j}\|^2 - \|2\tilde{\be}^{n}_{j} - \tilde{\be}^{n-1}_{j}\|^2 \right)+\frac{\bar{\nu}_{\min}}{2}\|\nabla\be^{n+1}_{j,0}\|^2 \nonumber \\ &\le\frac{C}{\Delta t} \left(\frac{1}{\alpha_{\min}\zeta}+1+\frac{1}{\Delta t}+\Delta t^2\right)\|\nabla\be^{n+1}_{j,\bR}\|^2+ \frac{C}{\alpha_{\min}}\left(\|\hspace{-1mm}<\hspace{-1mm}{\be}\hspace{-1mm}>^n\hspace{-1mm}\|^2+\| \be^{'n}_{j}\|^2+\frac{\Delta t}{h^3\alpha_{\min}}\sum_{i=1}^J\|\be_i^{'n}\|^2  \right)\nonumber\\&+\left(\frac{C}{\Delta t} + \frac{C}{\alpha_{\min}} \|\nabla \cdot \bhu^{'n}_{j,h}\|^{2}_{L^\infty}+1\right) \left(\|\be^{n}_{j} \|^{2} +  \|\be^{n-1}_{j} \|^{2}\right) +\|\nu^{'}_j\|_{\infty} \|\nabla \be^{n}_{j,0}\|^2 + \frac{1}{2} \|\nu^{'}_j\|_{\infty} \|\nabla \be^{n-1}_{j,0}\|^2.
		\end{align}
		Use \eqref{tilde-bd2} and \eqref{eprime-bd}, and rearranging
		\begin{align}
			&\frac{1}{4 \Delta t} \left( \|\be^{n+1}_{j}\|^2 + \|2\be^{n+1}_{j} - \tilde{\be}^{n}_{j}\|^2 - \|\be^{n}_{j}\|^2 - \|2\be^{n}_{j} - \tilde{\be}^{n-1}_{j}\|^2 \right)+\frac{\bar{\nu}_{\min}}{2}\left(\|\nabla\be^{n+1}_{j,0}\|^2-\|\nabla \be^{n}_{j,0}\|^2\right) \nonumber \\ &+\left(\frac{\bar{\nu}_{\min}}{2}-\|\nu^{'}_j\|_{\infty}\right)\left(\|\nabla \be^{n}_{j,0}\|^2-\|\nabla \be^{n-1}_{j,0}\|^2\right)+\frac{\alpha_{\min}}{2}\|\nabla \be^{n-1}_{j,0}\|^2\nonumber \\ &\le\frac{C}{\Delta t} \left( \frac{1}{\alpha_{\min}\zeta}+1+\frac{1}{\Delta t}+\Delta t^2\right)\|\nabla\be^{n+1}_{j,\bR}\|^2+\left(\frac{C}{\Delta t} + \frac{C}{\alpha_{\min}} \|\nabla \cdot \bhu^{'n}_{j,h}\|^{2}_{L^\infty}+1\right) \left(\|\be^{n}_{j} \|^{2} +  \|\be^{n-1}_{j} \|^{2}\right)\nonumber\\&+ \frac{C}{\alpha_{\min}}\left\{ \left(\frac{1}{J}+\frac{\Delta t}{h^3\alpha_{\min}}\right) \sum^{J}_{i=1}\left(\|\be^{n}_{i}\|^2+\|\be^{n-1}_{i}\|^2\right)+\|\be^{n}_{j}\|^2+\|\be^{n-1}_{j}\|^2  \right\}.
		\end{align}
		Multiplying both sides by $4\Delta t$, summing over time step  $n=1, \cdots, M-1$, and reducing, we obtain
		\begin{align}
			&\|\be^{M}_{j}\|^2+\|2\be^{M}_{j}-\tilde{\be}^{M-1}_{j}\|^2+2\bar{\nu}_{\min} \Delta t\|\nabla\be^{M}_{j,0}\|^2+4\Delta t\left(\frac{\bar{\nu}_{\min}}{2}-\|\nu^{'}_{j}\|_\infty\right)\|\nabla\be^{M-1}_{j,0}\|^2\nonumber \\
			&  +2\alpha_{\min} \Delta t\sum_{n=1}^{M-1}\|\nabla\be^{n-1}_{j,0}\|^2 \leq \|\be^{1}_{j}\|^2+\|2\be^{1}_{j}-\tilde{\be}^{0}_{j}\|^2+4\Delta t\left(\frac{\bar{\nu}_{\min}}{2}-\|\nu^{'}_{j}\|_\infty\right)\|\nabla\be^{0}_{j,0}\|^2+2\bar{\nu}_{\min} \Delta t\|\nabla\be^{0}_{j,0}\|^2\nonumber \\&+C \left(\frac{1}{\alpha_{\min}\zeta}+1+\frac{1}{\Delta t}+\Delta t^2\right)\sum_{n=1}^{M-1}\|\nabla\be^{n+1}_{j,\bR}\|^2+ \frac{C\Delta t}{\alpha_{\min}} \left(\frac{1}{J}+\frac{\Delta t}{h^3\alpha_{\min}}\right) \sum^{J}_{i=1}\sum_{n=0}^{M-1}\|\be^{n}_{i}\|^2\nonumber\\&+C\Delta t\left(\frac{1}{\alpha_{\min}}+\frac{1}{\Delta t} + \frac{D_{\infty}^2}{\alpha_{\min}}+1\right)\sum_{n=0}^{M-1} \|\be^{n}_{j} \|^{2}.
		\end{align}
		Dropping nonnegative terms from left-hand-side, and reducing yields
		\begin{align}
			&\|\be^{M}_{j}\|^2 +2\alpha_{\min} \Delta t\sum_{n=2}^{M}\|\nabla\be^{n}_{j,0}\|^2 \leq C  \left(\frac{1}{\alpha_{\min}\zeta}+1+\frac{1}{\Delta t}+\Delta t^2\right)\sum_{n=2}^{M}\|\nabla\be^{n}_{j,\bR}\|^2\nonumber \\&+ \frac{C\Delta t}{\alpha_{\min}} \left(\frac{1}{J}+\frac{\Delta t}{h^3\alpha_{\min}}\right) \sum^{J}_{i=1}\sum_{n=2}^{M-1}\|\be^{n}_{i}\|^2+C\Delta t\left(\frac{1}{\alpha_{\min}}+\frac{1}{\Delta t} + \frac{D_{\infty}^2}{\alpha_{\min}}+1\right)\sum_{n=2}^{M-1} \|\be^{n}_{j} \|^{2}.
		\end{align}
		Summing over $j=1, \cdots, J$, and reducing, we have
		\begin{align}
			&\sum_{j=1}^J\|\be^{M}_{j}\|^2 +2\alpha_{\min} \Delta t\sum_{j=1}^J\sum_{n=2}^{M}\|\nabla\be^{n}_{j,0}\|^2\leq C  \left(\frac{1}{\alpha_{\min}\zeta}+1+\frac{1}{\Delta t}+\Delta t^2\right)\sum_{j=1}^J\sum_{n=2}^{M}\|\nabla\be^{n}_{j,\bR}\|^2\nonumber\\
			&+ \frac{C\Delta t}{\alpha_{\min}} \left(\frac{\Delta t}{h^3\alpha_{\min}}+1+\frac{\alpha_{\min}}{\Delta t}+\alpha_{\min}+D_{\infty}^2\right) \sum^{J}_{j=1}\sum_{n=2}^{M-1}\|\be^{n}_{j}\|^2.
		\end{align}
		We now apply the discrete Gr\"onwall inequality given in Lemma \ref{dgl} to obtain
		\begin{align}
			\sum_{j=1}^J\|\be^{M}_{j}\|^2 +2\alpha_{\min} \Delta t\sum_{j=1}^J\sum_{n=2}^{M}\|\nabla\be^{n}_{j,0}\|^2 &\leq C\exp\left\{ \frac{C}{\alpha_{\min}} \left(\frac{\Delta t}{h^3\alpha_{\min}}+1+\frac{\alpha_{\min}}{\Delta t}+\alpha_{\min}+D_{\infty}^2\right)\right\}\nonumber\\&\times \Delta t \left( \frac{1}{\Delta t\alpha_{\min}\zeta}+\frac{1}{\Delta t}+\frac{1}{\Delta t^2}+\Delta t\right)\sum_{j=1}^J\sum_{n=2}^{M}\|\nabla\be^{n}_{j,\bR}\|^2.
		\end{align}
		Using the estimate in \eqref{NSE-error-12-n} and simplifying, we obtain
		\begin{align}
			&\sum_{j=1}^J\|\be^{M}_{j}\|^2 +2\alpha_{\min} \Delta t\sum_{j=1}^J\sum_{n=2}^{M}\|\nabla\be^{n}_{j,0}\|^2\leq \frac{CC_R^2}{\gamma^2}\left( \frac{1}{\Delta t\alpha_{\min}\zeta}+\frac{1}{\Delta t}+\frac{1}{\Delta t^2}+\Delta t\right)\nonumber\\&\times\exp\left\{ \frac{C}{\alpha_{\min}} \left(\frac{\Delta t}{h^3\alpha_{\min}}+1+\frac{\alpha_{\min}}{\Delta t}+\alpha_{\min}+D_{\infty}^2\right)\right\}.
		\end{align}
		Finally, applying the triangle and Young's inequalities to $\|\nabla\hspace{-1mm}\lab\bu_h\rab^n-\nabla\hspace{-1mm}\lab\bhu_h\rab^n\hspace{-1mm}\|^2$ gives the desired velocity estimate.
		
		\textbf{Step 3:} Now we begin the pressure estimate by subtracting \eqref{NSE-error-1} from \eqref{couple-eqn-1-new} yields for all $\bchi_{h} \in \bX_h$
		\begin{align}
			&\frac{1}{2 \Delta t} \left(3 \be^{n+1}_j - 4\be^n_{j} + \be^{n-1}_{j}, \bchi_{h} \right) + b^* \left(<{\bhu}_h>^n, \be^{n+1}_j , \bchi_{h} \right) + b^* \left(<{\be}>^n, \bu^{n+1}_{j,h}, \bchi_{h}\right)  \nonumber \\ &+ \left(\bar{\nu}\nabla \be^{n+1}_{j}, \nabla \bchi_{h}\right) + \gamma \left(\nabla \cdot \be^{n+1}_{j,\bR}, \nabla \cdot\bchi_{h}\right) - \left(p^{n+1}_{j,h} - \frac{4}{3}\hat{p}^n_{j,h} + \frac{1}{3}\hat{p}^{n-1}_{j,h}, \nabla \cdot 
			\bchi_{h} \right)  \nonumber \\ &+ \left(\nu_T(\hat{u}_{h}^{'},t_n)\nabla \be_j^{n+1},\nabla\bchi_h\right)+ \mu\Delta t\left(\big\{(\mathcal{L}^n(u^{'}_{h}))^2-(\mathcal{L}^n(\hat{u}^{'}_{h}))^2\big\}\nabla\bu_{j,h}^{n+1},\nabla\bchi_h\right) \nonumber \\ &= - b^* \left(\bhu^{'n}_{j,h} , 2\be^{n}_{j} - \be^{n-1}_{j} , \bchi_{h}\right) - b^* \left(\be^{'n}_{j} , 2\bu^{n}_{j,h} - \bu^{n-1}_{j,h} , \bchi_{h}\right) - \left(\nu^{'}_j \nabla (2\be^n_{j} - \be^{n-1}_{j}) , \nabla \bchi_{h}\right).\label{p1}
		\end{align}
		Using $\nabla \cdot \be_{j,\bR} = \nabla \cdot \be_{j} = - \nabla \cdot \bhu_{j,h}$, we rewrite \eqref{p1} as
		\begin{align}
			&\left(p^{n+1}_{j,h} -\left( \frac{4}{3}\hat{p}^n_{j,h} - \frac{1}{3}\hat{p}^{n-1}_{j,h}-\gamma\nabla \cdot \bhu^{n+1}_{j,h}\right), \nabla \cdot 
			\bchi_{h} \right)=\frac{1}{2 \Delta t} \left(3 \be^{n+1}_j - 4\be^n_{j} + \be^{n-1}_{j}, \bchi_{h}\right)\nonumber \\&+ b^* \left(<{\bhu}_h>^n, \be^{n+1}_j, \bchi_{h} \right) + b^* \left(<{\be}>^n, \bu^{n+1}_{j,h}, \bchi_{h}\right) + \left(\bar{\nu}\nabla \be^{n+1}_{j}, \nabla \bchi_{h}\right) \nonumber \\& + \left(\nu_T(\hat{u}_{h}^{'},t_n)\nabla \be_j^{n+1},\nabla\bchi_{h}\right)+ \mu\Delta t\left(\big\{(\mathcal{L}^n(u^{'}_{h}))^2-(\mathcal{L}^n(\hat{u}^{'}_{h}))^2\big\}\nabla\bu_{j,h}^{n+1},\nabla\bchi_{h}\right) \nonumber \\&+ b^* \left(\bhu^{'n}_{j,h}, 2\be^{n}_{j} - \be^{n-1}_{j}, \bchi_{h}\right) + b^* \left(\be^{'n}_{j}, 2\bu^{n}_{j,h} - \bu^{n-1}_{j,h}, \bchi_{h}\right) + \left(\nu^{'}_j \nabla (2\be^n_{j} - \be^{n-1}_{j}), \nabla \bchi_{h}\right).\label{p2}
		\end{align} 
		Using Cauchy-Schwarz, triangle and  Poincar\'e inequalities
		\begin{align}
			\frac{1}{2 \Delta t} \left(3 \be^{n+1}_j - 4\be^n_{j} + \be^{n-1}_{j}, \bchi_{h}\right)&\le \frac{1}{2 \Delta t}\left(3\|\be^{n+1}_j\|+4\|\be^{n}_j\|+\|\be^{n-1}_j\|\right)\|\bchi_h\|\nonumber\\& \le\frac{C}{\Delta t}\left(\|\nabla\be^{n+1}_j\|+\|\nabla\be^{n}_j\|+\|\nabla\be^{n-1}_j\|\right)\|\nabla\bchi_h\|.\label{step3-cauchy}
		\end{align}
		Using the bound in \eqref{nonlinearbound1} and stability estimate, gives
		\begin{align}
			b^* \left(<{\bhu}_h>^n, \be^{n+1}_j, \bchi_{h} \right)&\le C \|<{\bhu}_h>^n\|^{\frac12}\|\nabla <{\bhu}_h>^n\|^{\frac12}\|\nabla\be^{n+1}_j\|\|\nabla \bchi_{h}\|\nonumber\\&\le C\alpha_{\min}^{-\frac12}\zeta^{-\frac14}\Delta t^{-\frac14}\|\nabla\be^{n+1}_j\|\|\nabla \bchi_{h}\|.
		\end{align}
		Using the bound in \eqref{nonlinearbound3}, Lemma \ref{lemma-L3-infty}, triangle and Poincar\'e inequalities, yields
		\begin{align}
			b^* \left(<{\be}>^n, \bu^{n+1}_{j,h}, \bchi_{h}\right)&\le C\|<{\be}>^n\|\left(\|\nabla  \bu^{n+1}_{j,h}\|_{L^3}+\| \bu^{n+1}_{j,h}\|_{L^\infty}\right)\|\nabla \bchi_{h}\|\nonumber\\&\le C\|<{\be}>^n\|\|\nabla \bchi_{h}\|\nonumber\\&\le C\|\nabla \bchi_{h}\|\sum_{i=1}^J\left( \|\nabla\be^{n}_{i}\|+\|\nabla\be^{n-1}_{i}\|\right).
		\end{align}
		Applying the H\"older's inequality, yields\begin{align}
			\left(\bar{\nu}\nabla \be^{n+1}_{j}, \nabla \bchi_{h}\right)\le\|\bar{\nu}\|_{\infty}\|\nabla \be^{n+1}_{j}\|\|\nabla \bchi_{h}\|.
		\end{align}
		Using the H\"older's and triangle inequalities, we obtain
		\begin{align}
			\left(\nu^{'}_j \nabla (2\be^n_{j} - \be^{n-1}_{j}), \nabla \bchi_{h}\right)\le\|\nu_{j}^{'}\|_{\infty}\left(\|\nabla \be^n_{j}\| +\nabla\|\be^{n-1}_{j}\|\right)\|\nabla\bchi_h\|.
		\end{align}
		Apply the H\"older's inequality and \eqref{uniform-boundedness-lemma-proof}, gives
		\begin{align}
			\left(\nu_T(\hat{u}_{h}^{'},t_n)\nabla \be_j^{n+1},\nabla\bchi_{h}\right)\le C\|\nabla \be_j^{n+1}\|\|\nabla\bchi_{h}\|.
		\end{align}
		Apply the H\"older's inequality,  \eqref{prandtl-diff}, stability estimate of Algorithm \ref{coupled-alg-com} in \eqref{stability-bdf2-statement}, triangle and Poincar\'e inequalities, yields
		\begin{align}
			\mu\Delta t\Big(\big\{(\mathcal{L}^n(u^{'}_{h}))^2-(\mathcal{L}^n(\hat{u}^{'}_{h}))^2\big\}\nabla\bu_{j,h}^{n+1},&\nabla\bchi_h\Big)
			\le C\alpha_{\min}^{-\frac12}h^{-\frac32}\Delta t^{\frac12}\sum_{i=1}^J \|\be^{'n}_{i}\|\|\nabla\bchi_{h}\|\nonumber\\
			&\le C\alpha_{\min}^{-\frac12}h^{-\frac32}\Delta t^{\frac12}\|\|\nabla\bchi_{h}\|\sum_{i=1}^J\left( \|\nabla\be^{n}_{i}\|+\|\nabla\be^{n-1}_{i}\|\right).
		\end{align}
		Using \eqref{nonlinearbound}, triangle inequality and stability estimate in \eqref{stability-penalty}, gives
		\begin{align}
			b^* \left(\bhu^{'n}_{j,h}, 2\be^{n}_{j} - \be^{n-1}_{j}, \bchi_{h}\right)
			\le C\zeta^{-\frac12}\Delta t^{-\frac12}\left(\|\nabla \be^{n}_{j} \|+\|\nabla \be^{n-1}_{j}\|\right)\|\nabla\bchi_h\|.
		\end{align}
		Applying the nonlinear bound in \eqref{nonlinearbound3}, the triangle inequality, Lemma \ref{lemma-L3-infty}, and the Poincar\'e inequality, we have
		\begin{align}
			-b^* \left(\be^{'n}_{j} , 2\bu^{n}_{j,h} - \bu^{n-1}_{j,h} , \bchi_{h}\right) 
			& \leq C\|\be^{'n}_{j}\|\|\nabla \bchi_{h}\|\le C\|\nabla\bchi_h\|\sum_{i=1}^J\left( \|\nabla\be^{n}_{i}\|+\|\nabla\be^{n-1}_{i}\|\right). \label{step-3-nonbd}
		\end{align}
		Substituting the bounds in \eqref{step3-cauchy}-\eqref{step-3-nonbd} in \eqref{p2} and using the triangle and Poincar\'e inequalities yields
		\begin{align}
			&\frac{	\left(p^{n+1}_{j,h} -\left( \frac{4}{3}\hat{p}^n_{j,h} - \frac{1}{3}\hat{p}^{n-1}_{j,h}-\gamma\nabla \cdot \bhu^{n+1}_{j,h}\right), \nabla \cdot 
				\bchi_{h} \right)}{\|\nabla\bchi_h\|} \nonumber\\&\le C\left(\alpha_{\min}^{-\frac12}\zeta^{-\frac14}\Delta t^{-\frac14}+\|\bar{\nu}\|_{\infty}+\Delta t+\Delta t^{-1}\right)\|\nabla \be_j^{n+1}\|+C\left(1+\alpha_{\min}^{-\frac12}h^{-\frac32}\Delta t^{\frac12}\right)\sum_{i=1}^J \left(\|\nabla\be_i^{n}\|+\|\nabla\be_i^{n-1}\|\right)\nonumber\\&+C\left(\zeta^{-\frac12}\Delta t^{-\frac12}+1+\|\nu_{j}^{'}\|_{\infty}+\Delta t^{-1}\right)\left(\|\nabla\be^{n}_{j}\| +\|\nabla \be^{n-1}_{j}\|\right).
		\end{align}
		Applying the \textit{inf-sup} condition gives
		\begin{align}
			\beta\Big\|p^{n+1}_{j,h} -&\left( \frac{4}{3}\hat{p}^n_{j,h} - \frac{1}{3}\hat{p}^{n-1}_{j,h}-\gamma\nabla \cdot \bhu^{n+1}_{j,h}\right)\Big\|\le C\left(\alpha_{\min}^{-\frac12}\zeta^{-\frac14}\Delta t^{-\frac14}+\|\bar{\nu}\|_{\infty}+\Delta t+\Delta t^{-1}\right)\|\nabla \be_j^{n+1}\|\nonumber\\&+C\left(1+\alpha_{\min}^{-\frac12}h^{-\frac32}\Delta t^{\frac12}\right)\sum_{i=1}^J \left(\|\nabla\be_i^{n}\|+\|\nabla\be_i^{n-1}\|\right)\nonumber\\&+C\left(\zeta^{-\frac12}\Delta t^{-\frac12}+1+\|\nu_{j}^{'}\|_{\infty}+\Delta t^{-1}\right)\left(\|\nabla\be^{n}_{j}\| +\|\nabla \be^{n-1}_{j}\|\right) .
		\end{align}
		Squaring both sides, multiplying both sides by $\Delta t$, summing over time step $n=1, 2, \cdots, M-1$, and reducing yields
		\begin{align}
			\Delta t\sum_{n=1}^{M-1}\Big\|p^{n+1}_{j,h} -&\left( \frac{4}{3}\hat{p}^n_{j,h} - \frac{1}{3}\hat{p}^{n-1}_{j,h}-\gamma\nabla \cdot \bhu^{n+1}_{j,h}\right)\Big\|^2\nonumber\\&\le C\left(\alpha_{\min}^{-1}\zeta^{-\frac12}\Delta t^{-\frac12}+\|\bar{\nu}\|_{\infty}^2+\Delta t^2+\Delta t^{-2}\right)\Delta t\sum_{n=1}^{M-1}\|\nabla \be_j^{n+1}\|^2\nonumber\\&+C\left(1+\alpha_{\min}^{-1}h^{-3}\Delta t\right)\Delta t\sum_{n=1}^{M-1}\sum_{i=1}^J \left(\|\nabla\be_i^{n}\|^2+\|\nabla\be_i^{n-1}\|^2\right)\nonumber\\&+C\left(\zeta^{-1}\Delta t^{-1}+1+\|\nu_{j}^{'}\|_{\infty}^2+\Delta t^{-2}\right)\Delta t\sum_{n=1}^{M-1}\left(\|\nabla\be^{n}_{j}\|^2 +\|\nabla \be^{n-1}_{j}\|^2\right). 
		\end{align}
		Summing over $j=1,2,\cdots,J$ and simplifying yields
		\begin{align}
			&\Delta t\sum_{j=1}^{J}\sum_{n=1}^{M-1}\Big\|p^{n+1}_{j,h} -\left( \frac{4}{3}\hat{p}^n_{j,h} - \frac{1}{3}\hat{p}^{n-1}_{j,h}-\gamma\nabla \cdot \bhu^{n+1}_{j,h}\right)\Big\|^2\le C\Delta t\sum_{j=1}^J\sum_{n=1}^{M}\|\nabla\be^{n}_{j}\|^2\nonumber\\&
			\times\left(\zeta^{-1}\Delta t^{-1}+1+\max_{1\le j\le J}\|\nu_{j}^{'}\|_{\infty}^2+\Delta t^{-2}+\alpha_{\min}^{-1}h^{-3}\Delta t+\Delta t^2+\alpha_{\min}^{-1}\zeta^{-\frac12}\Delta t^{-\frac12}+\|\bar{\nu}\|_{\infty}^2\right).
		\end{align}
		Substituting the velocity error estimate completes the proof.
	\end{proof}

Therefore, \begin{align}
		\|<\bu>-<\bhu_h>\|_{2,1}\le C(h^k+\Delta t^2)+\frac{C(h,\Delta t)}{\gamma}.
	\end{align}
	\begin{remark}
		If $\gamma$ is chosen sufficiently large so that the penalty-projection error is of order $O(h^k+\Delta t^2)$, then Algorithm \ref{NSE-FEM} inherits the second-order temporal and optimal spatial accuracy of Algorithm \ref{coupled-alg-com}.
	\end{remark}
	\section{Numerical experiments}\label{numerical-experiment} In this section, we present a series of numerical tests to validate the predicted convergence rates and demonstrate the performance of the projection method in Algorithm \ref{NSE-FEM} on various two- and three-dimensional benchmark problems.  We write the spatial variable $\bx=(x_1,x_2)$ for 2D problems and $\bx=(x_1,x_2,x_3)$ for 3D problems. All simulations are conducted using the deal.II platform \cite{dealII93}, a C++ finite element library. The computations use the $(\mathbb{Q}_2^d,\mathbb{Q}_1)$ TH element, the direct solver UMFPACK \cite{davis2004algorithm}, and $\mu=1$ unless otherwise stated.
	
	\subsection{Verification of Convergence Rates}
	In the validation experiments, we verify the theoretically derived convergence rates using the following 2D analytical solution:
	\[ {\bu}=\left(\begin{array}{c} \cos x_2+(1+e^t)\sin x_2 \\ \sin x_1+(1+e^t)\cos x_1 \end{array} \right)\;\;\text{and}\;\; \ p =\sin(x_1+x_2)(1+e^t),
	\]
	on the domain $\cD = [0,1]^2$. Next, we introduce noise as follows: $\bu_j = (1 + k_j \epsilon) \bu$ and $p_j = (1 + k_j \epsilon) p$, where $\epsilon$ is the perturbation parameter and $k_j := (-1)^{j+1} 4\lceil j/2 \rceil / J$, with $j = 1, 2, \dots, J$, and $J=20$ unless stated otherwise. The analytical solution is clearly divergence-free. We examine two values of $\epsilon$: $10^{-3}$ and $10^{-2}$, which introduce noise into the initial and boundary conditions, as well as the forcing term. The forcing term $\bif_j$ is computed using the synthetic data in \eqref{gov1}. We model the viscosity $\nu$ as a continuous uniform random variable and consider three random samples of size $J$ with $\bE[\nu] = 10^{-3}, 10^{-4}$, and $10^{-5}$, where the samples are generated with a 10$\%$ variation from the mean. The boundary conditions are prescribed by $\bu_{j,h}|{\partial\cD} = \bu_j$, and the initial conditions are given by $\bu_{j,h}^0 = \bu_j(0,\bx)$ and $\bu_{j,h}^1 = \bu_j(\Delta t,\bx)$. The mean-velocity error is defined as $<\hspace{-1mm}\be\hspace{-1mm}> := <\hspace{-1mm}\bu\hspace{-1mm}> - <\hspace{-1mm}\bu_h\hspace{-1mm}>$ and $\|\hspace{-1mm}<\hspace{-1mm}\be\hspace{-1mm}>\hspace{-1mm}\|_{2,1}:=\sqrt{\Delta t\sum\limits_{n=2}^M\|\nabla<\hspace{-1mm}\be\hspace{-1mm}>^n\|^2}$. Structured quadrilateral meshes are considered.
	
	\subsubsection{Convergence of the Proposed BDF-2-EEV-SPP Scheme to the  BDF-2-EEV-Coupled Scheme as $\gamma \to \infty$} In this section, we compute the rate at which the proposed penalty-projection method in Algorithm \ref{NSE-FEM} converges to the coupled method in Algorithm \ref{coupled-alg-com}. The error is defined as $<\hspace{-1mm}\hat{\be}\hspace{-1mm}> := <\hspace{-1mm}\bu_{h}\hspace{-1mm}> - <\hspace{-1mm}\bhu_{h}\hspace{-1mm}>$, which represents the difference between the results obtained from the BDF-2-EEV-Coupled and BDF-2-EEV-SPP schemes. To this end, we set the end time $T = 1$, the time-step size $\Delta t = T/10$, $\epsilon = 0.001$, and $h = 1/32$. We first set $\gamma = 0$ and then increase $\gamma$, beginning with $10^{-2}$, by successive factors of 10, record the velocity errors and compute the corresponding convergence rates. These rates are summarized in Table \ref{conv-table}. We observe that, as $\gamma$ increases, the convergence rates asymptotically approach 1, which aligns well with the theoretically predicted convergence rates in terms of $\gamma$ given in \eqref{gamma-theorem}.
	
	\begin{table}[h!]
		\centering
		\begin{tabular}{|c|c|c|c|c|c|c|}
			\hline
			& \multicolumn{6}{c|}{Fixed $T = 1$, $\Delta t = T/10$, $h = 1/32$} \\
			\hline
			$\epsilon = 10^{-3}$
			& \multicolumn{2}{c|}{$\bE[\nu] = 10^{-3}$} & \multicolumn{2}{c|}{$\bE[\nu] = 10^{-4}$} & \multicolumn{2}{c|}{$\bE[\nu] = 10^{-5}$} \\
			\hline
			$\gamma$ & $\|\hspace{-1mm}< \hspace{-1mm}\hat{\be} \hspace{-1mm}>\hspace{-1mm} \|_{2,1}$ & rate & $\|\hspace{-1mm}<\hspace{-1mm} \hat{\be} \hspace{-1mm}> \hspace{-1mm}\|_{2,1}$ & rate & $\|\hspace{-1mm}<\hspace{-1mm} \hat{\be}\hspace{-1mm} > \hspace{-1mm}\|_{2,1}$ & rate \\
			\hline
			$0$   & 7.1830e-0 & --   & 4.5918e+1 & --   & 1.4696e+2 & --   \\
			\hline
			1e-2   & 4.0926e-0 & -- & 1.1894e+1 & -- &  2.0956e+1 & -- \\
			\hline
			1e-1   & 2.9609e-0 & 0.14 & 6.2126e-0 & 0.28 & 7.5358e-0 & 0.44 \\
			\hline
			1e+0  & 1.1014e-0 & 0.43 & 1.7004e-0 & 0.56 & 1.8802e-0 & 0.60 \\
			\hline
			1e+1  & 1.8862e-1 & 0.77 & 2.1280e-1 & 0.90 & 2.1632e-1 & 0.94 \\
			\hline
			1e+2  & 2.0069e-2 & 0.97 & 2.4048e-2 & 0.95 & 2.4932e-2 & 0.94 \\
			\hline
		\end{tabular}
		\caption{The BDF-2-EEV-SPP scheme converges to the BDF-2-EEV-Coupled scheme as $\gamma$ increases.} \label{conv-table}
	\end{table}
	\subsubsection{Spatial Convergence} To examine spatial convergence, we keep the temporal error sufficiently small by setting a very short simulation end time of $T = 0.001$. We consider $\epsilon = 10^{-2}$. The mesh width $h$ is successively refined by a factor of 1/2, and simulations are run for each refinement. The errors and convergence rates are recorded in Table \ref{sp-table1} for the BDF-2-EEV-SPP scheme. We observe second-order spatial convergence rates for our scheme in all three random-viscosity samples. These results demonstrate optimal second-order spatial convergence for the $(\mathbb{Q}_2^2, \mathbb{Q}_1)$ element and are consistent with the theoretically predicted spatial accuracy.

	\begin{table}[h!]
		\centering
		\begin{tabular}{|c|c|c|c|c|c|c|}
			\hline
			& \multicolumn{6}{c|}{Spatial convergence (fixed $T = 0.001$, $\Delta t = T/8$)} \\
			\hline
			$\epsilon = 10^{-2}$
			& \multicolumn{2}{c|}{$\bE[\nu] = 10^{-3}$} & \multicolumn{2}{c|}{$\bE[\nu] = 10^{-4}$} & \multicolumn{2}{c|}{$\bE[\nu] = 10^{-5}$} \\
			\hline
			$h$ & $\|\hspace{-1mm}<\hspace{-1mm} \be\hspace{-1mm} >\hspace{-1mm} \|{2,1}$ & rate & $\|\hspace{-1mm}<\hspace{-1mm} \be\hspace{-1mm} >\hspace{-1mm} \|{2,1}$ & rate & $\|\hspace{-1mm}<\hspace{-1mm} \be\hspace{-1mm} >\hspace{-1mm} \|_{2,1}$ & rate \\
			\hline
			$1/2$  & 3.9419e-4 & --  & 3.9419e-4 & --   & 3.9419e-4 & -- \\
			\hline
			$1/4$  & 1.0076e-4 & 1.97 & 1.0076e-4 & 1.97 & 1.0076e-4 & 1.97 \\
			\hline
			$1/8$  & 2.5326e-5 & 1.99 & 2.5326e-5 & 1.99 & 2.5326e-5 & 1.99 \\
			\hline
			$1/16$  & 6.3402e-6 & 2.00 & 6.3399e-6 & 2.00 & 6.3398e-6 & 2.00 \\
			\hline
			$1/32$  & 1.5887e-6 & 2.00 & 1.5869e-6 & 2.00 & 1.5863e-6 & 2.00 \\
			\hline
		\end{tabular}
		\caption{Spatial errors and convergence rates of the BDF-2-EEV-SPP scheme for $\bhu$ with $\gamma = $ 5e+5, 7.5e+5, and 1e+6 respectively.}\label{sp-table1}
	\end{table}

	\subsubsection{Temporal Convergence} To assess temporal convergence, we fix the mesh size at $h = 1/64$ and set the simulation end time to $T = 1$. The simulations are run with various time-step sizes $\Delta t$, beginning with $T/2$. The time-step size is then successively reduced by a factor of 1/2. The errors are recorded, the convergence rates are computed, and the results are presented in Table \ref{temp-table1}. The results demonstrate a second-order temporal convergence rate for the BDF-2-EEV-SPP scheme, consistent with the theoretical prediction.\vspace{-3mm}  
	
	\begin{table}[h!]
		\centering
		\begin{tabular}{|c|c|c|c|c|c|c|}
			\hline
			& \multicolumn{6}{c|}{Temporal convergence (fixed $T = 1$, $h = 1/64$)} \\
			\hline
			$\epsilon = 10^{-3}$
			& \multicolumn{2}{c|}{$\bE[\nu] = 10^{-3}$} & \multicolumn{2}{c|}{$\bE[\nu] = 10^{-4}$} & \multicolumn{2}{c|}{$\bE[\nu] = 10^{-5}$} \\
			\hline
			$\Delta t$ & $\|\hspace{-1mm}<\hspace{-1mm} \be\hspace{-1mm} >\hspace{-1mm} \|_{2,1}$ & rate & $\|\hspace{-1mm}<\hspace{-1mm} \be\hspace{-1mm} > \hspace{-1mm}\|_{2,1}$ & rate & $\|\hspace{-1mm}<\hspace{-1mm} \be\hspace{-1mm} >\hspace{-1mm} \|_{2,1}$ & rate \\
			\hline
			$T/2$   & 2.9730e-1 & --   & 2.9730e-1 & --   & 2.9730e-1 & --   \\ \hline
			$T/4$   & 9.3898e-2 & 1.66 & 9.5448e-2 & 1.64 & 9.5942e-2 & 1.63 \\ \hline
			$T/8$   & 2.5574e-2 & 1.88 & 2.6189e-2 & 1.87 & 2.6399e-2 & 1.86 \\ \hline
			$T/16$  & 6.4461e-3 & 1.99 & 6.6507e-3 & 1.98 & 6.7280e-3 & 1.97 \\ \hline
			$T/32$  & 1.4173e-3 & 2.19 & 1.4915e-3 & 2.16 & 1.5208e-3 & 2.15 \\ \hline
		\end{tabular}
		\caption{Temporal errors and convergence rates of the BDF-2-EEV-SPP scheme for $\bhu$ with $\gamma =$ 1e+5.}
		\label{temp-table1}
	\end{table}
	
	\subsection{Channel Flow over a Unit Square Step}
	This test investigates how uncertainties arising from geometric modifications (e.g., wall bumps), boundary conditions, or turbulence-model parameters propagate through the flow field and affect predictions of the underlying physics and quantities of engineering interest. In this context, aleatory uncertainties represent the inherent variability in inputs such as inlet turbulence conditions, whereas epistemic uncertainties arise from incomplete knowledge of the physical processes represented by turbulence closure models. This benchmark problem serves as a conceptual UQ study because channel flow is a canonical configuration that is widely used to analyze and design a broad range of engineering systems.
	
	We consider a two-dimensional benchmark problem consisting of a $30\times 10$ rectangular channel with a $1\times 1$ step located along the bottom boundary, five units away from the inlet. The following perturbed parabolic flow profile is prescribed as the initial condition and at the inflow and outflow boundaries:\begin{align*}
		\bu_{j,h}(0,\bx)=\bu_{j,h}(t,\bx)|_{\text{inlet},\; \text{outlet}}=\lp 1+k_j\epsilon\rp{{\frac{x_2(10-x_2)}{25}}\choose{0}},
	\end{align*} where $k_j$ is a sample of an independent and identically distributed (i.i.d.) uniform random variable $k\sim U(-1,1)$, forcing $\bif_j=\textbf{0}$  for $j=1,2,\cdots,5$, and $\epsilon=10^{-3}$. Since second-order time-stepping schemes require two initial conditions to start, we use an equivalent first-order-accurate time-stepping scheme \cite{berry2025efficient} for the first time step of the BDF-2-EEV-Coupled scheme. Similarly, to initialize the BDF-2-EEV-SPP scheme, we solve an equivalent first-order-accurate penalty projection EEV scheme \cite{raveendran2026efficient} at the first time step and use the resulting solution as the second initial condition.
	
	In the BDF-2-EEV-Coupled scheme, a no-slip boundary condition is imposed on the domain walls and the step. In contrast, for the BDF-2-EEV-SPP scheme, we impose the no-slip boundary condition on the domain walls and the step only in Step 1, and we weakly enforce the vanishing of the normal velocity component on the boundary in Step 2. To ensure a fair comparison, we keep the following quantities invariant for both schemes. The channel geometry and unstructured quadrilateral meshes are generated using the free software Gmsh \cite{geuzaine2009gmsh}, resulting in a total of 100,835 DoFs for each realization at every time step. We run simulations until $T=40$ using a uniformly distributed random viscosity $\nu$ with $\bE[\nu]=10^{-4}$, $\gamma =$ 1e+4 and $\Delta t=0.1$. The speed contour plots at $t=40$ are visualized in ParaView \cite{ayachit2015paraview} and presented in Fig. \ref{channel-comparison}. The speed contours of the ensemble average in Fig. \ref{channel-comparison}(a) and Fig. \ref{channel-comparison}(b) correspond to the BDF-2-EEV-Coupled and BDF-2-EEV-SPP schemes, respectively, and exhibit similar patterns. 
	
	Define $$\text{Energy }(t_n):=\frac12\int_{\cD}\|\nabla<\bu_h>^n\|^2d\cD.$$  In Fig. \ref{channel-comparison}(c), the Energy vs. Time graph is presented for both the BDF-2-EEV-Coupled and BDF-2-EEV-SPP schemes, showing excellent agreement between the two schemes.
	\begin{figure} [ht]
		\centering	
		\subfloat[]{\includegraphics[width=0.46\textwidth,height=0.15\textwidth]{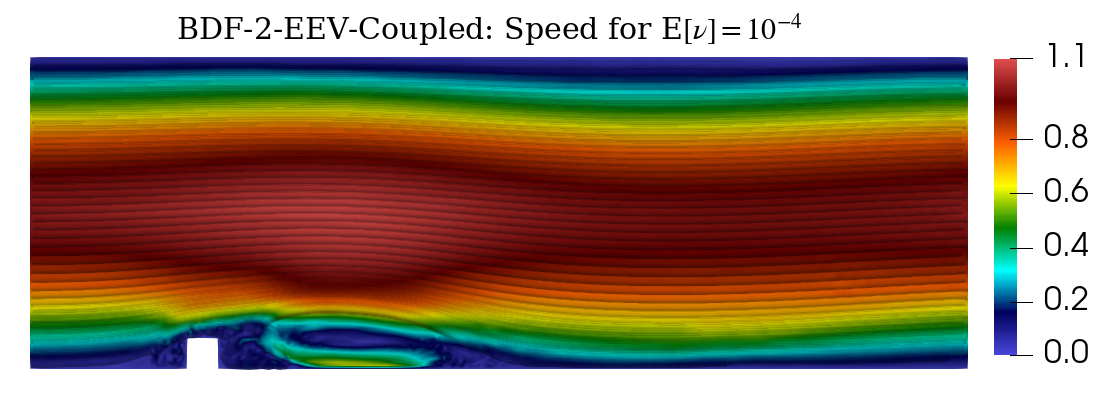}}\hspace{10mm}	
		\subfloat[]{\includegraphics[width=0.46\textwidth,height=0.15\textwidth]{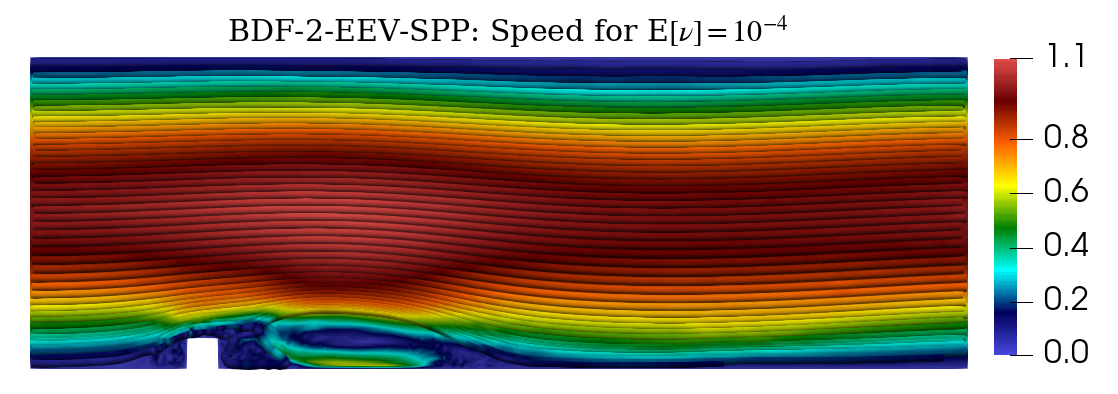}}\vspace{-3ex}
		\subfloat[]{\includegraphics[width=0.45\textwidth]{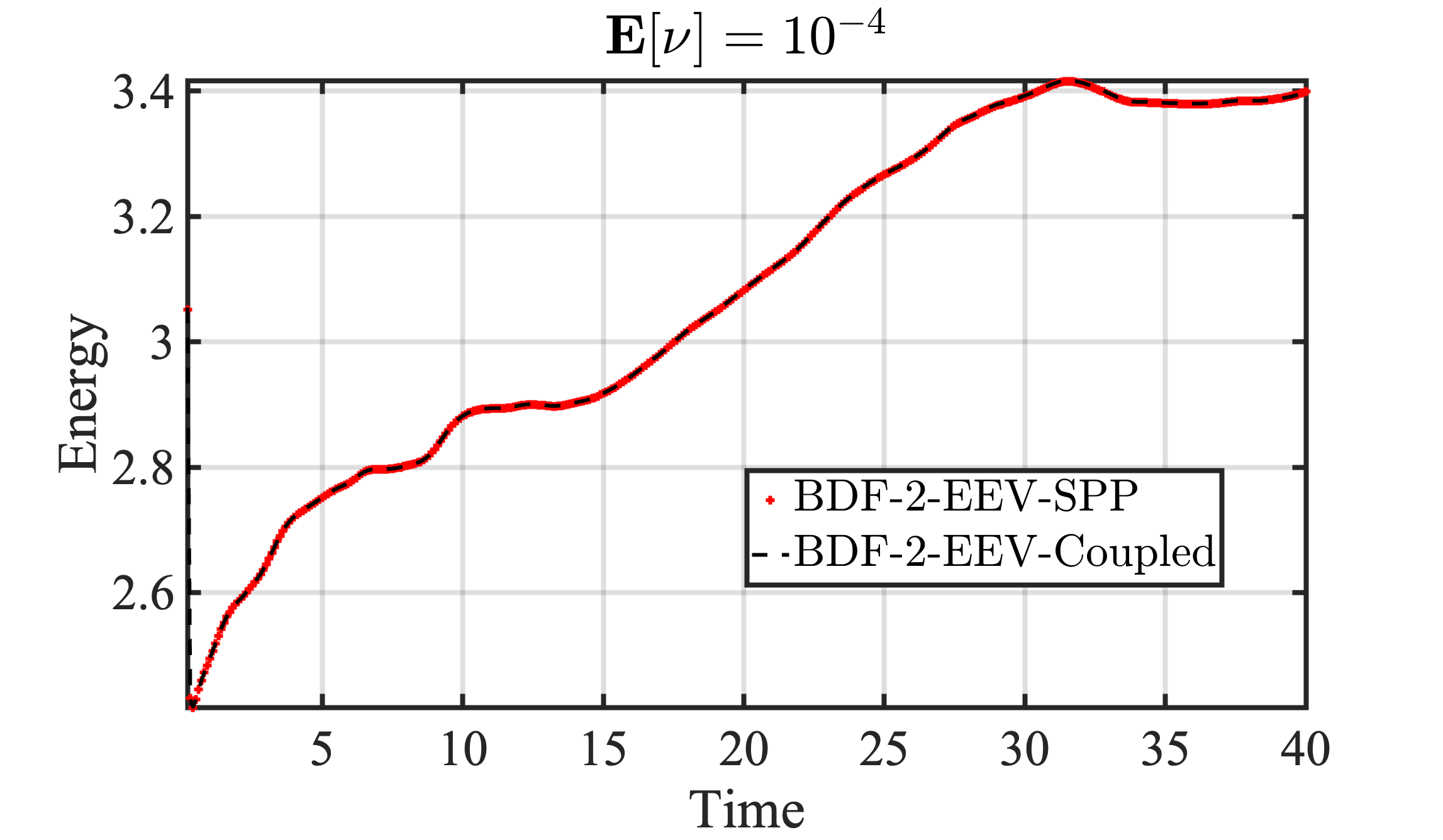}}\vspace{-4ex}\caption{\footnotesize{Flow over a step problem with $\mathbb{E}[\nu]=10^{-4}$: Speed contours of the ensemble-average solution at $t=40$ for (a) the BDF-2-EEV-Coupled scheme and (b) the BDF-2-EEV-SPP scheme. The Energy vs. Time graph in (c) demonstrates excellent agreement between the BDF-2-EEV-Coupled and BDF-2-EEV-SPP schemes.}}\vspace{-0ex}\label{channel-comparison}
	\end{figure}

	\subsubsection{Efficiency in Computational Time}
	To demonstrate the efficiency of the BDF-2-EEV-SPP scheme, we compare the computational time required by the BDF-2-EEV-SPP and BDF-2-EEV-Coupled schemes.  In this experiment, we reduce the Step 2 of Algorithm \ref{NSE-FEM} from a $2 \times 2$ block system to a $1 \times 1$ block system as described in \eqref{poisson}. We consider $T=2$ and $\Delta t=1$. At the first time step, we use a first-order EEV time-stepping efficient backward Euler (BE) coupled scheme and a BE penalty-projection scheme to initialize the BDF-2-EEV-Coupled and BDF-2-EEV-SPP schemes, respectively. We increase the number of DoFs keeping all other parameters and hardware the same for both the schemes. We ran the simulations on a MacBook M3 Max laptop with 64 GB RAM and recorded the wall-clock time in Table \ref{tab:wallclocktime}. The results show that the proposed projection method outperforms the coupled scheme in computational time.

	\begin{table}[htbp]
		\centering
		\begin{tabular}{lccc}
			\hline
			DoFs & 100,835 & 401,251 & 1,600,835  \\
			\hline
			BDF-2-EEV-SPP     & 23.86 s & 103.37 s & 411.45 s\\
			BDF-2-EEV-Coupled & 30.13 s  & 145.30 s & 763.62 s \\
			\hline
		\end{tabular}
		\caption{Wall-clock time (s):  Comparison of the computational times required by the BDF-2-EEV-SPP and BDF-2-EEV-Coupled schemes as the number of DoFs increases.}	\label{tab:wallclocktime}
	\end{table}
	The channel flow over a unit square step incorporates uncertainties associated with adverse pressure gradients. This configuration resembles 3D flow around buildings in urban environments and provides valuable insights into flow models that can inform urban planning and infrastructure design.
	
	\subsection{Channel Flow past a Circular Cylinder} In this experiment, the two-dimensional domain is defined as the intersection of a $2.2\times 0.41$ rectangular channel and the exterior of a circular cylinder with radius $0.05$ centered at $(0.2,0.2)$; see Fig. \ref{cylinder-dom}(a) \cite{john2004reference}. The inflow and outflow profiles are $$\bu_{j,h}(t,0,x_2)=\bu_{j,h}(t,2.2,x_2)=\lp 1+k_j\epsilon\rp\frac{6}{0.41^2}\sin\left(\frac{\pi t}{8}\right)\begin{pmatrix}
		x_2(0.41-x_2)\\0
	\end{pmatrix},\;\; 0\le x_2\le 0.41.$$ No-slip conditions are prescribed on the remaining boundaries for the BDF-2-EEV-Coupled scheme and in Step 1 of the BDF-2-EEV-SPP scheme. In Step 2 of the BDF-2-EEV-SPP scheme, we weakly impose a zero normal velocity component on the boundary. The flow starts from rest with no external forces applied. A level-0 quadrilateral computational mesh generated using Gmsh is presented in Fig. \ref{cylinder-dom}(b). We consider a problem with a total of 132,570 DoFs, using a time-step size $\Delta t = 0.01$ and a stabilization parameter $\gamma =$ 1e+4.
	
	\begin{figure} [ht]
		\centering	
		\subfloat[]{\includegraphics[width=0.4\textwidth,height=0.135\textwidth]{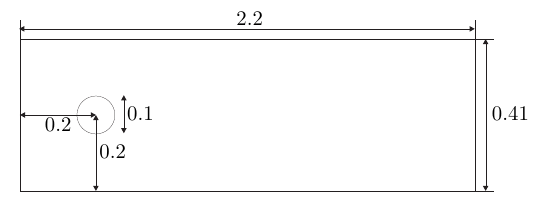}}
		\subfloat[]{\includegraphics[width=0.5\textwidth,height=0.115\textwidth]{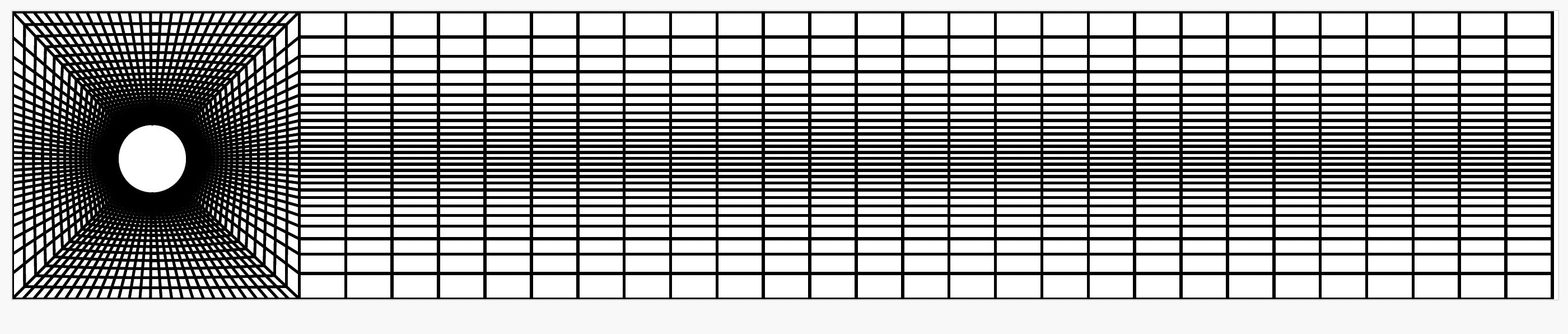}}\vspace{-4ex}
		\caption{\footnotesize{(a) The domain $\cD$ for channel flow around a cylinder and (b) the coarsest level-0 grid.}}\vspace{-0ex}\label{cylinder-dom}
	\end{figure} We performed simulations over the time interval $[0,8]$ using a small perturbation parameter $\epsilon=10^{-3}$ and $J=5$. The simulations use i.i.d. samples $k_j\sim U(-1,1)$, together with $\bE[\nu]=10^{-3},10^{-4},$ and $10^{-5}$. Moreover, Fig. \ref{cyl-validation} shows streamlines superimposed on the speed contours at $t=8$ s for (a) the BDF-2-EEV-Coupled scheme and (b) the BDF-2-EEV-SPP scheme with $\bE[\nu]=10^{-3}$. The two schemes exhibit excellent agreement with each other, and the resulting flow structure is qualitatively consistent with the benchmark result shown in Figure 2 of \cite{john2004reference}. In Fig. \ref{cyl-validation}(c), the Energy vs. Time graph shows that the energy curve from the BDF-2-EEV-SPP scheme is superimposed on the energy curve from the BDF-2-EEV-Coupled scheme.
	\begin{figure} [ht]
		\centering	
		\subfloat[]{\includegraphics[width=0.46\textwidth,height=0.12\textwidth]{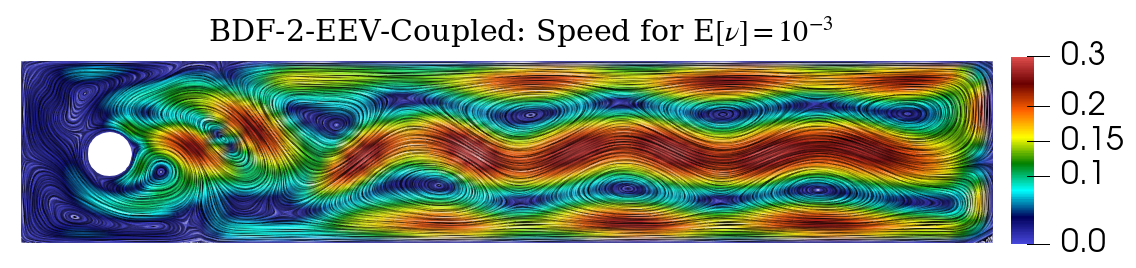}}\hspace{10mm}	
		\subfloat[]{\includegraphics[width=0.46\textwidth,height=0.12\textwidth]{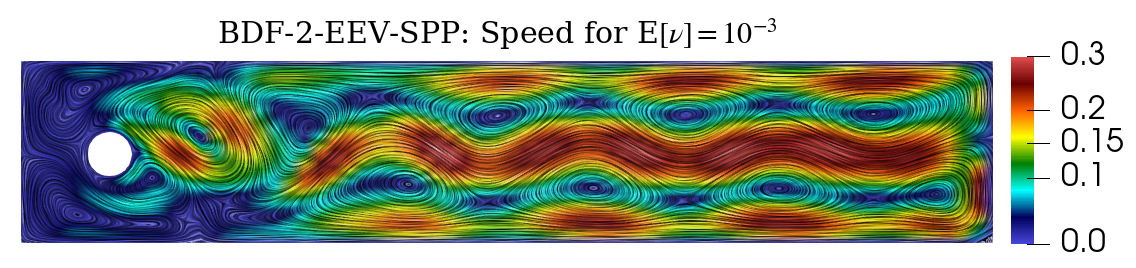}}\vspace{-3ex}
		\subfloat[]{\includegraphics[width=0.35\textwidth]{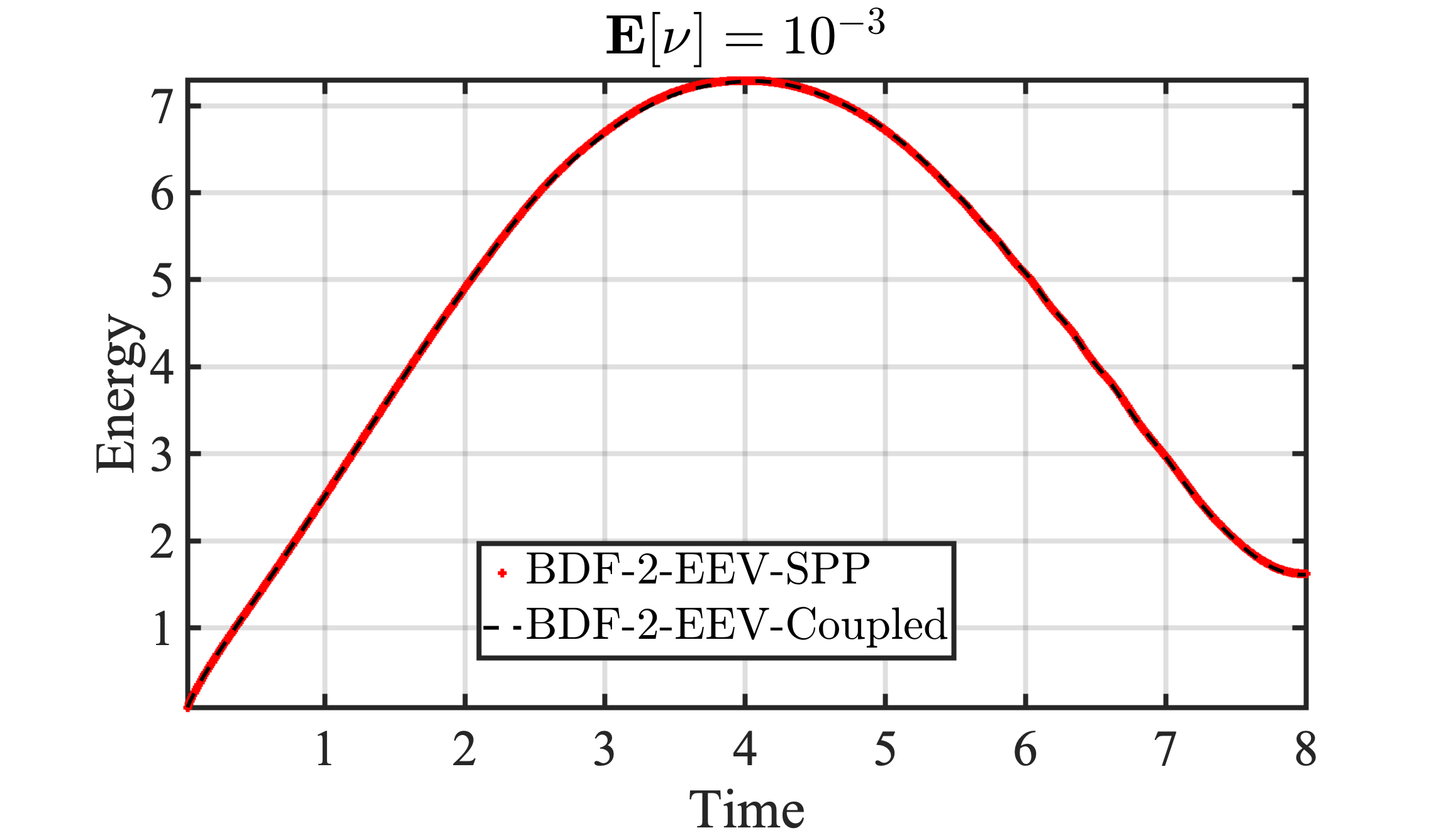}}
		\vspace{-4ex}\caption{\footnotesize{Flow past a circular cylinder: Streamlines over the speed contours for $\bE[\nu]=10^{-3}$ with (a) the BDF-2-EEV-Coupled scheme and (b) the BDF-2-EEV-SPP scheme at $t=8$ s. The Energy-versus-Time plot in (c) exhibits agreement between the BDF-2-EEV-Coupled and BDF-2-EEV-SPP schemes over $[0,8]$.}}\vspace{-0ex}\label{cyl-validation}
	\end{figure}

	Fig. \ref{cyl-speed} shows the speed contours at $t=5.5$ s, with the left column corresponding to the BDF-2-EEV-Coupled scheme and the right column corresponding to the BDF-2-EEV-SPP scheme. For $\bE[\nu]=10^{-3}$, a K\'arm\'an vortex street is clearly observed at the considered resolution. As $\bE[\nu]$ decreases, the large-scale vortical structures break down, and the initially laminar flow transitions to a chaotic and turbulent regime. Due to the limited spatial resolution, the small-scale flow structures are not fully resolved. Nevertheless, both schemes remain stable for small values of $\bE[\nu]$ and produce nearly identical results.
	
	\begin{figure} [ht]
		\centering	
		{\includegraphics[width=0.49\textwidth,height=0.12\textwidth]{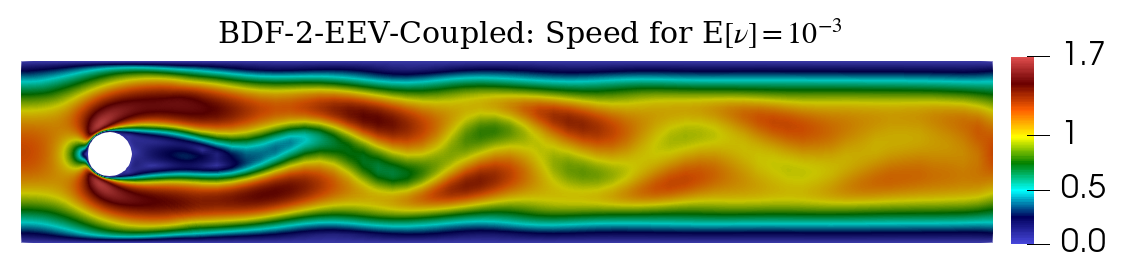}}
		{\includegraphics[width=0.49\textwidth,height=0.12\textwidth]{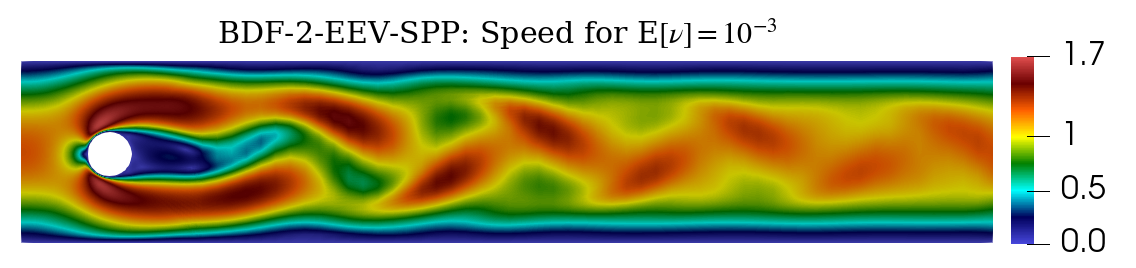}}
		{\includegraphics[width=0.49\textwidth,height=0.12\textwidth]{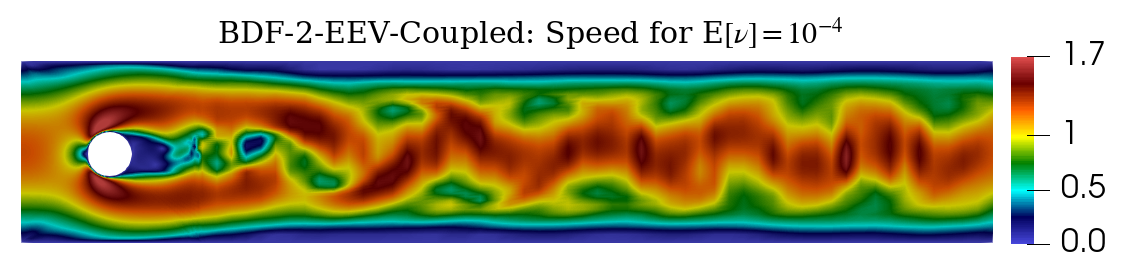}}
		{\includegraphics[width=0.49\textwidth,height=0.12\textwidth]{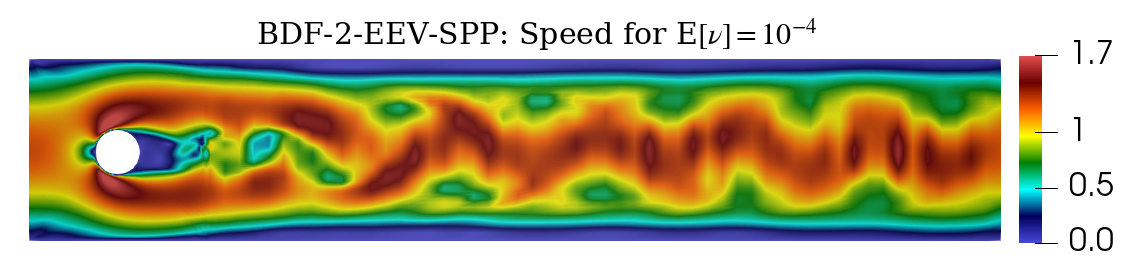}}
		\hspace{-2mm}{\includegraphics[width=0.49\textwidth,height=0.12\textwidth]{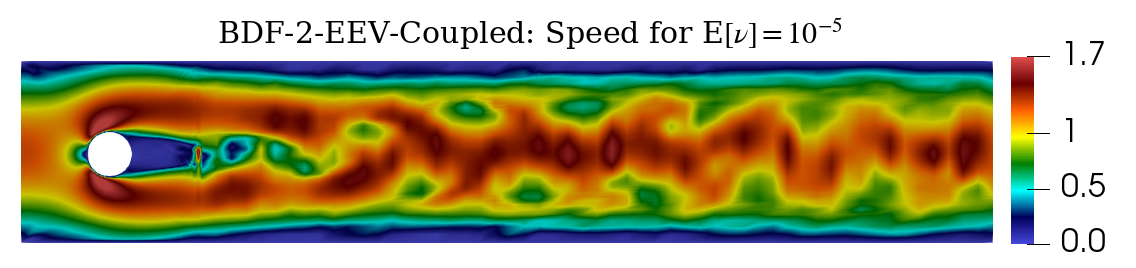}}{\includegraphics[width=0.49\textwidth,height=0.12\textwidth]{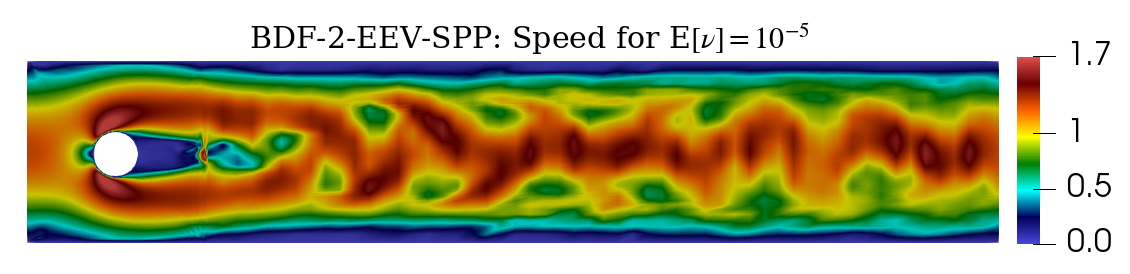}}\vspace{-0ex}
		\caption{\footnotesize{Flow past a circular cylinder: Speed contours at $t=5.5$ s as $\bE[\nu]$ varies for BDF-2-EEV-Coupled (left) and BDF-2-EEV-SPP (right).}}\vspace{-0ex}\label{cyl-speed}
	\end{figure}
	
	 The channel flow past a circular cylinder is a canonical benchmark problem for studying vortex shedding phenomena as the Reynolds number increases.
	
	\subsection{RLDC} The RLDC problem has a domain $\cD=[-1,1]^3$. The flow is assumed to start moving from rest. No-slip boundary conditions are applied to all surfaces except the lid, where we set the velocity as
	\begin{align*}
		\bu_{j,h}|_{lid}=\lp 1+k_j\epsilon\rp<
			(1-x_1^2)^2(1-x_2^2)^2,0,0>^T.
	\end{align*}
	
	\begin{figure}[h!] 
		\begin{center}    
			\subfloat[]
			{\includegraphics[width = 0.35\textwidth, height=0.28\textwidth]{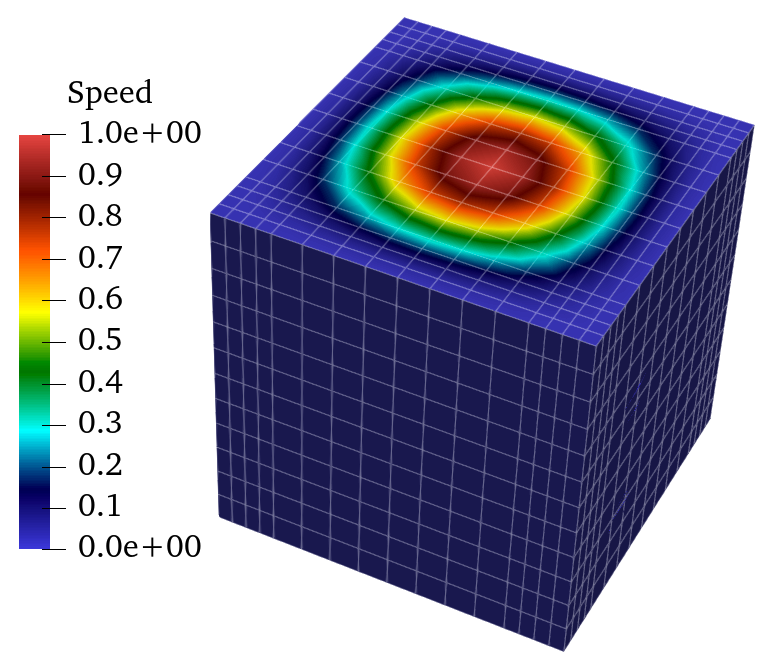}}\hspace{10mm} 
			\subfloat[]{
				\includegraphics[width = 0.3\textwidth, height=0.28\textwidth]{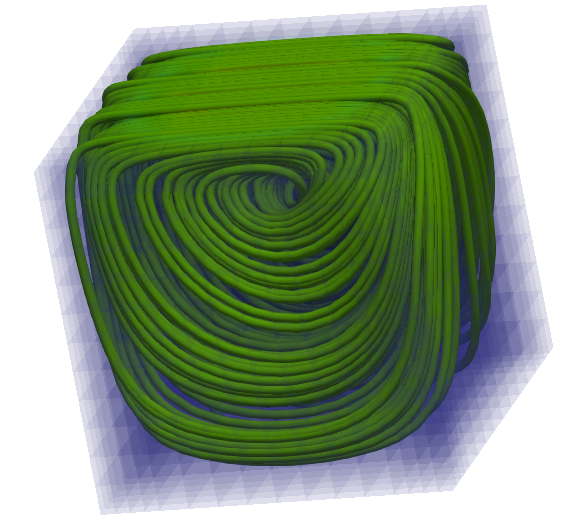}}
		\end{center}\vspace{-3mm}
		\caption{3D RLDC problem with $\bE[\nu]= 10^{-2}$  at $t=10$ with BDF-2-EEV-SPP: (a) Ensemble average of the speed contour together with a coarse mesh, and (b) Streamlines over the pressure contour.}\label{RLDC_3Dn}
	\end{figure}
	
	\begin{figure}[ht]
		\centering
		\includegraphics[width=0.35\textwidth]{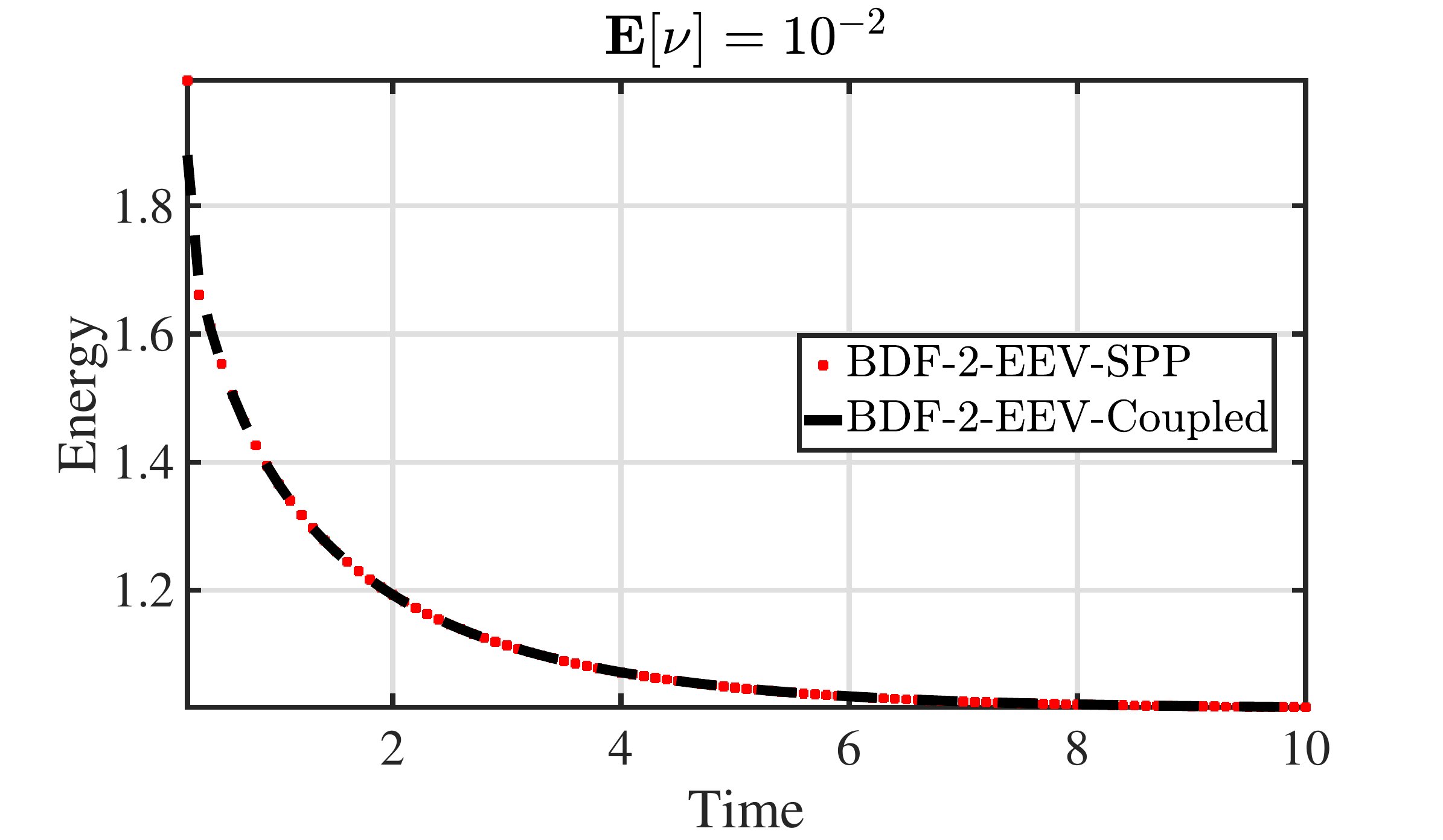}
		\caption{Energy vs. Time plot of the 3D RLDC problem, showing close agreement between the BDF-2-EEV-SPP and BDF-2-EEV-Coupled schemes for $\gamma=100$.}
		\label{energy_vs_time_ldcn}
	\end{figure}
	
	We assume that the viscosity $\nu$ is a continuous uniformly distributed i.i.d. random variable. A random sample of size $J=5$ is generated with a $10\%$ variation about the mean, yielding an expected viscosity $\bE[\nu]=10^{-2}$. The three-dimensional geometry and the corresponding structured hexahedral mesh are generated using Gmsh, resulting in 85,432 DoFs per time step for each realization.
	
	The simulations are performed with the parameters $\Delta t=0.1$, $\mu=1$, the $(\mathbb{Q}_2^3,\mathbb{Q}_1)$ finite element pair, and $\epsilon=10^{-3}$. For the stabilization parameter, we use $\gamma=100$ for both the BDF-2-EEV-SPP scheme and the BDF-2-EEV-Coupled scheme.
	
	The numerical experiments are carried out in deal.II using a parallel MPI implementation of the BDF-2-EEV schemes on 5 cores. The solution of the BDF-2-EEV-SPP scheme at $t=10$ is visualized using ParaView. Fig.~\ref{RLDC_3Dn}(a) shows speed contours of the ensemble average solution obtained with the BDF-2-EEV-SPP scheme over the computational mesh, while Fig.~\ref{RLDC_3Dn}(b) presents the corresponding streamlines over the pressure contour. Fig.~\ref{energy_vs_time_ldcn} shows the Energy vs. Time plot over $[0, 10]$, demonstrating close agreement between the BDF-2-EEV-SPP and BDF-2-EEV-Coupled schemes.
	
	\subsubsection{Effect of EEV on Convection-Dominated Problems} The ensemble eddy viscosity (EEV) stabilization plays an important role in improving the robustness of numerical algorithms for highly convection-dominated and computationally challenging flow problems, such as the regularized 3D RLDC problem at high Reynolds numbers. To investigate the effect of the EEV term, we consider the 3D RLDC problem using the same model parameters described previously, except that we set $\bE[\nu]=10^{-6}$, $T=100$, $\Delta t=1$, and vary $\mu$. The energy evolution is examined in Fig. \ref{mu_change} for several values of the EEV coefficient, $\mu = 0,0.1,0.25,0.5, \text{and } 1$, where $\mu = 0$ corresponds to the scheme without EEV stabilization.
	The Energy vs. Time plots demonstrate a clear stabilizing effect of the EEV term. In the absence of EEV stabilization $(\mu=0)$, the numerical solution becomes unstable and eventually blows up around $t = 35$, whereas the simulations with $\mu>0$ remain stable over a substantially longer time interval.
	
	\begin{figure}[ht]
		\centering
		\includegraphics[width=0.35\textwidth]{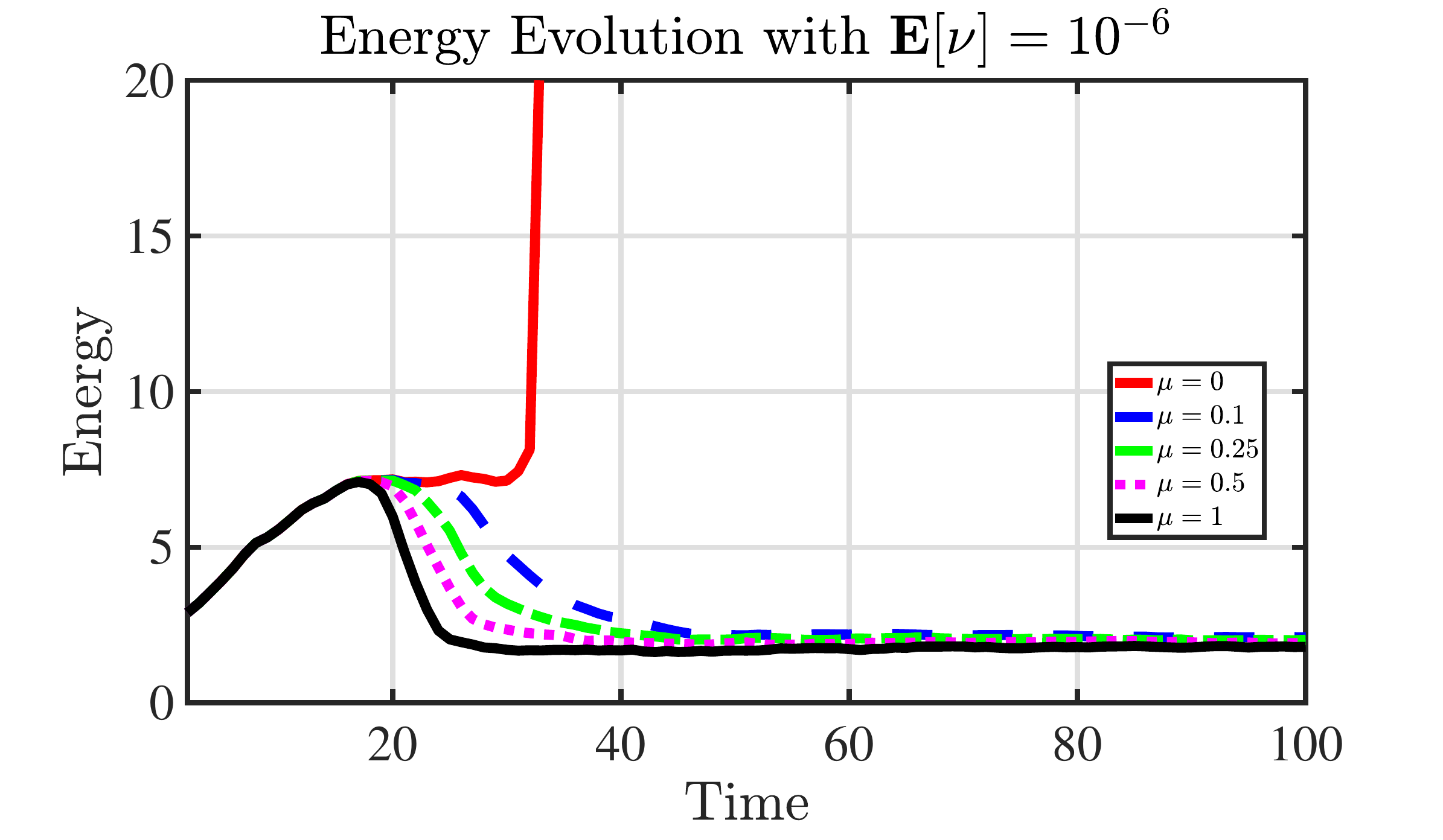}
		\caption{Effect of ensemble eddy viscosity in the 3D RLDC problem with $\bE[\nu]=10^{-6}$: Energy vs. Time for varying values of $\mu$.}
		\label{mu_change}
	\end{figure}
	These numerical observations demonstrate that EEV stabilization is particularly beneficial for convection-dominated stochastic flow simulations, where insufficient numerical dissipation can lead to instability. When combined with grad-div stabilization, the EEV term provides additional control of the unresolved velocity fluctuations while maintaining stable long-time simulations. Thus, the EEV-based ensemble algorithms provide improved robustness compared with their counterparts without EEV stabilization, particularly for challenging high-Reynolds-number flow problems.

	The convergence-rate verification and the experiments on three benchmark problems support the theoretical findings and robustness of EEV schemes in this paper. The present investigation compares two numerical schemes within a UQ framework. In doing so, it contributes to the development of numerical UQ methodologies while demonstrating their applicability to three important classes of real-world flow problems: bluff-body flows with adverse pressure gradients, vortex-dominated wake flows, and 3D lid-driven cavity problem.

	\section{Conclusion} In this study, we developed, analyzed, and tested a fully discrete, high-order accurate BDF-2-EEV-SPP time-stepping splitting scheme for parameterized ensemble flow problems. The scheme is a fast, robust, and accurate penalty-projection method. The elegant linearized scheme is designed in such a way that, at each time step, the system matrix of each subproblem remains the same across all realizations. This shared-matrix structure significantly reduces memory usage and computational time. The use of a large penalty parameter $\gamma$ minimizes the splitting error and ensures the optimal accuracy of the projection method. The algorithm is EEV-regularized so that it can handle convection-dominated problems.
	
	We rigorously proved the stability of the scheme. We also proved that the proposed scheme converges to an equivalent coupled scheme as $\gamma$ increases, thereby ensuring second-order temporal and optimal spatial accuracy. The numerical experiments verify the following: (1) the asymptotic linear convergence of the proposed scheme to the coupled scheme as $\gamma\to\infty$, (2) the second-order temporal convergence of the proposed scheme, and (3) optimal spatial convergence with the $(\mathbb{Q}_2^2,\mathbb{Q}_1)$ element pair. The outcomes of these experiments support the theoretical findings of this paper.
	
	To evaluate the performance of the proposed scheme, we implemented both the proposed scheme and the equivalent coupled scheme on benchmark two- and three-dimensional problems: (1) a $30\times 10$ channel flow over a unit square step, (2) a channel flow past a circular cylinder, and (3) 3D RLDC problems. We find that the proposed scheme performs well and shows excellent agreement with the coupled scheme for all three problems.

	For more complex and realistic problems, the proposed scheme is expected to perform better than the coupled scheme. This is because each subproblem in the proposed scheme is easier to solve than the corresponding coupled saddle-point problem, and the combined computational cost of solving all subproblems is lower than that of the coupled scheme. Although the scheme is stable without an explicit time-step restriction under the stated lower-bound condition on the EEV parameter $\mu$, the restrictive condition on the viscosity data needs to be examined further.
	
	For future work, we plan to explore Yosida-type algebraic splitting \cite{AMRX17,REBHOLZ2020112366} for block saddle-point problems arising in efficient UQ computations. We also plan to extend the proposed approach to parameterized NSE problems using Physics-Informed Neural Networks (PINNs), following the work in \cite{aziz2025self}.
	
	\section{Acknowledgments}This research was supported by the National Science Foundation under Grant DMS-2425308. We also gratefully acknowledge the Alabama Supercomputer Authority (ASA) for providing computational resources.
	
	\bibliographystyle{plain}
	\bibliography{BE}
	\appendix
	\section{Proof of Lemma \ref{lemma-L3-infty}}\label{appendix-C}
	\begin{proof}
		Assume that $\bu_j(t_n)\in H^{k+1}(\mathcal D)^d$ with $k\ge 2$, and let
		$I_h\bu_j(t_n)$ be a suitable finite element interpolant. Define $\boldsymbol{\phi}_h^n
		:=
		\bu_{j,h}^n-I_h(\bu_j(t_n))$, and write
		\[
		\bu_{j,h}^n
		=
		\left(\bu_{j,h}^n-I_h(\bu_j(t_n))\right)
		+
		I_h(\bu_j(t_n))=\boldsymbol{\phi}_h^n+I_h(\bu_j(t_n)).
		\]
		Then,
		\begin{align}
			\|\nabla \bu_{j,h}^n\|_{L^3}
			+\|\bu_{j,h}^n\|_{L^\infty}
			&\le
			\|\nabla \boldsymbol{\phi}_h^n\|_{L^3}
			+\|\boldsymbol{\phi}_h^n\|_{L^\infty}
			+\|\nabla I_h(\bu_j(t_n))\|_{L^3}
			+\|I_h(\bu_j(t_n))\|_{L^\infty}.
			\label{eq:interp_split}
		\end{align}
		Using the Sobolev embedding and the discrete inverse
		inequality gives
		\[
		\|\nabla \boldsymbol{\phi}_h^n\|_{L^3}
		\le
		C h^{-1/2}\|\nabla \boldsymbol{\phi}_h^n\|.
		\]
		Moreover, using the discrete inverse inequality together with the Sobolev embedding,
		\[
		\|\boldsymbol{\phi}_h^n\|_{L^\infty}
		\le
		C h^{-1/2}\|\boldsymbol{\phi}_h^n\|_{L^6}
		\le
		C h^{-1/2}\|\nabla \boldsymbol{\phi}_h^n\|.
		\]
		Thus,
		\begin{equation}
			\|\nabla \bu_{j,h}^n\|_{L^3}
			+\|\bu_{j,h}^n\|_{L^\infty}
			\le
			C h^{-1/2}\|\nabla \boldsymbol{\phi}_h^n\|
			+
			\|\nabla I_h(\bu_j(t_n))\|_{L^3}
			+
			\|I_h(\bu_j(t_n))\|_{L^\infty}.
			\label{eq:inverse_phi}
		\end{equation}
		By the regularity of $\bu_j$ and the stability of the interpolation operator,
		\[
		\|\nabla I_h(\bu_j(t_n))\|_{L^3}
		+
		\|I_h(\bu_j(t_n))\|_{L^\infty}
		\le C.
		\]
		Hence,
		\begin{equation}
			\|\nabla \bu_{j,h}^n\|_{L^3}
			+\|\bu_{j,h}^n\|_{L^\infty}
			\le
			C h^{-1/2}\|\nabla \boldsymbol{\phi}_h^n\|+C.
			\label{eq:inverse_bound}
		\end{equation}
		Using the triangle inequality, the interpolation error bound, and the velocity error estimate in \eqref{error-eqn-coupled}, we obtain
		\begin{align}
			\|\nabla \boldsymbol{\phi}_h^n\|
			=
			\|\nabla(\bu_{j,h}^n-I_h(\bu_j(t_n)))\|&\le
			\|\nabla(\bu_{j,h}^n-\bu_j(t_n))\|
			+
			\|\nabla(\bu_j(t_n)-I_h(\bu_j(t_n)))\|
			\nonumber\\
			&\le
			\|\nabla(\bu_{j,h}^n-\bu_j(t_n))\|
			+
			C h^k\nonumber\\&\le C\left(\frac{h^k}{\Delta t^{\frac12}}+\Delta t^{\frac32}\right)+Ch^k.
			\label{eq:phi_bound}
		\end{align}
		Therefore		\begin{align}
			\|\nabla \bu_{j,h}^n\|_{L^3}
			+\|\bu_{j,h}^n\|_{L^\infty}
			&\le
			C h^{-1/2}
			\left(
			h^k\Delta t^{-1/2}
			+
			\Delta t^{3/2}
			+
			h^k
			\right)+C
			\nonumber\\
			&\le
			C\left(
			h^{k-\frac12}\Delta t^{-1/2}
			+
			h^{-\frac12}\Delta t^{3/2}
			+
			h^{k-\frac12}
			\right)+C.
			\label{eq:final_prebound}
		\end{align}
		For $k\ge 2$, assuming $O(h^{2k-1})
		\le
		\Delta t
		\le O(h^{1/3})$ gives $$\max_{1\le n\le M}
		\left(
		\|\nabla \bu_{j,h}^n\|_{L^3}
		+
		\|\bu_{j,h}^n\|_{L^\infty}
		\right)
		\le 3+C.$$
		Setting $C_*:=3+C$ and	using the discrete \textit{inf-sup} condition yields the pressure bound, which completes the proof.
	\end{proof}

	\section{Proof of Lemma \ref{uniform-boundedness-lemma-proof}}  \label{appendix}
	\begin{proof}
		Base step: $\bhu_{j,h}^0=I_h(\bu_j(0,\bx))$ and $\bhu_{j,h}^1=I_h(\bu_j(\Delta t,\bx)),$ where $I_h$ is an appropriate interpolation operator. The regularity assumption on the true solution $\bu_j$ gives $\|\bhu_{j,h}^0\|_{L^\infty}\le C_*$ and $\|\bhu_{j,h}^1\|_{L^\infty}\le C_*$ for some constant $C_*>0$.\\
		Inductive step: Assume that, for some $L\in\mathbb{N}$ and $L<M$, $\|\bhu_{j,h}^n\|_{L^\infty}\le C_*$ holds for $n=0,1,\cdots,L$. Then, using the triangle inequality and Lemma \ref{lemma-L3-infty}, we obtain
		\begin{align*}
			\|\bhu_{j,h}^{L+1}\|_{L^\infty}\le\|\bhu_{j,h}^{L+1}-\bu_{j,h}^{L+1}\|_{L^\infty}+K_*.
		\end{align*}
		Using the discrete inverse inequality in three dimensions yields
		\begin{align}
			\|\bhu_{j,h}^{L+1}\|_{L^\infty}\le Ch^{-\frac32}\|\bhu_{j,h}^{L+1}-\bu_{j,h}^{L+1}\|+K_*.
		\end{align}
		Next, using the induction hypothesis and carrying out the velocity convergence theorem proof up to \eqref{after-gronwall-n}, we obtain
		\begin{align}
			&\|\bhu_{j,h}^{L+1}\|_{L^\infty}\le K_*+\frac{C}{h^{\frac32}\gamma^{\frac{1}{2}}} \exp\left\{\frac{C}{\alpha_{\min}} \left(\frac{\Delta t}{h^3\alpha_{\min}}+1\right)\right\}.
		\end{align}
		For a fixed mesh and time-step size, letting $\gamma\rightarrow \infty$ yields $\|\bhu_{j,h}^{L+1}\|_{L^\infty}\le K_*$. Setting $C_*=K_*$, we obtain the desired bound. Hence, by strong induction, $\|\bhu_{j,h}^{n}\|_{L^\infty}\le C_*$ holds true for $0\le n\le M$.
	\end{proof}
	
\end{document}